\documentclass[10pt,reqno]{amsart}
\usepackage{url}
\usepackage{amsthm}
\usepackage{fullpage}
\usepackage{amsmath}
\usepackage{amssymb}
\usepackage{amsfonts}

\usepackage{enumitem}
\usepackage{mathscinet}
\usepackage{lipsum}
\usepackage[english]{babel}
\usepackage[autostyle]{csquotes}
\usepackage{bbold}
\usepackage{hyperref}

\hypersetup{
    colorlinks=true,
    linkcolor=red,
    citecolor=blue,
    }

\usepackage{graphicx}

\usepackage[utf8]{inputenc}
\usepackage[english]{babel}

\usepackage{color}
\usepackage{comment}

\newtheorem{thm}{Theorem}[section]
\newtheorem{definition}[thm]{Definition}
\newtheorem{theorem}[thm]{Theorem}
\newtheorem*{oseledec*}{Oseledec Theorem}
\newtheorem{cor}[thm]{Corollary}
\newtheorem{claim}[thm]{Claim}
\newtheorem{lemma}[thm]{Lemma}
\newtheorem{prop}[thm]{Proposition}

\makeatletter
\def\moverlay{\mathpalette\mov@rlay}
\def\mov@rlay#1#2{\leavevmode\vtop{%
   \baselineskip\z@skip \lineskiplimit-\maxdimen
   \ialign{\hfil$\m@th#1##$\hfil\cr#2\crcr}}}
\newcommand{\charfusion}[3][\mathord]{
    #1{\ifx#1\mathop\vphantom{#2}\fi
        \mathpalette\mov@rlay{#2\cr#3}
      }
    \ifx#1\mathop\expandafter\displaylimits\fi}
\makeatother

\newcommand{\bigcupdot}{\charfusion[\mathop]{\bigcup}{\cdot}}
\let\ul\underline
\let\wh\widehat

\newcommand{\nocontentsline}[3]{}
\newcommand{\tocless}[2]{\bgroup\let\addcontentsline=\nocontentsline#1{#2}\egroup}

\usepackage{wasysym}
\newcommand{\osharp}{\charfusion[\mathbin]{\#}{\ocircle}}

\def\Jac{\ensuremath{\mathrm{Jac}}}
\def\HWT{\ensuremath{\mathrm{RWT}}}
\def\WT{\ensuremath{\mathrm{WT}}}
\def\Vol{\ensuremath{\mathrm{Vol}}}
\def\CAR{\ensuremath{\curvearrowright}}
\def\Sig{\ensuremath{\widehat{\Sigma}}}
\title{Hyperbolic SRB measures and the Leaf Condition
}
\author{Snir Ben Ovadia}
\address{Faculty of Mathematics and Computer Science, Weizmann Institute of Science, POB 26, Rehovot, 76100 ISRAEL} \email{snir.benovadia@weizmann.ac.il}

\begin{document}
\maketitle
\begin{abstract}
Let $M$ be a Riemannian, boundaryless, and compact manifold, with $\dim M\geq 2$ and let $f$ be a $C^{1+}$ diffeomorphism. We show that there is a hyperbolic SRB measure if and only if there exists an unstable leaf with a subset of positive leaf volume of hyperbolic points which return to some Pesin set with positive frequency. This answers a question of Pesin.
\end{abstract}

\tableofcontents
\section{Introduction and Overview of Main Results}
In this paper, we give necessary and sufficient conditions for the existence of hyperbolic SRB measures for non-uniformly hyperbolic diffeomorphisms on compact manifolds of arbitrary dimension. We begin by recalling what are SRB measures and why they are interesting.
\subsection{SRB Measures}
Let $M$ be a compact Riemannian manifold with no boundary, of dimension $d\geq 2$. Let $f\in \mathrm{Diff}^{1+\beta}(M)$, $\beta\in (0,1)$, the set of $C^1$ diffeomorphisms on $M$ with $\beta$-H\"older continuous differential.

An embedded submanifold $V^u$ is called an {\em unstable leaf} if $\forall x,y\in V^u$, $\limsup_{n\rightarrow\infty}\frac{1}{n}\log d(f^{-n}(x), f^{-n}(y))<0$. We say that $V^u$ is {\em of maximal dimension}, if it is not contained in any unstable leaf with larger dimension.

Given a measurable partition $\xi$, let $\xi(x)=$ the unique $\xi$-element which contains $x$. By Rokhlin's disintegration theorem, any Borel probability $\mu$ carried by $\cup\xi$, can be disintegrated into conditional measures w.r.t. $\xi$, $\mu=\int_M \mu_{\xi(x)}d\mu(x)$, where for $\mu$-a.e. $x$, $\mu_{\xi(x)}$ is a Borel probability carried by $\xi(x)$. A measurable partition is said to be {\em subordinated to the unstable foliation}, if every partition element is contained in an unstable leaf of maximal dimension $V^u(\xi(x))$, and $\lambda_{V^u(\xi(x))}(\xi(x))>0$, where $\lambda_{V^u(\xi(x))}$ is the induced Riemannian volume of $V^u(\xi(x))$.

An $f$-invariant Borel probability $\mu$ is called an SRB measure (after Sinai, Ruelle and Bowen), if given a measurable partition which is subordinated to the unstable foliation $\xi$, $\mu_{\xi(x)}\ll \lambda_{V^u(\xi(x))}$ for $\mu$-a.e. $x$.

SRB measures are interesting, because they serve as good substitutes for absolutely continuous invariant measures (a.c.i.m.), in cases when acims do not exist. This is due to the following three important properties of SRB measures:

\begin{enumerate}
	\item \textbf{Geometric Structure:} Compatibility with the induced Riemannian structure on unstable leaves even when the Riemannian volume is not preserved.
	\item \textbf{Physicality:} Every ergodic and hyperbolic SRB measure $\mu$ is physical: $$\Vol\left(\left\{x\in M:\forall g\in C(M), \lim_{n\rightarrow\infty}\frac{1}{n}\sum_{k=0}^{n-1}g\circ f^k(x)=\int gd\mu\right\}\right)>0,$$
where Vol denotes the Riemannian volume of $M$. This was first shown for uniformly hyperbolic diffeomorphisms by Ruelle \cite{Ruelle76}. For volume preserving non-uniformly hyperbolic diffeomorphisms, see \cite{Pesin77}, and for non volume preserving see \cite{KS,PughAndShub}. Physicality does not imply the SRB property. For the gap between the two properties, see \cite{Tsuji}.
 	\item \textbf{Entropic variational principle:} SRB measures satisfy Pesin's entropy formula $$h_\mu(f)=\int \sum_{i:\chi_i>0}\chi_i(x)d\mu(x)$$
where $\chi_i(x)$ is the $i$-th Lyapunov exponent of $x$ (with multiplicity),\cite{Ruelle76,Pesin77,EntropyFormulaSRB, StrelcynEntropyFormulaSRB}; and the LHS is strictly smaller than the RHS for all other measures \cite{Ruelle76,MargulisRuelleIneq,LedrappierYoungI}.
\end{enumerate}
For other nice properties of SRB measures, such as their role in the theory of random perturbation of smooth dynamical systems, see \cite{Kifer} and \cite{YoungStochasticStability}.
\subsection{The Problem} The natural question arises: which systems admit SRB measures?

In \cite{RodriguezHertz}, Rodriguez-Hertz, Rodriguez-Hertz, Tahzibi and Ures have introduced the notion of an {\em ergodic homoclinic class} associated to a periodic hyperbolic point. They have shown that each ergodic homoclinic class carries at most one (hyperbolic) SRB measure, and that every ergodic hyperbolic SRB measure is carried by a single ergodic homoclinic class.

Because of these results, 
we focus our work on constructing ergodic hyperbolic SRB measures on a given fixed ergodic homoclinic class
.
\subsection{Main Results} In this paper we give a necessary and sufficient condition for the existence of a hyperbolic SRB measure in terms of the structure of the unstable leaves of hyperbolic points (``Leaf condition"). We give an overview of our main results.

Building on our earlier works of \cite{SBO,LifeOfPi}, we describe in \textsection \ref{HWTsection} {\em explicit sets} $\HWT_\chi^\epsilon\subseteq M (\epsilon,\chi>0)$ with the following properties:
\begin{enumerate}
	\item $\HWT_\chi^\epsilon$ is $f$-invariant.
	\item $\HWT_\chi^\epsilon$ carries all hyperbolic $f$-invariant measures all of whose Lyapunov exponents have modulus bigger than $\chi$.
	\item $\HWT_\chi^\epsilon=\bigcup_{r=1}^\infty\Lambda_r$, where $\Lambda_r$ are suitably defined Pesin sets (see Definition \ref{PesinLevelSets}).
	\item $\exists \epsilon_\chi(\chi,M,f,\beta)$ s.t. $\forall \epsilon\in(0, \epsilon_\chi]$ one can code $f:\HWT_\chi^\epsilon\rightarrow \HWT_\chi^\epsilon$ by a two-sided countable Markov shift $\sigma:\Sig\rightarrow\Sig$ (see \textsection \ref{symbolicdynamicspart}).
\end{enumerate}
We now make the following definitions:
\begin{enumerate}
	\item $\HWT_\chi:=\HWT_\chi^{\epsilon_\chi}$,
	\item $\HWT_\chi^{\mathrm{PR}}:=\left\{x\in \HWT_\chi:\exists r_x\text{ s.t. }\limsup_{n\rightarrow\infty}\frac{1}{n}\sum_{k=0}^{n-1}\mathbb{1}_{\Lambda_{r_x}}(f^k(x))>0\right\}$,
	\item Given a hyperbolic periodic point $p$, the {\em ergodic homoclinic class} of $p$ (w.r.t. $\chi$) is $$H_\chi(p):=\{x\in\HWT_\chi:W^u(x)\pitchfork W^s(o(p))\neq\varnothing, W^s(x)\pitchfork W^u(o(p))\neq\varnothing\},$$
	where $W^u(x)=\bigcup_{n\geq0}f^{n}[V^u(f^{-n}(x))]$,  $W^s(x)=\bigcup_{n\geq0}f^{-n}[V^s(f^{n}(x))]$, where $V^{u/s}(\cdot)$ is the respective local unstable/stable leaf of a non-uniformly hyperbolic point, and $o(p)=\{p,\ldots,f^{\mathrm{Per}(p)-1}(p)\}$.
\end{enumerate}

Given a smooth submanifold $V\subseteq M$, let $\lambda_V$ denote the induced Riemannian volume on $V$. The main results of this paper are (Theorem \ref{theorem1}):
\begin{enumerate}
	\item[(a)] $f$ admits a $\chi$-hyperbolic SRB measure if and only if $\exists$ unstable leaf of maximal dimension $V^u$ s.t. 
\begin{equation}\label{Leaf-Condition}
\lambda_{V^u}(\HWT_\chi^{\mathrm{PR}})>0.
\end{equation}

	\item[(b)] Suppose $\exists$ unstable leaf of maximal dimension $V^u$ s.t. $\lambda_{V^u}(\HWT_\chi)>0$, then $\exists$ hyperbolic periodic point $q$ s.t. $H_\chi(q)$ carries an ergodic, conservative, invariant and $\sigma$-finite measure $\mu$ with absolutely continuous measures on unstable leaves. This is an SRB measure if and only if $\mu$ is finite.
\end{enumerate}

We remark that $V^u$ in (a),(b) can always be taken to be the unstable leaf of maximal dimension of a hyperbolic periodic point.

Our results extend the classical results of Sinai and Ruelle on the existence of SRB measures for uniformly hyperbolic diffeomorphisms to the non-uniformly hyperbolic setup. Let us explain how they apply in the particular examples of Anosov or Axiom A diffeomorphisms. For Anosov systems, the condition is satisfied trivially, since orbits are uniformly hyperbolic, and every point on an unstable leaf is hyperbolic. For Axiom A systems, orbits are uniformly hyperbolic; and when the basin of attraction has a positive volume, the volume of the hyperbolic points in an unstable leaf transversal to the saturation of stable leaves, must be positive due to Pesin's absolute continuity theorem.
Our results also extend other known results in the non-uniformly hyperbolic setup, such as \cite{Y} and \cite{PesinVaughnLuzatto} (see \textsection \ref{ComparisonTo} for more details).

 While the leaf condition is difficult to check in concrete examples (indeed we do not have new examples where we can check it), we believe it has theoretic value especially for studying the stability and genericity of the existence of SRB measures in ``natural classes" of dynamical systems. This is because the leaf condition can be checked on a single leaf of a single hyperbolic periodic point.

Our methods also give a characterization of the existence of a hyperbolic SRB  measure in the language of thermodynamic formalism of countable Markov shifts. The restriction of $f$ to an ergodic homoclinic class $H_\chi(p)$ admits a coding by a topologically transitive countable Markov shift $\sigma:\Sig\rightarrow \Sig$ (see \cite{BCS} and \textsection \ref{hertzetAl} below). Let
$$\Sig_L=\{(\ldots,R_{-2},R_{-1},R_0):\ul{R}\in \Sig\}$$	$$\sigma_R:\Sig_L\rightarrow \Sig_L,\text{ }\left(\sigma_R\ul{R}\right)_i=R_{i-1}.$$
In \textsection \ref{leavesandmeasures} we construct for each $\ul{R}\in \Sig_L$ an unstable leaf of maximal dimension $V^u(\ul{R})$ and a measure $m_{V^u(\ul{R})}$ on $V^u(\ul{R})$ s.t. $m_{V^u(\ul{R})}\sim \lambda_{V^u(\ul{R})}$ and
$$m_{V^u(\sigma_R\ul{R})}\circ f^{-1}|_{V^u(\ul{R})}=e^{\phi(\ul{R})}\cdot m_{V^u(\ul{R})}$$
where $\phi(\cdot)$ is an explicit one-sided cohomolog of the geometric potential in symbolic coordinates. We show
\begin{enumerate}
	\item[(c)] $\phi:\Sig_L\rightarrow\mathbb{R}$ is bounded, H\"older continuous, and has non-positive Gurevich pressure, and under the assumptions of (b), the measure from (b) is finite (and thus an ergodic hyperbolic SRB) 
	if and only if $\phi$ is positive recurrent with zero Gurevich pressure. (See \textsection \ref{lalipop} for Definitions, and Theorem \ref{naturalmeasures}, Claim \ref{VolHol}, Appendix A, Theorem \ref{ruellemargulis}, Theorem \ref{theorem1}).
\end{enumerate}

Finally, in proving (c), we also prove in a new way that a hyperbolic SRB measure satisfies Pesin's entropy formula, see Corollary \ref{kukilida}.

\subsection{Comparison to Other Results in the Literature}\label{ComparisonTo}

In their pioneering work, Sinai, Ruelle, and Bowen gave necessary and sufficient conditions for the existence of SRB measures in the context of Anosov and Axiom A systems (see \cite{Si1,Ruelle76,B4,BowenRuelle}). See also \cite{uGibbsPesinSinai} for existence of $u$-Gibbs measures on compact attractors with uniform expansion. In \cite{YoungCounterExample}, Hu and Young constructed an example of an ``almost Anosov" system, which admits no SRB measures. Later, in her celebrated result \cite{Y}
, Young showed the existence of an SRB measure for non-uniformly hyperbolic maps with Young towers, subject to an assumption of integrability of the return-time to 
the base of the tower w.r.t. the Lebesgue measure.

In \cite{PesinVaughnDima}, Climenhaga, Dolgopyat and Pesin have introduced the notion of a {\em leaf condition}, with ``effectively hyperbolic" points. 
 Climenhaga, Dolgopyat and Pesin show that if there exists a forward invariant set of positive Riemannian volume, which admits a measurable invariant family of stable and unstable cones, and the leaf condition is satisfied for an effectively hyperbolic subset of it, then there exists a hyperbolic SRB measure.

In the non-uniformly hyperbolic setup, Climenhaga, Luzzatto and Pesin have recently achieved a similar result to ours in the two-dimensional setup, \cite{PesinVaughnLuzatto}. They assume the existence of a geometric rectangle, with boundaries defined by the stable and unstable leaves of two homoclinically related periodic hyperbolic points, such that it contains a set of non-uniformly hyperbolic points which return with positive frequency to the geometric rectangle and to some Pesin set; and such that the saturation of their stable leaves has a positive Riemannian volume. In this case, they show that a hyperbolic SRB measure exists, and go further to show that when a hyperbolic SRB measure exists, this condition is satisfied. Their methods involve the construction of a Young tower, and are inherently two-dimensional (although their methods involve proving a refined shadowing theorem which one expects to extend to high dimensions). An important achievement in their work is the construction of a Young tower which is a first-return tower for a power of $f$. The power depends on the periods of the two periodic points which define the geometric rectangle.

The main difference between our results and theirs is that we do not assume $\dim M=2$, nor the existence of a rectangle as above.

The question whether a leaf condition implies the existence of an SRB measure was raised by Y. Pesin \cite{PrivatePesin}.

\subsection{An overview of the Proof}\label{ProofOutline}
Since the proof is long and technical, we thought it might be useful to give the reader an informal overview of the argument.

The main result says that the leaf condition \eqref{Leaf-Condition} is necessary and sufficient for the existence of an SRB measure.
The  direction $(\Leftarrow)$ is easy:
Suppose a $\chi$-hyperbolic SRB measure exists, then it gives $\HWT^{\mathrm{PR}}_\chi$ full measure. Therefore a.e. conditional measure of the SRB measure gives $\HWT^{\mathrm{PR}}_\chi$ full measure. Since the conditional measures of an SRB measure are absolutely continuous with respect to the induced Riemannian volume measure, the leaf condition \eqref{Leaf-Condition} must hold.
Henceforth,  we focus on the difficult direction $(\Rightarrow)$:
{\em If the leaf condition \eqref{Leaf-Condition} holds, then  a $\chi$-hyperbolic SRB measure exists.}

Constructing an SRB measure means constructing an invariant probability measure whose conditional measures on unstable leaves are absolutely continuous. To do this
{\em we will first  construct a family of absolutely continuous leaf measures which are carried by hyperbolic points, and  then we will integrate this family into a finite invariant measure $\mu$, using s suitably chosen measure $p$  on the space of leaves. The SRB property of $\mu$ is guaranteed, but to ensure invariance, we will need to make sure that the family of leaf measures and the measure $p$ we use to integrate it both satisfy certain explicit transformation laws.}


\medskip
\noindent\textbf{Step 1:  A family of smooth  measures on unstable leaves with a transformation law for $f^{-1}$.} In this step we construct a family of {\em smooth measures $m_{V^u}$ on unstable leaves $V^u$ with an explicit transformation law for the pull-back by $f^{-1}$}.

The full description of the transformation law can be found in Theorem \ref{naturalmeasures}, and we will not repeat it here. It says, roughly, that if $V_1^u, V^u_2$ are two unstable leaves such that $f^{-1}[V^u_1]\subset V^u_2$, then
\begin{equation}\label{Pull-back-trans-law}
m_{V^u_2}\circ f^{-1}|_{V^u_1}=e^{\phi(V^u_1,V^u_2)}m_{V^u_1}
\end{equation}
where {\em the Radon-Nikodym derivative $e^{\phi(V^u_1,V^u_2)}$ is  constant on $V^u_1$}.

The construction is based on a limiting procedure, as in \cite{Si1,uGibbsPesinSinai}: We  start from an absolutely continuous measure on an exponentially small unstable leaf of the point $f^{-n}(x)$, push it forward using $f^n$ to an unstable leaf of $x$, and pass to the limit as $n\to\infty$.
Most of the technical work goes into  controlling the regularity of the densities of the measures we get this way, in the limit as $n\to\infty$.

%

\medskip
\noindent\textbf{Step 2: The Markovian structure of the space of  smooth leaf measures.} Our next task is the find a measure $p$ on the space of leaf measures so that we can integrate the leaf measures into a finite invariant measure.
%
As preparation for this, we first parameterize the family of leaf measures by a one-sided countable Markov shift. This will later enable us obtain $p$ from a suitable measure on this shift space.

We use the {\em countable Markov partition} from \cite{SBO}. While this partition does not cover all of the manifold, it does cover $\HWT^{\mathrm{PR}}_\chi$, and thus it carries all $\chi$-hyperbolic invariant probability measures (see section \ref{symbolicdynamicspart} for details).

A Markov partition induces a coding by the space of all the {\em admissible bi-infinite words} whose letters are elements of the partition.\footnote{Admissible means that for any two consecutive letters which represent partition elements $a$ and $b$ resp., $a\cap f^{-1}[b]\neq\varnothing$.}
The space $\Sig_L$ of all admissible words $\underline{R}=(\ldots, w_{-1},w_0)$ which are infinite to the left (made of letters which are elements of the Markov partition) is endowed with its own natural dynamics and topology.

{\em For every  infinite word $\underline{R}\in\Sig_L$, we construct a corresponding local unstable leaf of maximal dimension $V^u(\underline{R})$}, such that (a) $\underline{R}\mapsto V^u(\underline{R})$  has strong continuity properties; and (b) the pull-back (under $f^{-1}$) of an unstable leaf associated to an infinite word $(\ldots,w_{-2},w_{-1},w_0)$ is contained in the unstable leaf associated to the word $(\ldots,w_{-3}, w_{-2},w_{-1})$.
We then check that (c) the composed map
$$\text{infinite word}\mapsto\text{local unstable leaf of maximal dimension}\mapsto\text{smooth leaf measure from step 1}$$
has good continuity properties. 
At the end of all this work, we have a continuous parametrization of the space of smooth leaf measures from step 1 by means of the collection of left infinite words $(\ldots, w_{-1},w_0)$ of a countable Markov shift.

%

\medskip
\noindent\textbf{Step 3: Canonical codings and a family of absolutely continuous measures with a transformation rule for $f$.}
At this point we need to address a subtle but crucial problem. 

Let $\sigma_R:\Sig_L\to\Sig_L$ be the right shift $\sigma_R(\ldots,w_{-1},w_0):=(\ldots,w_{-2},w_{-1}).$
We would have liked to use  the transformation rule \eqref{Pull-back-trans-law} for $f^{-1}$  to deduce the following transformation rule for $f$:
 $$
m_{V^u(\underline{R})}\circ f^{-1}\overset{?}{=}\sum_{\sigma_R(\underline{S})=\underline{R}}e^{\phi(V^u(\underline{S}))}
m_{V^u(\underline{S})}.
$$
{\em But $\overset{?}{=}$ may be false, because of possible overlaps between $V^u(\underline{S})$ for different $\underline{S}\in\sigma_R^{-1}[\{\underline{R}\}]$.  }
The problem is not just overlaps  at the boundaries of the  elements of the Markov partition of $\HWT^{\mathrm{PR}}_\chi$, but also the overlaps  outside the set covered by the Markov partition.
In step 3 we address this difficulty.

The coding from the second step may not be one-to-one, however the orbit of each point in the set which is covered by the partition has a unique itinerary (i.e. a list of the partition elements in which its orbit elements lie). We call the  collection of all such itineraries, the space of {\em the canonical codings}.

We restrict the smooth measures $m_{V^u(\underline{R})}$ from the second step to the intersection of $V^u(\underline{R})$ with the stable leaves associated to canonical codings of points in $V^u(\underline{R})$. Call these restrictions $\mu_{\underline{R}}$. (It is not clear that $\mu_{\underline{R}}\neq 0$, we address this in step 5.)
%
%
For these measures, we do indeed have: 
\begin{enumerate}
\item  {\em Transformation law for $f$:\/} Let $\phi(\underline{R})=\phi(V^u(\underline{R}), V^u(\sigma_R\ul{R}))$  (see step 1), then
\begin{equation}\label{Trans-Law}
\mu_{\underline{R}}\circ f^{-1}=\sum_{\sigma_R(\underline{S})=\underline{R}}e^{\phi(\underline{S})}\mu_{\underline{S}}.
\end{equation}

\item Absolute continuity with respect to the induced Riemannian leaf volume.
\item  $\mu_{\underline{R}}$ a.e. point is the intersection of a stable leaf and an unstable leaf.

\end{enumerate}
We call the family $\{\mu_{\underline{R}}\}$ the family of  {\em absolutely continuous leaf measures}. 

\medskip
This is the family of leaf measures which we plan to  integrate into an invariant SRB measure of the form $\mu:=\int_{\Sig_L}\mu_{\underline{R}}dp$, for some measure $p$. It remains to find $p$.

Because of the transformation law \eqref{Trans-Law}, in order for $\mu$ to be $f$-invariant, it is sufficient to choose $p$ to be an 
eigenmeasure with eigenvalue one for the Ruelle operator of $\phi$ on $\Sig_L$ (see Definition \ref{RuelleOpe}).
We will call such measures {\em $\phi$-conformal measures.} 

The remaining steps of the proof show that such measures exist. This is done using  the ``generalized Ruelle's Perron-Frobenius Theorem" from \cite{SarigNR}. This theorem  gives a necessary and sufficient condition for the existence of a conservative eigenmeasure $p$ for the Ruelle measure of $\phi$. The following steps verify  these sufficient conditions.

%
%
%

\medskip
\noindent\textbf{Step 4: Modulus of continuity of $\phi(\underline{R})$ on $\Sig_L$.}

The first condition to check is that $\phi(\underline{R})$ is H\"older continuous on $\Sig_L$,  with respect to the natural metric on $\Sig_L$. 
%
%
%
We do this by a careful analysis of the modulus of continuity of the maps in the following composition:
 $$\text{infinite word $\underline{R}$}\mapsto\text{local unstable leaf } V^u(\underline{R})\mapsto\text{smooth leaf measure }  m_{V^u(\underline{R})}\mapsto\text{log-Jacobian }\phi(\underline{R}).$$

\medskip
\noindent\textbf{Step 5: A non-trivial transitive component.}
Sarig's construction of $\phi$-conformal measures also requires  {\em topological transitivity} of the shift. 

The right shift on $\Sig_L$ is not necessarily topologically transitive. In step 5 we find a topologically transitive component of $\Sig_L$ on which $\mu_{\underline{R}}\neq 0$.

The transitive components we use  arise as the codings of ergodic homoclinic classes in the sense of \cite{RodriguezHertz}, see \cite{BCS} and \cite{LifeOfPi}.
%
To find a transitive component with non-vanishing leaf measures $\mu_{\underline{R}}$ we use the leaf condition \eqref{Leaf-Condition}.
This condition guarantees the existence of at least one non-zero $\mu_{\underline{R}}$, and we show that if $\mu_{\underline{R}}\neq 0$ for one leaf, then $\mu_{\underline{S}}\neq 0$ for all $\underline{S}$ in a transitive component of $\Sig_L$.

\medskip
\noindent\textbf{Step 6: Thermodynamic properties of $\phi$.} Next we  check the following thermodynamic properties of $\phi$ (see section \ref{lalipop} for definitions):

\begin{enumerate}
\item the {\em recurrence} of $\phi$,
\item the {\em Gurevich pressure} of $\phi$ is $0$.
\end{enumerate}
This is done using the Leaf condition, and a Borel-Cantelli argument.

Once the thermodynamic properties of $\phi$ are verified, we can apply the generalized Ruelle's Perron-Frobenius Theorem and deduce the existence of an eigenmeasure $p$ for the Ruelle operator. 
 The recurrence of $\phi$ guarantees that $p$ is conservative (i.e. has no non-trivial wandering sets). The vanishing of the Gurevich pressure, ensures that the eigenvalue is equal to one, and therefore that $p$ is a $\phi$-conformal measure. Together with the transformation rule \eqref{Trans-Law}, $\phi$-conformality implies that 
$$\mu:=\int_{\Sig_L}\mu_{\underline{R}}dp$$ is an $f$-invariant, non-identically zero, conservative, and $\sigma$-finite measure for which almost every point has a stable and an unstable local leaf. 
However, it is still not clear that $\mu$ is a finite measure.




\medskip
\noindent\textbf{Step 7: Finiteness of $\mu$.} At this point in the proof, we have a measure
$\mu=\int_{\Sig_L}\mu_{\underline{R}}dp$ as follows:
%
%
\begin{enumerate}
\item $\mu$ is not the zero measure (by choice of the transitive component in step 5),
\item $\mu$ is $f$-invariant (because of the transformation law \eqref{Trans-Law}),	
\item $\mu$ is conservative (by recurrence),
\item $\mu$ is carried by hyperbolic points (by the construction of $\mu_{\underline{R}}$),
\item $\mu$ is $\sigma$-finite and ergodic (these properties hold for $p$ bt \cite{SarigNR}, and we show they extend to $\mu$). 
\end{enumerate}

%
%
%
%

We strengthen (5) and show  that $\mu$ is finite on Pesin sets.

By the leaf condition, for  $\mu$-a.e. orbit, the density of times when  the orbit enters some (fixed) Pesin set is positive. This implies that $\mu$ is finite, otherwise by ergodicity, conservativity, and the Ratio Ergodic Theorem, the density  would have to be equal  to zero.
So $\mu$ must be finite. 

\medskip
\noindent
\textbf{Completion of the proof.} 
Let us normalize $\mu$ to have a total mass one. As explained in the end of step 6, $\mu$ is $f$-invariant. By construction, $\mu_{\underline{R}}$ are carried by intersections of stable and unstable leaves of canonical codings (step 3), and this can be used to show that $\mu$ almost every orbit is hyperbolic. Finally, the  conditional measures of $\mu$ on $V^u(\underline{R})$ are absolutely continuous with respect to the induced Riemannian volume measure, because by construction, they are  proportional to $\mu_{\underline{R}}$. Therefore $\mu$ is a {hyperbolic SRB measure}, and we have shown that the leaf condition \eqref{Leaf-Condition} implies the existence of a hyperbolic SRB measure.

\medskip
\noindent
\textbf{A road map to the proof:}
 Step 1 can found in Theorem \ref{naturalmeasures}. Step 2 can be found in Definition \ref{Doomsday}, Theorem \ref{pizzahat}, Definition \ref{unstablemanifold}, and Lemma \ref{0.1}. Step 3 can be found in Definition \ref{A_R} and Proposition \ref{PropMuR}. Step 4 can be found in Claim \ref{VolHol}. Step 5 can be found in Lemma \ref{canonicassumption} and Proposition \ref{psitilda}. Step 6 can be found in Theorem \ref{PG0}. Step 7 can be found in Theorem \ref{theorem1} and Theorem \ref{posrecclaim}.
\subsection{Notation}
\begin{enumerate}
	\item For every $a,b\in\mathbb{R}$, $c\in\mathbb{R}^+$, $a=e^{\pm c}\cdot b$ means $e^{-c} \cdot b\leq a\leq e^c\cdot b$, and $a=b\pm c$ means $b-c\leq a\leq b+c$.
	\item $\forall a,b\in\mathbb{R}$, $a\wedge b:=\min\{a,b\}$.
	\item TMS stands for a topological Markov shift (i.e. the space of all bi-infinite admissible paths on a directed graph).
	\item For every topological Markov shift $\Sigma$ which is induced by a graph $\mathcal{G}:=(\mathcal{V},\mathcal{E})$ (e.g. Theorem \ref{mainSBO}), for every finite admissible path $(v_0,...,v_l)$, $v_i\in \mathcal{V}$,$0\leq i\leq l$, a cylinder is a subsets of the form $[v_0,...,v_l]_m=\{\ul{u}\in \Sigma: u_{i+m}=v_i,\forall0\leq i\leq l\}$. When the $m$ subscript is omitted, if not mentioned otherwise, $m=0$ or $m=-l$.
\item Let $x\in M$. $T_xM$ is the tangent space to $M$ at $x$, $d_xf:T_xM\rightarrow T_{f(x)}M$ is the differential of $f$ at $x$, and $\Jac(d_xf):=|\det(d_xf)|$ is the Jacobian of $d_xf$ w.r.t. the Riemannian metrics of $T_xM, T_{f(x)}M$.
\item $-\mathbb{N}:=\{\ldots,-2,-1,0\}$.
\item Given a Borel measure $\mu$ and a measurable function $\rho$ which is defined $\mu$-a.e., the measure $\rho\cdot \mu$ (also $\rho\mu$ for short) is defined by $(\rho \mu)(A):=\int_A \rho(x)d\mu(x)$.
\end{enumerate}


\section{The Set of Hyperbolic Points  $\HWT_\chi$}\label{HWTsection}

In this section we introduce the set of hyperbolic points which we use in our construction. The significance of this set is that it carries every hyperbolic $f$-invariant probability measure with Lyapunov exponents  outside of $[-\chi,\chi]$, $\chi>0$), and in particular the hyperbolic SRB measure we wish to construct. In addition, $\HWT_\chi$ is covered by a Markov partition which we use later on (see \textsection \ref{Symba}). Moreover, this set is in a sense the ``maximal" set with these properties, as it also contains all elements of the Markov partition. 

The reader may wonder why we make an effort to describe  $\HWT_\chi$ explicitly, given the fact that it is  a set of full measure for all $\chi$-hyperbolic invariant measures. The reason is that in our proof we will also use leaf measures and infinite measures. For such measures, $\HWT_\chi$ is not necessarily a set of full measure, and it will be important to us to know which points it contains, and which not.

\begin{definition}\label{ChiHyp} Fix $\chi>0$.

\begin{enumerate}
         \item A point $x\in M$ is called {\em $\chi$-summable} if it belongs to the following set:
         \begin{align*}
        \chi\text{-}\mathrm{summ}:=&\{x\in M: \exists\text{ a unique splitting }T_xM=H^s(x)\oplus H^u(x)\text{ s.t. }\\
        &\sup_{\xi_s\in H^s(x),|\xi_s|=1}\sum_{m=0}^\infty|d_xf^m\xi_s|^2e^{2\chi m}<\infty,\sup_{{\xi_u\in H^u(x),|\xi_u|=1}}\sum_{m=0}^\infty|d_xf^{-m}\xi_u|^2e^{2\chi m}<\infty\}.
    \end{align*}
    For each $x\in\chi\text{-}\mathrm{summ}$, write $s(x):=\mathrm{dim}(H^s(x)),u(x):=\mathrm{dim}(H^u(x))$.
    \item A point $x\in M$ is called {\em $\chi$-hyperbolic} if it belongs to the following set:
    \begin{align*}
    \chi\mathrm{-hyp}:=&\{x\in M: \exists\text{ a unique splitting }T_xM=H^s(x)\oplus H^u(x)\text{ s.t. }\forall\xi_s\in H^s(x)\setminus\{0\},\xi_u\in H^u(x)\setminus\{0\},\\
    &\limsup_{n\rightarrow\infty}\frac{1}{n}\log|d_xf^{n}\xi_s|,\limsup_{n\rightarrow\infty}\frac{1}{n}\log|d_xf^{-n}\xi_u|<-\chi\}.
    \end{align*}
    A measure carried by $\chi\mathrm{-hyp}$ is called {\em $\chi$-hyperbolic}.
\end{enumerate}
\end{definition}
Notice that $\chi$-hyp$\subseteq\chi\text{-}\mathrm{summ}$.


The Pesin-Oseledec reduction theorem has many different versions, which are suitable for different setups (see \cite{BP}). We use the version which appears, with proof, in \cite[Theorem~2.4]{SBO}.

\begin{theorem}[Oseledec-Pesin~Reduction~Theorem]\label{pesinoseledec}
For each point $x\in\chi\text{-}\mathrm{summ}$, there exists an invertible linear map $C_\chi(x):\mathbb{R}^d\rightarrow T_xM$, which depends measurably on $x$, such that 
$C_\chi(x)[\mathbb{R}^{s(x)}\times\{0\}]=H^s(x),C_\chi(x)[\{0\}\times\mathbb{R}^{u(x)}]=H^u(x)$. 
$C_\chi(\cdot)$ can be chosen measurably on $\chi\text{-}\mathrm{summ}$, and the choice is unique up to a composition with orthogonal self mappings of $H^s(x),H^u(x)$. In addition,
$$
C_\chi^{-1}(f(x))\circ d_xf\circ C_\chi(x)=\begin{pmatrix}D_s(x)  &   \\  & D_u(x)
\end{pmatrix},
$$
where $D_s(x),D_u(x)$ are square matrices of dimensions $s(x),u(x)$ respectively, and $\|D_s(x)\|,\|D_u^{-1}(x)\|\leq e^{-\chi}$,$\|D_s^{-1}(x)\|,\|D_u(x)\|\leq \kappa$ for some constant $\kappa=\kappa(f,\chi)>1$.
\end{theorem}

It is possible to show that $\|C_\chi^{-1}(x)\|^2=\sup\limits_{\overset{|\xi_s+\xi_u|=1,}{\xi_s\in H^s(x), \xi_u\in H^u(x)}}\left\{2\sum\limits_{m\geq0}|d_xf^m\xi_s|^2e^{2\chi m}+ 2\sum\limits_{m\geq0}|d_xf^{-m}\xi_u|^2e^{2\chi m}\right\}$. See \cite[Theorem~2.4]{SBO} for details. $\|C_\chi^{-1}(x)\|$ depends only on the unique splitting $T_xM=H^s(x)\oplus H^u(x)$. $\|C_\chi^{-1}(x)\|$ serves as a measurement of the hyperbolicity of $x$: The greater the norm, the worse the hyperbolicity (slow contraction/expansion on stable/unstable spaces, or small angle between the stable and unstable spaces).

\begin{definition}\label{littleQ} Let $\epsilon>0$, and let $x\in\chi\text{-}\mathrm{summ}$, then

\medskip
    $$Q_\epsilon(x):=\max\{Q\in \{e^{\frac{-\ell\epsilon}{3}}\}_{\ell\in\mathbb{N}}:Q\leq 3^\frac{-6}{\beta}\epsilon^\frac{90}{\beta}\|C^{-1}_\chi(x)\|^\frac{-48}{\beta}\}.$$
\end{definition}
$Q_\epsilon(\cdot)$ depends only on the norm of $C_\chi^{-1}(\cdot)$ , and is indifferent to composition with orthogonal self mappings of the ``stable" and ``unstable" subspaces.

\begin{definition}[Recurrent $\epsilon$-weak temperability]\label{temperable}
Let $\epsilon>0$. A point $x\in\chi\text{-}\mathrm{summ}$ is called {\em $\epsilon$-weakly temperable} if  $\exists q:\{f^n(x)\}_{n\in\mathbb{Z}}\rightarrow \{e^{\frac{-\ell\epsilon}{3}}\}_{\ell\in\mathbb{N}}$ s.t.
\begin{enumerate}
	\item $\frac{q\circ f}{q}=e^{\pm\epsilon}$,
	\item $\forall n\in\mathbb{Z}$, $q(f^n(x))\leq Q_\epsilon(f^n(x))$.
\end{enumerate}
An $\epsilon$-weakly temperable point $x$ is called {\em recurrently $\epsilon$-weakly temperable}, if in addition to (1),(2), $q:\{f^n(x)\}_{n\in\mathbb{Z}}\rightarrow \{e^{-\frac{\ell \epsilon}{3}}\}_{\ell\in\mathbb{N}}$ can be chosen to satisfy also (3):
\begin{enumerate}
\item[(3)]$\limsup\limits_{n\rightarrow\infty}q(f^n(x)), \limsup\limits_{n\rightarrow\infty}q(f^{-n}(x))>0$.
\end{enumerate}
Define $$\WT_\chi^\epsilon:= \{x\in\chi\text{-}\mathrm{summ}:x\text{ is } \epsilon\text{-weakly temperable}\},$$ and  $$\HWT_\chi^\epsilon:= \{x\in\chi\text{-}\mathrm{summ}:x\text{ is recurrently } \epsilon\text{-weakly temperable}\}.$$
 $\WT_\chi^\epsilon$ is the set of {\em weakly temperable points}, with parameters $\chi,\epsilon>0$, and $\HWT_\chi^\epsilon$ is the set of {\em recurrently weakly temperable points}, with parameters $\chi,\epsilon>0$.
\end{definition}

\noindent Theorem \ref{epsilonika} gives an important application to the definition of $\HWT_\chi^\epsilon$.

\medskip
\noindent In the following parts of this paper, when $\chi>0$ is fixed, the subscript of $\epsilon_\chi$ would be omitted to ease notation. In addition, we may assume $\epsilon>0$ is arbitrarily small, since the results of \cite{SBO,LifeOfPi} apply to all $\epsilon\in (0,\epsilon_\chi]$, for a fixed $\chi>0$.

\begin{definition}\label{NUHsharp}
 $$\HWT_\chi:=\Big\{x\in\chi\text{-}\mathrm{summ}:x\text{ is recurrently }\epsilon_\chi\text{-weakly temperable}\Big\}.$$
\end{definition}

\noindent\textbf{Remark:} $\HWT_\chi$ carries all $\chi$-hyperbolic $f$-invariant probability measures; and $\HWT_\chi$ is defined canonically,\footnote{I.e. its definition does not rely on a specific construction of symbolic dynamics, but only on the quality of hyperbolicity of the orbit of the point.} see \cite{LifeOfPi}. In the upcoming parts of this paper, we focus our attention to this set, when constructing a $\chi$-hyperbolic SRB measure.

\section{Preliminary Constructions}\label{symbolicdynamicspart}
\subsection{Symbolic Dynamics}\label{Symba}
Sarig constructed a Markov partition for non-uniformly hyperbolic surface diffeomorphisms in \cite{Sarig}. Later, we extended his results to manifolds of any dimension greater or equal to 2 in \cite{SBO}. These codings do not code all $x\in M$. In \cite{LifeOfPi}, we gave a canonical description of all points with codings which do not escape to infinity (i.e. do not escape every compact set of the symbolic space, see Proposition \ref{cafeashter}). In this section, we present an overview of these results.

Since $M$ is compact, $\exists r=r(M)>0,\rho=\rho(M)>0$ s.t. the exponential map $\exp_x: \{v\in T_xM:|v|\leq r\}\rightarrow B_\rho(x)=\{y\in M: d(x,y)<\rho\}$ is well defined and smooth. When $\epsilon\leq r$, the following is well defined since $C_\chi(\cdot)$ is a contraction (see \cite{BP},\cite{KM},\cite[Lemma~2.9]{SBO}):
\begin{definition}[Pesin-charts]\label{PesinCharts}\text{}
\begin{enumerate}
	\item $\psi_x^\eta:=\exp_x\circ C_\chi(x):\{v\in T_xM:|v|\leq \eta\}\rightarrow B_\rho(x)$, $\eta\in (0,Q_\epsilon(x)]$, is called a {\em Pesin-chart}.
	\item A {\em double Pesin-chart} is an ordered couple $\psi_x^{p^s, p^u}:=(\psi_x^{p^s}, \psi_x^{p^u})$, where $\psi_x^{p^s}$ and $\psi_x^{p^u}$ are Pesin-charts.
\end{enumerate}
\end{definition}

\begin{theorem}\label{mainSBO}
	 $\forall \chi>0$ s.t. $\exists p\in\chi$-$\mathrm{hyp}$ a periodic hyperbolic point, $\exists$ a countable and locally-finite\footnote{I.e. finite number of in-going and out-going edges at each vertex.} directed graph $\mathcal{G}= \left(\mathcal{V},\mathcal{E}\right)$ which induces a topological Markov shift $
\Sigma:=\{\ul{u}\in\mathcal{V}^\mathbb{Z}:
(u_i,u_{i+1})\in	\mathcal{E},\forall i\in\mathbb{Z}\}$. $\Sigma$ admits a map $\pi:\Sigma\rightarrow M$ with the following properties:
	 \begin{enumerate}
	 	\item $\sigma:\Sigma\rightarrow\Sigma$, $(\sigma \ul{u})_i:=u_{i+1}$, $i\in \mathbb{Z}$ (the left-shift); $\pi\circ\sigma=f\circ\pi$.
	 	\item $\pi$ is a H\"older continuous map w.r.t. to the metric $d(\ul{u},\ul{v}):=\exp\left(-\min\{i\geq0:u_i\neq v_i\text{ or } u_{-i}\neq v_{-i}\}\right)$.
	 	\item Let $\Sigma^\#:=\left\{\ul{u}\in \Sigma:\exists n_k,m_k\uparrow\infty\text{ s.t. }u_{n_k}=u_{n_0}, u_{-m_k}=u_{-m_0},\forall k\geq0\right\}$. Then 
	 $\pi[\Sigma^\#]$ carries all $f$-invariant, $\chi$-hyperbolic probability measures.
	 \end{enumerate}
\end{theorem}
This theorem is the content of \cite[Theorem~3.13]{SBO} (and similarly, the content of \cite[Theorem~4.16]{Sarig} when $d=2$). $\mathcal{V}$ is a collection of double Pesin-charts (see 
Definition \ref{PesinCharts}),
with the following discreteness property: Every $v\in\mathcal{V}$ is a double Pesin-chart of the form $v=\psi_x^{p^s,p^u}$ with $0<p^s,p^u\leq Q_\epsilon(x)$; and discreteness means that $\forall \eta>0:$ $\#\{v\in\mathcal{V}:v=\psi_x^{p^s,p^u} \text{ with }p^s\wedge p^u>\eta\}<\infty$. 

\begin{definition}\label{Doomsday}
\text{ }

\medskip
\begin{enumerate}
	\item $\forall u\in \mathcal{V}$, $Z(u):=\{\pi[\ul{u}]:\ul{u}\in\Sigma^\#, u_0=u\}$, $\mathcal{Z}:=\{Z(u):u\in\mathcal{V}\}$.
	\item $\mathcal{R}$ is a countable partition of $\bigcup\limits_{v\in\mathcal{V}}Z(v)=\pi[\Sigma^\#]$, s.t.
	\begin{enumerate}
	\item $\mathcal{R}$ is a refinement of $\mathcal{Z}$: $\forall Z\in\mathcal{Z},R\in\mathcal{R}$, $R\cap Z\neq\varnothing\Rightarrow R\subseteq Z$.
	\item $\forall v\in\mathcal{V}$, $\#\{R\in\mathcal{R}:R\subseteq Z(v)\}<\infty$ (\cite[\textsection~11]{Sarig}).
	\item The rectangles property: $\forall R\in\mathcal{R}$,$\forall x,y\in R$ $\exists ! z:=[x,y]_R\in R$, s.t. $\forall i\geq0, R(f^i(z))= R(f^i(y)), R(f^{-i}(z))= R(f^{-i}(x))$, where $R(t):=$the unique partition member of $\mathcal{R}$ which contains $t$, for $t\in\pi[\Sigma^\#]$.
	\item Symbolic Markov property: Let $R,S\in\mathcal{R}$, and let $x\in R\cap f^{-1}[S]$. Let $\ul{u}\in\Sigma^\#$ s.t. $\pi(\ul{u})=x$. Let $W^u(x,R):=\{y\in R:y\in V^u(\ul{u})\}$, $W^s(x,R):=\{y\in R:y\in V^s(\ul{u})\}$ (see \textsection \ref{foliage} for $V^s(\cdot),V^u(\cdot)$). Then $f^{-1}[W^u(f(x),S)]\subseteq W^u(x,R)$ and $f[W^s(x,R)]\subseteq W^s(f(x),S)$.
	\end{enumerate}
	\item $\forall R,S\in\mathcal{R}$, we say $R\rightarrow S$ if $R\cap f^{-1}[S]\neq\varnothing$. Let $\widehat{\mathcal{E}}=\{(R,S)\in\mathcal{R}^2\text{ s.t. }R\rightarrow S\}$.
	\item $\Sig:=\{\ul{R}\in\mathcal{R}^{\mathbb{Z}}: R_i\rightarrow R_{i+1},\forall i\in\mathbb{Z}\}$. This is the TMS associated to the graph $\widehat{\mathcal{G}}=(\mathcal{R},\widehat{\mathcal{E}})$.
\end{enumerate}	
\end{definition}
\noindent\textbf{Remarks:}
\begin{enumerate}
\item Given $\mathcal{Z}$, such a refining partition as $\mathcal{R}$ exists by the Bowen-Sinai refinement, see \cite[\textsection~11.1]{Sarig}.
\item Property (2)(d) makes $\mathcal{R}$ a {\em Markov partition}, in sense close to \cite{Si2,B1,AW}.
\item By property (2)(b), and since $\Sigma$ is locally-compact (see Theorem \ref{mainSBO}, local-finiteness of $\mathcal{G}$ implies local-compactness of $\Sigma$), $\Sig$ is also locally-compact.
 	
\end{enumerate}
\begin{definition}[$\text{\cite[\textsection 12.2,\textsection 12.3]{Sarig}}$]\label{sigmasharp}\text{ }

\medskip
\begin{enumerate}
    \item $\Sig^\#:=\{\ul{R}\in\Sig:\exists n_k,m_k\uparrow\infty\text{ s.t. }R_{n_k}=R_{n_0},R_{-m_k}=R_{-m_0},\forall k\geq0\}$.
    \item Two partition members $R,S\in \mathcal{R}$ 
are said to be {\em affiliated} if $\exists u,v\in\mathcal{V}$ s.t. $R\subseteq Z(u), S\subseteq Z(v)$ and $Z(u)\cap Z(v)\neq\varnothing$
.
\end{enumerate}
\end{definition}
\begin{claim}[Local finiteness of the cover $\mathcal{Z}$]\label{localfinito}
	$\forall Z\in\mathcal{Z}$,$\#\{Z'\in\mathcal{Z}:Z'\cap Z\neq\varnothing\}<\infty$.
\end{claim}
This claim is the content of \cite[Theorem~5.2]{SBO} (and similarly \cite[Theorem~10.2]{Sarig} when $d=2$).
\noindent\textbf{Remark:} 
By Claim \ref{localfinito} and Definition \ref{Doomsday}(2)(b), every partition member of $\mathcal{R}$ has only a finite number of partition members affiliated to it.

The coding $\pi:\Sigma\rightarrow M$ is usually $\infty$-to-one. Using $\mathcal{R}$, we can obtain a finite-to-one coding as follows.
\begin{theorem}\label{pizzahat}
	Given $\Sig$ from Definition \ref{Doomsday}, there exists a map $\widehat{\pi}:\Sig\rightarrow M$ s.t.
	\begin{enumerate}
	\item $f\circ\widehat{\pi}=\widehat{\pi}\circ\sigma$, where $\sigma$ denotes the left-shift on $\Sig$.
	\item $\widehat{\pi}$ is H\"older continuous w.r.t. the metric $d(\ul{R},\ul{S})=\exp\left(-\min\{i\geq0: R_i\neq S_i\text{ or }R_{-i}\neq S_{-i}\}\right)$.
	\item $\widehat{\pi}|_{\Sig^\#}$ is finite-to-one.
	\item $\forall \ul{R}\in \Sig$, $\widehat{\pi}(\ul{R})\in \overline{R_0}$.
	\item $\widehat{\pi}[\Sig^\#]$ carries all $\chi$-hyperbolic invariant probability measures.
	\end{enumerate}
\end{theorem}
This theorem is the content of the main theorem of [BO18], Theorem 1.1 (and similarly the content of [Sar13, Theorem 1.3] when $d = 2$).

\begin{theorem}[\cite{LifeOfPi}]\label{epsilonika}
	$\exists \epsilon_\chi>0$ which depends only on $M,f,\beta$ and $\chi$, s.t. $\forall \epsilon\in(0,\epsilon_\chi]$, $\HWT_\chi^\epsilon=\widehat{\pi}[\Sig^\#]$, where $\widehat{\pi}[\Sig^\#]= \widehat{\pi}[\Sig^\#(\epsilon)]$ depends on $\epsilon$, and is given by Theorem \ref{pizzahat}.
\end{theorem}

\begin{prop}\label{cafeashter}
$\widehat{\pi}[\Sig^\#]=\pi[\Sigma^\#]=\bigcupdot\mathcal{R}=\HWT_\chi$.
\end{prop}
This is the content of \cite[Proposition~4.11, Corollary~4.12]{LifeOfPi}.

\subsection{Unstable Leaves of Maximal Dimension}\label{foliage}
\begin{definition}
An {\em unstable leaf} (of $f$) in $M$, $V^u$, is a $C^{1+\frac{\beta}{3}}$-regular, embedded, open, Riemannian submanifold of $M$, such that $\forall x,y\in V^u$, $\limsup\limits_{n\rightarrow\infty} \frac{1}{n}\log d(f^{-n}(x),f^{-n}(y))<0$.
Similarly, a {\em stable leaf} is an unstable leaf of $f^{-1}$.
\end{definition}
\begin{definition}\label{maximaldimensionleaves}
An unstable leaf is called {\em an unstable leaf of maximal dimension}, if it is not contained in any unstable leaf of a greater dimension.
\end{definition}
Notice that if $x\in \HWT_\chi$ belongs to an unstable leaf of maximal dimension $V^u$, then $\mathrm{dim}H^u(x)=\mathrm{dim}V^u$. This can be seen from the following claim.
\begin{claim}\label{chikfila}
$\forall \ul{u}\in \Sigma$, there exists an unstable leaf of maximal dimension $V^u(\ul{u})$, which depends only on $(u_i)_{i\leq0}$, and a stable leaf $V^s(\ul{u})$ of maximal dimension, which depends only on $(u_{i})_{i\geq0}$, s.t. $\{\pi(\ul{u})\}=V^u(\ul{u})\cap V^s(\ul{u})$.	
\end{claim}
This is the content of \cite[Proposition~3.12,Theorem~3.13, Proposition~4.4]{SBO} (and similarly \cite[Proposition~4.15,Theorem~4.16,Proposition~6.3]{Sarig} when $d=2$). By construction, $V^s(\ul{u}),V^u(\ul{u})$ are local unstable leaves with finite induced Riemannian volume.
\begin{claim}
	$\forall \ul{u}\in\Sigma$, $f[V^s(\ul{u})]\subset V^s(\sigma\ul{u})$, $f^{-1}[V^u(\ul{u})]\subset V^u(\sigma^{-1}\ul{u})$.
\end{claim}
This is the content of \cite[Proposition~3.12]{SBO} (and similarly \cite[Proposition~4.15]{Sarig} when $d=2$).

\subsection{Ergodic Homoclinic Classes and Maximal Irreducible Components}\label{hertzetAl}

In this section we present the definition of certain invariant sets which correspond to topologically transitive components of the symbolic coding. These components are essential for the discussion on thermodynamic properties of the symbolic space (see \textsection \ref{lalipop}).

\begin{definition}
	Let $x\in \HWT_\chi$, and let $\ul{u}\in\Sigma^\#$ s.t. $
	\pi(\ul{u})=x$. The {\em global stable (respectively unstable)} manifold of $x$ is $W^s(x):=\bigcup_{n\geq0}f^{-n}[V^s(\sigma^n\ul{u})]$ {\em(respectively }$W^u(x):=\bigcup_{n\geq0}f^{n}[V^u(\sigma^{-n}\ul{u})]${\em)}.
\end{definition}
This definition is proper and is independent of the choice of $\ul{u}$, for more details on the size of the leaves $V^u(\cdot)$ under shift see \cite[Definition~2.23,Definition~3.2]{SBO}.

Let $p$ be a periodic point in $\chi\text{-}\mathrm{summ}$, i.e. hyperbolic periodic point. Since $p$ is periodic, $\|C^{-1}_\chi(\cdot)\|$ is bounded along the orbit of $p$, and therefore $p\in \HWT_\chi$.
\begin{definition}\label{homoclinicclass}The {\em ergodic homoclinic class} of $p$ (w.r.t. $\chi$) is
    $$H_\chi(p):=\left\{x\in \HWT\chi:W^u(x)\pitchfork W^s(o(p))\neq\varnothing,W^s(x)\pitchfork  W^u(o(p))\neq\varnothing\right\},$$
    where $\pitchfork$ denotes transverse intersections of full codimension, $o(p)$ is the (finite) orbit of $p$, and $W^{s/u}(\cdot)$ are the global stable and unstable manifolds of the point, respectively.
    \end{definition}
\noindent This notion was introduced in \cite{RodriguezHertz}[\textsection~1.1], with a set of Lyapunov regular points replacing $\HWT_\chi$. Every ergodic conservative $\chi$-hyperbolic measure is carried by an ergodic homoclinic class of some periodic hyperbolic point.
\begin{definition}\label{irreducibility}
Consider the partition $\mathcal{R}$ from Definition \ref{Doomsday}.
\begin{enumerate}
    \item Define $\sim\subseteq\mathcal{R}\times\mathcal{R}$ by $R\sim S\iff \exists n_{RS},n_{SR}\in\mathbb{N}\text{ s.t. } R\xrightarrow[]{n_{RS}}S,S\xrightarrow[]{n_{SR}}R$, i.e. a path of length $n_{RS}$ connecting $R$ to $S$, and a path of length $n_{SR}$ connecting $S$ to $R$. The relation $\sim$ is transitive and symmetric. When restricted to $\{R\in \mathcal{R}:R\sim R\}$, it is also reflexive, and thus an equivalence relation. Denote the corresponding equivalence class of some representative $R\in\mathcal{R}$, $R\sim R$, by $\langle R\rangle$.
    \item A {\em maximal irreducible component} in $\Sig$, corresponding to $R\in\mathcal{R}$ s.t. $R\sim R$, is $\{\ul{R}\in\Sig: \ul{R}\in\langle R\rangle^\mathbb{Z}\}$.
\end{enumerate}
\end{definition}
\begin{prop}\label{homoclinicirreducible}
Let $p$ be a periodic $\chi$-hyperbolic point. Then, there exists a maximal irreducible component, $\tilde{\Sigma}\subseteq\Sig$, s.t. $\widehat{\pi}[\tilde{\Sigma}^\#]= H_\chi(p)$ modulo all conservative measures, where $\tilde{\Sigma}^\#:=\{\ul{u}\in\tilde{\Sigma}:\exists v,w\text{ s.t. }\#\{i>0:u_i=v\}=\#\{i<0:u_i=w\}=\infty\}$.
\end{prop}
This is the content of \cite[Theorem~5.9]{LifeOfPi}, and is based on the result of Buzzi, Crovisier and Sarig in \cite{BCS} for homoclinic classes of the type of Newhouse \cite{NewhousePeriodicEquivalenceRelation}, and the $f$-invariant, $\chi$-hyperbolic, probability measures which they carry.

\medskip
\subsection{The Canonical Part of the Symbolic Space}\label{CanoCodi}
In this section we introduce a dense invariant subset of the symbolic space, whose image covers the union of the Markov partition in a {\em one-to-one} way.
Later, we will use this set to define a collection of absolutely continuous leaf measures with a good transformation law for the action of $f$, see section \ref{ProofOutline}, step 3.
%
%

\begin{definition} Let $-\mathbb{N}=(\ldots,-2,-1,0)$ and set
\begin{align*}\Sig_L:=&\{(R_i)_{i\leq0}\in \mathcal{R}^{-\mathbb{N}}:\forall i\leq 0, R_{i-1}\rightarrow R_i\},\text{ }\sigma_R:\Sig_L\rightarrow\Sig_L,\sigma_R((R_i)_{i\leq0})=(R_{i-1})_{i\leq0}.
\end{align*}
\end{definition}
Notice, $\sigma_R$ is the right-shift, not the left-shift $\sigma$. In order to prevent any confusion, we will always notate $\sigma_R$ with a subscript $R$ (for ``right"), when considering the right-shift.

\begin{definition}[{\em The canonical coding} $\ul{R}(\cdot)$]
$$\forall x\in \pi[\Sigma^\#]=\bigcupdot\mathcal{R}\text{, }(\ul{R}(x))_i:=R(f^i(x)), i\in\mathbb{Z}.$$
\end{definition}

\begin{definition}\label{chikfilb}\label{unstabledisconnected} $$\forall \ul{R}\in\Sig_L,\text{ }W^u(\ul{R}):=\bigcap_{j=0}^\infty f^j[R_{-j}].$$
\end{definition}
One should notice that $\widehat{\pi}(\ul{R}(x))=x,W^u(\ul{R}(x))\ni x$, and so $W^u(\ul{R}(x))\neq\varnothing$.

We have the following very important property:
\begin{cor}\label{decomposition} $\forall \ul{R}\in\Sig_L$, $$f[W^u(\ul{R})]=\bigcupdot_{\sigma_R\ul{S}=\ul{R}}W^u(\ul{S}).$$
\end{cor}
\begin{proof} Since $f$ is a diffeomorphism,
 \begin{align*}
     f[W^u(\ul{R})]=&f[\bigcap_{j=0}^\infty f^j[R_{-j}]]=\bigcap_{j=0}^\infty f^{j+1}[R_{-j}]=f[R_0]\cap\bigcap_{j=1}^\infty f^{j+1}[R_{-j}]\\
     =&
     (\bigcupdot_{R_0\rightarrow S}S)\cap f[R_0]\cap\bigcap_{j=1}^\infty f^{j+1}[R_{-j}]=\bigcupdot_{R_0\rightarrow S}(S\cap\bigcap_{j=0}^\infty f^{j+1}[R_{-j}])=\bigcupdot_{\sigma_R\ul{S}=\ul{R}}W^u(\ul{S}),
 \end{align*}
 where the transition from the top equation to the bottom one, is due to the fact that $f[R_0]\subseteq\bigcupdot_{R_0\rightarrow S}S$ by definition, whence $f[R_0]=f[R_0]\cap\bigcupdot_{R_0\rightarrow S}S$
.

\end{proof}

\begin{definition}\label{canonicparts} Let
$$\Sig^\circ:=\{\ul{R}\in\Sig:\forall n\in\mathbb{Z}, f^n(\widehat{\pi}(\ul{R}))\in R_n\},$$
$$\Sig^\circ_L:=\{\ul{R}\in\Sig^\#_L:W^u(\ul{R})\neq\varnothing\},$$
where $\Sig^\#_L:=\{(R_i)_{i\leq 0}:(R_i)_{i\in\mathbb{Z}}\in \Sig^\#\}$.
We call $\Sig^\circ,\Sig^\circ_L$ {\em the canonical parts of the respective symbolic spaces}.
\end{definition}

Notice that $\ul{R}(\cdot)$ is the inverse of $\widehat{\pi}|_{\Sig^\circ}$, and $\ul{R}(x)\in \Sig^\circ$ for all $x\in\bigcupdot\mathcal{R}$.

\noindent\textbf{Remark:} One should notice that $\Sig^\circ\subseteq\Sig^\#$, $\Sig_L^\circ\subseteq\Sig^\#_L$. This can be seen as follows:
If $\ul{R}\in \Sig^\circ$, then $\widehat{\pi}(\ul{R})\in\bigcupdot\mathcal{R}=\pi[\Sigma^\#]$. Take any $\ul{u}\in \Sigma^\#$ s.t. $\pi(\ul{u})=\widehat{\pi}(\ul{R})\equiv x$, then $Z(u_i)\supseteq R(f^i(x))$, $\forall i\in\mathbb{Z}$, whence by the local-finiteness of the refinement and the pigeonhole principle, $\ul{R}(x)\in \Sig^\#$. In addition, since $\Sig^\circ_L=\tau[\Sig^\circ]$ (where $\tau$ is the projection to the non-positive coordinates), and $\Sig^\#_L=\tau[\Sig^\#]$, we get $\Sig_L^\circ\subseteq\Sig^\#_L$.

Next, since every admissible cylinder contains a point in $\Sig^\circ$ (\cite[Lemma~12.1]{Sarig}), we get that $\Sig^\circ$ is a dense invariant subset (Corollary \ref{imagecanonic} shows in addition that its image under $\widehat{\pi}$ covers the Markov partition elements). Thus, $\Sig^\circ_L$ is dense $\Sig_L$; and for every $\ul{R}\in \Sig^\circ_L,\ul{S}\in\Sig_L$ s.t. $\sigma_R\ul{S}=\ul{R}$, the Markov property tells us there is a point in $W^u(\ul{S})$- whence $\Sig^\circ_L$ is also invariant. 
\begin{cor}\label{imagecanonic}
$$\widehat{\pi}[\Sig^\circ]=\widehat{\pi}[\Sig^\#]=\pi[\Sigma^\#]=\bigcupdot\mathcal{R}.$$
\end{cor}
\begin{proof}
In the remark after Definition \ref{canonicparts} we saw that $\Sig^\circ\subseteq\Sig^\#$. In addition, $\pi[\Sigma^\#]\subseteq \widehat{\pi}[\Sig^\circ]$ because for any $x\in\pi[\Sigma^\#]=\bigcupdot\mathcal{R}$, $\ul{R}(x)\in\Sig^\circ$ and $\widehat{\pi}(\ul{R}(x))=x$. 
In total, by Proposition \ref{cafeashter},
$$\widehat{\pi}[\Sig^\circ]\subseteq\widehat{\pi}[\Sig^\#]=\bigcupdot\mathcal{R}=\pi[\Sigma^\#]\subseteq\widehat{\pi}[\Sig^\circ].$$
\end{proof}

\section{Space of Absolutely Continuous Leaf Measures}\label{leavesandmeasures}
\subsection{Unstable Leaves and Smooth Measures}\label{smoothies}
In this section we give a definition for unstable leaves of chains in $\Sig$. This is different from  the unstable leaves from \textsection \ref{foliage}, which were parameterized by  chains in $\Sigma$. The main challenge in doing this, is that unlike the elements of $\Sigma$, chains in $\Sig$ are not made of (double) Pesin charts. Therefore, they do not come equipped with a canonical choice of a local domain for their unstable leaf. We show how to choose the domain in such a way that the resulting family of leaves has the following  properties:
(1) $f^{-1}[V^u(\ul{R})]\subseteq V^u(\sigma_R\ul{R})$; and (2) The map $\ul{R}\mapsto V^u(\ul{R})$ has good continuity properties. 
We need these properties for the proofs of  Theorem \ref{naturalmeasures} and Lemma \ref{0.1}.

\begin{lemma}\label{nextstepcover}
If $R_0\rightarrow R_1$, where $R_0,R_1\in\mathcal{R}$, and $R_0\subseteq Z(u_0)$, $u_0\in\mathcal{V}$ then $\exists u_1\in\mathcal{V}$ s.t. $u_0\rightarrow u_1$ and $R_1\subseteq Z(u_1)$. Similarly, if $R_0\rightarrow R_1$ and $R_1\subseteq Z(v_1)$, then $\exists v_0$ s.t. $v_0\rightarrow v_1$ and $R_0\subseteq Z(v_0)$.
\end{lemma}
\begin{proof}
Let $x\in f^{-1}[R_1]\cap R_0$ (exists by definition since $R_0\rightarrow R_1$). By assumption, $x\in Z(u_0)$, hence $\exists \tilde{\ul{u}}\in\Sigma^\#$ s.t. $\tilde{u}_0=u_0$ and $\pi(\tilde{\ul{u}})=x$. Define $u_1:=\tilde{u}_1$; thus $f(x)=\pi(\sigma\tilde{\ul{u}})$. Hence $f(x)\in R_1\cap Z(\tilde{u}_1)=R_1\cap Z(u_1)$. Since $R_1\cap Z(u_1)\neq\varnothing$, and $\mathcal{R}$ refines $\mathcal{Z}$, $R_1\subseteq Z(u_1)$. The proof of the second part is similar.
\end{proof}
\begin{definition}
Given $\underline{R}\in \Sig_L$ we say that $\underline{u}\in\Sig_L$ {\em covers} $\underline{R}$, if $R_i\subseteq Z(u_i)$ for all $i\leq0$. We write $\ul{u}\CAR \ul{R}$.
\end{definition}
By Lemma \ref{nextstepcover}, every $\underline{R}\in \Sig_L$, is covered by some $\ul{u}\in\Sig_L$.

\begin{definition}\label{unstablemanifold} Given $R\in\mathcal{R}$, define $W(R):=\bigcap\{\psi_{x_0}[R_{p_0^u}(0)]:\ul{u}\CAR{}\ul{R}, u_i=\psi_{x_i}^{p_i^s,p_i^u},R_0=R\}=\bigcap\{\psi_{x_0}[R_{p_0^u}(0)]:Z(u)\supseteq R,u=\psi_{x_0}^{p_0^s,p_0^u}\}$, where $R_{p_0^u}(0)$ is the open $|\cdot|_\infty$-ball (box) around $0$ of radius $p_0^u$.

\medskip

Given a chain $\underline{R}\in \Sig_L$, we define $V^u(\ul{R}):=W(R_0)\cap V^u(\ul{u})$ for some (any) $\ul{u}\CAR{}\ul{R}$.
\end{definition}
\noindent The equality in the first part of Definition \ref{unstablemanifold} is since Lemma \ref{nextstepcover} implies that for every $\ul{R}$ s.t. $R_0=R$, and every $u$ s.t. $Z(u)\supseteq R$, $u$ can be extended to a chain $\ul{u}$ s.t. $\ul{u}\CAR\ul{R}$.

\begin{lemma}\label{properness} Definition \ref{unstablemanifold} is proper: $V^u(\ul{R})$ is independent of the choice of $\ul{u}$.
\end{lemma}
\begin{proof}
By Claim \ref{localfinito}, $\forall R\in\mathcal{R}$, $|\{u:Z(u)\supseteq R\}|<\infty$. Therefore $W(R_0)$ is well defined and is an open set. Assume
$\underline{u},\ul{v}\CAR\ul{R}$, and write $u_i=\psi_{x_i}^{p^s_i,p^u_i}, v_i=\psi_{y_i}^{q^s_i,q^u_i}$, $i\leq0$. We show $V^u(\ul{u})\cap W(R_0)=V^u(\ul{v})\cap W(R_0)$. Foe every integer $i\leq0$, $Z(u_i)\cap Z(v_i)\supseteq R_i$, hence $\exists z_i\in R_i, \tilde{\ul{u}}^{(i)},\tilde{\ul{v}}^{(i)}\in\Sigma^\#$ s.t. $\pi(\tilde{\ul{u}}^{(i)})=\pi(\tilde{\ul{v}}^{(i)})=z_i$ and $\tilde{u}^{(i)}_0=u_i,\tilde{v}^{(i)}_0=v_i$.
Therefore, by \cite[Theorem~4.13]{SBO},
$\psi_{y_{-i}}^{-1}\circ\psi_{x_{-i}}=O_{-i}+a_{-i}+\Delta_{-i}$, where $O_{-i}$ is an orthogonal linear transformation, $|a_{-i}|_\infty<10^{-1}(q_{-i}^s\wedge q_{-i}^u)$ is a vector of constants, and $\Delta_{-i}:R_\epsilon(0)\rightarrow\mathbb{R}^d$ is a differentiable map s.t. $\Delta_{-i}(0)=0$ and  $\|d_\cdot\Delta_{-i}\|\leq\frac{1}{2}\epsilon^\frac{1}{3}$. The same theorem also states that $\frac{p_{-i}^u}{q_{-i}^u}=e^{\pm\epsilon^\frac{1}{3}}$.
Take some $z\in V^u(\ul{u})\cap W(R_0)\subseteq V^u(\ul{u})\cap \psi_{y_0}[R_{q^u_0}(0)]$, then $\forall i\geq0$ $f^{-i}(z)\in\psi_{x_{-i}}[R_{p_{-i}^u}(0)]=\psi_{y_{-i}}[\psi_{y_{-i}}^{-1}\circ\psi_{x_{-i}}[R_{p_{-i}^u}(0)]]\subseteq\psi_{y_{-i}}[R_{10\sqrt{d}Q_{\epsilon}(y_{-i})}(0)]$. 
As in \cite[Proposition~2.21,Proposition~3.12(4)]{SBO},
$$V^u(\ul{v})=\{x\in\psi_{y_0}[R_{q^u_0}(0)]:\forall i\geq0, f^{-i}(x)\in \psi_{y_{-i}}[R_{10\sqrt{d}Q_{\epsilon}(y_{-i})}(0)]\}.$$ Hence, $z\in V^u(\ul{v})$. So $V^u(\ul{u})\cap W(R_0)\subseteq V^u(\ul{v})\cap W(R_0)$. By symmetry $V^u(\ul{u})\cap W(R_0)= V^u(\ul{v})\cap W(R_0)$.
%
\end{proof}
\begin{definition}
	$$\Sigma_L:=\{(u_i)_{i\leq0}:\ul{u}\in\Sigma\}.$$
\end{definition}
\begin{cor}\label{Riemannianvolume}
$V^u(\ul{R})$ is an open submanifold of $M$. As such it is equipped with an induced (positive and finite) Riemannian volume measure.
\end{cor}
\begin{proof}
Let some $\ul{u}\in \Sigma_L$ s.t. $\ul{u}\CAR\ul{R}$. By Definition \ref{unstablemanifold}, $V^u(\ul{R})=V^u(\ul{u})\cap W(R_0)$. $V^u(\ul{u})$ is an open submanifold (with a finite volume) of $M$ by definition, and $W(R_0)$ is a finite intersection of open subsets of $M$
. The claim follows.
\end{proof}
\begin{definition}
	Let $\ul{R}\in \Sig_L$, then $\Vol(V^u(\ul{R}))\in (0,\infty)$ denotes the volume of $V^u(\ul{R})$ w.r.t. its induced Riemannian leaf volume.
\end{definition}
\begin{cor}\label{secondform}
For all $\ul{R}\in\Sig_L$, $V^u(\ul{R})=\bigcap\{V^u(\ul{u}):\ul{u}\CAR\ul{R}\}$.
\end{cor}
\begin{proof}
($\supseteq$): Recall, $V^u(\ul{R})=\bigcap\{\psi_{x_0}[R_{p^u_0}(0)]:Z(u)\supseteq R_0,u=\psi_{x_0}^{p^s_0,p^u_0}\}\cap V^u(\ul{u})$ for any $\ul{u}$ s.t. $\ul{u}\CAR\ul{R}$. Fix some $\ul{u}'$ s.t. $\ul{u}'\CAR\ul{R}$, and let $z\in \bigcap\{V^u(\ul{u}):\ul{u}\CAR\ul{R}\}$. For each $u$ s.t. $Z(u)\supseteq R_0$, use Lemma \ref{nextstepcover} in succession to extend $u$ to a chain $\ul{u}$ s.t. $\ul{u}\CAR\ul{R}$ and $u_0=u$; Then, since $V^u(\ul{u})\subseteq\psi_{x_0}[R_{p_0^u}(0)]$ for $u_0=\psi_{x_0}^{p^s_0,p^u_0}$, we get $z\in \bigcap\{\psi_{x_0}[R_{p^u_0}(0)]:Z(u)\supseteq R_0,u=\psi_{x_0}^{p^s_0,p^u_0}\}$. In addition, $z$ is in $V^u(\ul{u}')$, whence $z\in\bigcap\{\psi_{x_0}[R_{p^u_0}(0)]:Z(u)\supseteq R_0,u=\psi_{x_0}^{p^s_0,p^u_0}\}\cap V^u(\ul{u}')=V^u(\ul{R})$.

\medskip
\noindent($\subseteq$): If $x\in V^u(\ul{R})$ and $\ul{u}\CAR\ul{R}$, then $x\in V^u(\ul{u})\cap W(R_0)$ by definition. Hence $x\in V^u(\ul{u})$.
\end{proof}

\begin{cor}\label{containment}
For all $\ul{R}\in\Sig_L$, $f[V^u(\sigma_R\ul{R})]\supseteq V^u(\ul{R})$.
\end{cor}
\begin{proof}
By Corollary \ref{secondform}, $V^u(\ul{R})=\bigcap\{V^u(\ul{u}):\ul{u}\CAR\ul{R}\}$. Therefore
\begin{align}
    f^{-1}[V^u(\ul{R})]= & \bigcap\{f^{-1}[V^u(\ul{u})]:\ul{u}\CAR\ul{R}\}\\
    \subseteq & \bigcap\{V^u(\sigma_R\ul{u}):\ul{u}\CAR\ul{R}\}\\
    = & \bigcap\{V^u(\ul{v}):\ul{v}\CAR\sigma_R\ul{R}\}=V^u(\sigma_R\ul{R}).
\end{align}

\noindent The equality in line (3) is given by two sided inclusions: $\subseteq$ is due to Lemma \ref{nextstepcover}: $\forall \ul{v}\in \Sigma_L$ s.t. $\ul{v}\CAR\sigma_R\ul{R}$, there exists $\ul{u}\in \Sigma_L$ s.t. $\ul{u}\CAR\ul{R}$ and $\sigma_R\ul{u}=\ul{v}$; thus the intersection in line (2) is over a bigger collection of sets, and thus smaller. The inclusion $\supseteq$ is straightforward.
\end{proof}

\medskip
Corollary \ref{Riemannianvolume} and Corollary \ref{containment} allow us to adapt a construction by Sinai
(\cite{Si1}), and build a family of absolutely continuous measures with good transformation properties. We will call these measures the {\em natural measures}. They serve as the first step in the construction of the leaf measures which we later integrate into a hyperbolic SRB measure.

\begin{theorem}\label{naturalmeasures} There exists a family of {\em natural measures} $\{m_{V^u(\ul{R})}\}_{\ul{R}\in\Sig_L}$ s.t. $\forall \ul{R}\in \Sig_L$ $m_{V^u(\ul{R})}$ is a measure on $V^u(\ul{R})$, and
$m_{V^u(\sigma \ul{R})}\circ f^{-1}|_{V^u(\ul{R})}=m_{V^u(\ul{R})}\cdot e^{\phi(\ul{R})}$
, where $\phi(\ul{R}):=\log\Big(\lim\limits_{n\rightarrow\infty}\frac{\Vol (f^{-n-1}[V^u(\ul{R})])}{\Vol (f^{-n}[V^u(\sigma_R\ul{R})])}\Big)$, $\phi:\Sig_L\rightarrow (-\infty,0]$. In addition, $m_{V^u(\ul{R})}\sim \lambda_{V^u(\ul{R})}$, $\frac{dm_{V^u(\ul{R})}}{d\lambda_{V^u(\ul{R})}}=e^{\pm\epsilon}$, where $\lambda_{V^u(\ul{R})}$ is the normalized induced Riemannian leaf volume on $V^u(\ul R)$.
\end{theorem}
\begin{proof}
Fix $\ul{R}\in \Sig_L$. By Corollary \ref{containment}, $\forall n\geq0$, $f^{-n}[V^u(\ul{R})]\subseteq V^u(\sigma_R^n(\ul{R}))$. Denote by $\lambda_n$ the normalized Riemannian leaf volume on $f^{-n}[V^u(\ul{R})]$. Define $\mu_n=\lambda_n\circ f^{-n}$. This is an absolutely continuous probability measure on $V^u(\ul{R})$. Let $\rho_n ^{\ul{R}}(y):=\frac{d\mu_n}{d\lambda_{V^u(\ul{R})}}(y)=\mathrm{Jac}(d_y f^{-n}|_{T_y V^u(\ul{R}) })\cdot\frac{\Vol(V^u(\ul{R}))}{\mathrm{Vol}\Big(f^{-n}[V^u(\ul{R})]\Big)}$. Define $\overline{m}$ as the Riemannian leaf volume on $V^u(\ul{R})$ (not normalized). Then $$\rho_n ^{\ul{R}}(y)=\mathrm{Jac}(d_y f^{-n}|_{T_y V^u(\ul{R}) })\cdot\frac{1}{\int_{V^u(\ul{R})}\mathrm{Jac}(d_z f^{-n}|_{T_z V^u(\ul{R}) })d\overline{m}(z)}\cdot\Vol(V^u(\ul{R})).$$

Define $g_n^{\ul{R}}(y):=\int_{V^u(\ul{R})}\frac{\mathrm{Jac}(d_z f^{-n}|_{T_z V^u(\ul{R})})}  {\mathrm{Jac}(d_y f^{-n}|_{T_y V^u(\ul{R})})}d\overline{m}(z)$. Then
\begin{align}\label{JakeAmir2}g_n^{\ul{R}}(y)=&\int_{V^u(\ul{R})}\prod_{j=0}^{n-1}\frac{\mathrm{Jac}\Big(d_{f^{-j}(z)}f^{-1}|_{T_{f^{-j}(z)}f^{-j}[V^u(\ul{R})]}\Big)}{\mathrm{Jac}\Big(d_{f^{-j}(y)}f^{-1}|_{T_{f^{-j}(y)}f^{-j}[V^u(\ul{R})]}\Big)}d\overline{m}(z)\\
=&\int_{V^u(\ul{R})}\exp\Big[\sum_{j=0}^{n-1}\log\mathrm{Jac}\Big(d_{f^{-j}(z)}f^{-1}|_{T_{f^{-j}(z)}f^{-j}[V^u(\ul{R})]}\Big)-\log\mathrm{Jac}\Big(d_{f^{-j}(y)}f^{-1}|_{T_{f^{-j}(y)}f^{-j}[V^u(\ul{R})]}\Big)\Big]d\overline{m}(z).\nonumber\end{align}
Let $z\in V^u(\ul{R})$. Thus by \cite[Proposition~4.4(3)]{SBO}, for any $n\geq0$, and $\omega(f^{-n}(y)),\omega(f^{-n}(z))$ any normalized volume forms on $T_{f^{-n}(y)}f^{-n}[V^u(\ul{R})],T_{f^{-n}(z)}f^{-n}[V^u(\ul{R})]$ respectively,
\begin{equation}\label{forholderphi}\forall m\geq0,\text{ } \left|\log|d_{f^{-n}(y)}f^{-m}\omega(f^{-n}(y))|-\log|d_{f^{-n}(z)}f^{-m}\omega(f^{-n}(z))|\right|\leq \epsilon d(f^{-n}(y),f^{-n}(z))^ \frac{\beta}{4}.\end{equation}
Let any $\ul{u}\in \Sigma_L, u_i=\psi_{x_i}^{p^s_i,p^u_i},i\leq0$ s.t. $\ul{u}\CAR \ul{R}$, then by the strong bound shown in the proof of \cite[Proposition~6.3(1)]{Sarig}, $\epsilon d(f^{-n}(y),f^{-n}(z))^\frac{\beta}{4}\leq \epsilon(6p_0^ue^{\frac{-\chi n}{2}})^\frac{\beta}{4}\leq \epsilon^\frac{3}{4}(p_0^u)^\frac{\beta}{4}e^{\frac{-\beta\chi n}{8}}$.\footnote{\cite{Sarig} applies for the case $d=2$, while \cite{SBO} extends its results to $d\geq2$. However, \cite{SBO}[Proposition~4.4(1)] refers to \cite{Sarig}[Proposition~6.3(1)] for proof, as the proofs for these claims are almost identical. Therefore we refer the reader to the proof in the a-priori 2 dimensional case for the required bound.} Hence

\begin{align}\label{DeltaFlight}
\Big|&\sum_{j=0}^{n+m-1}\log\mathrm{Jac}\Big(d_{f^{-j}(z)}f^{-1}|_{T_{f^{-j}(z)}f^{-j}[V^u(\ul{R})]}\Big)-\log\mathrm{Jac}\Big(d_{f^{-j}(y)}f^{-1}|_{T_{f^{-j}(y)}f^{-j}[V^u(\ul{R})]}\Big)\\
-&\Big[\sum_{j=0}^{n-1}\log\mathrm{Jac}\Big(d_{f^{-j}(z)}f^{-1}|_{T_{f^{-j}(z)}f^{-j}[V^u(\ul{R})]}\Big)-\log\mathrm{Jac}\Big(d_{f^{-j}(y)}f^{-1}|_{T_{f^{-j}(y)}f^{-j}[V^u(\ul{R})]}\Big)\Big]\Big|\nonumber\\
=&\Big|\sum_{j=n}^{n+m-1}\log\mathrm{Jac}\Big(d_{f^{-j}(z)}f^{-1}|_{T_{f^{-j}(z)}f^{-j}[V^u(\ul{R})]}\Big)-\log\mathrm{Jac}\Big(d_{f^{-j}(y)}f^{-1}|_{T_{f^{-j}(y)}f^{-j}[V^u(\ul{R})]}\Big)\Big|\nonumber\\
=&\Big|\log\mathrm{Jac}\Big(d_{f^{-n}(z)}f^{-m}|_{T_{f^{-n}(z)}f^{-n}[V^u(\ul{R})]}\Big)-\log\mathrm{Jac}\Big(d_{f^{-n}(y)}f^{-m}|_{T_{f^{-n}(y)}f^{-n}[V^u(\ul{R})]}\Big)\Big|\leq\epsilon^\frac{3}{4}(p_0^u)^\frac{\beta}{4}e^{-\frac{\chi\beta}{8}n}.\nonumber
\end{align}
This bound is uniform in $y,z\in V^u(\ul{R})$, thus $\sum_{j=0}^n\log\mathrm{Jac}\Big(d_{f^{-j}(z)}f^{-1}|_{T_{f^{-j}(z)}f^{-j}[V^u(\ul{R})]}\Big)-$

$\log\mathrm{Jac}\Big(d_{f^{-j}(y)}f^{-1}|_{T_{f^{-j}(y)}f^{-j}[V^u(\ul{R})]}\Big)$ converges uniformly as $n\rightarrow\infty$ to a finite limit. Therefore, the sequence of its exponents, $\prod_{j=0}^{n-1}\frac{\mathrm{Jac}\Big(d_{f^{-j}(z)}f^{-1}|_{T_{f^{-j}(z)}f^{-j}[V^u(\ul{R})]}\Big)}{\mathrm{Jac}\Big(d_{f^{-j}(y)}f^{-1}|_{T_{f^{-j}(y)}f^{-j}[V^u(\ul{R})]}\Big)}$, converges uniformly to a finite non zero limit. A uniform limit commutes with integral, so $g_n^{\ul{R}}$ also converge for every $y$, and thus also $\rho_n$. In fact, since the bound for a fixed $n$ is uniform in $y,z\in V^u(\ul{R})$, then $g_n^{\ul{R}}\xrightarrow[]{\text{uniformly}}g^{\ul{R}} $, to a continuous limit (since $g_n^{\ul{R}}$ are continuous),
\begin{equation}\label{JakeAmir}g^{\ul{R}}(y):=\lim_{n\rightarrow\infty}g_n^{\ul{R}}(y)=\int_{V^u(\ul{R})}\lim_{n\rightarrow\infty} \prod_{j=0}^{n-1}\frac{\mathrm{Jac}\Big(d_{f^{-j}(z)}f^{-1}|_{T_{f^{-j}(z)}f^{-j}[V^u(\ul{R})]}\Big)}{\mathrm{Jac}\Big(d_{f^{-j}(y)}f^{-1}|_{T_{f^{-j}(y)}f^{-j}[V^u(\ul{R})]}\Big)} d\overline{m}(z).\end{equation}
$\rho^{\ul{R}} =\frac{\Vol(V^u(\ul{R}))}{g^{\ul{R}}}$, and by equation \eqref{forholderphi} (with $n=0$) and the remark after it, $g^{\ul{R}} =\Vol(V^u(\ul{R}))e^{\pm\sqrt{\epsilon}(p_0^u)^\frac{\beta}{4}}= \Vol(V^u(\ul{R})) e^{\pm\epsilon}$, then 
$\rho^{\ul{R}} =e^{\pm\epsilon}$. Therefore, $\rho_n^{\ul{R}}\xrightarrow[]{\text{uniformly}}\rho^{\ul{R}} $, a continuous density. Define the measure on $V^u(\ul{R})$: $$m_{V^u(\ul{R})}(A)=\int_A \rho^{\ul{R}} d\lambda_{V^u(\ul{R})}.$$

%

\noindent\underline{Notice:} 
Henceforth we write $\Jac^u(d_xf^k)\equiv \Jac(d_xf^k|_{T_xV^u(\ul{S})})$ for $x\in V^s(\ul{S})$.

\medskip
\noindent\underline{Claim:} If $A\subseteq f^{-1}[V^u(\ul{R})]\subseteq V^u(\sigma_R\ul{R})$ is measurable, then  $$m_{V^u(\sigma_R\ul{R})}(A)=\Big(m_{V^u(\ul{R})}\circ f\Big)(A)\cdot\Big(m_{V^u(\sigma_R\ul{R})}\circ f^{-1}\Big)(V^u(\ul{R})).$$

\medskip
\noindent\underline{Proof:} Without loss of generality $\lambda_{V^u(\ul{R})}(f[A])>0$. By definition, for $\tilde{A}:=f[A]$:

\begin{align*}m_{V^u(\sigma_R\ul{R})}\circ f^{-1}(\tilde{A})=&\int_{f^{-1}[\tilde{A}]}\lim\limits_{n\rightarrow\infty}\frac{\Jac^u(d_zf^{-n})}{\Vol (f^{-n}[V^u(\sigma_R\ul{R})])}\cdot\Vol(V^u(\sigma_R\ul{R}))\frac{d\overline{m}_{V^u(\sigma_R\ul{R})}(z)}{\Vol(V^u(\sigma_R\ul{R}))}\\
=&\int_{\tilde{A}}\lim\limits_{n\rightarrow\infty}\frac{\Jac^u(d_{f^{-1}(z)}f^{-n})\Jac^u(d_zf^{-1})}{\Vol (f^{-n-1}[V^u(\ul{R})])}d\overline{m}_{V^u(\ul{R})}(z)
\cdot\lim\limits_{n\rightarrow\infty}\frac{\Vol (f^{-n-1}[V^u(\ul{R})])}{\Vol (f^{-n}[V^u(\sigma_R\ul{R})])}\\
=&m_{V^u(\ul{R})}(\tilde{A})\cdot\lim\limits_{n\rightarrow\infty}\frac{\Vol (f^{-n-1}[V^u(\ul{R})])}{\Vol (f^{-n}[V^u(\sigma_R\ul{R})])},\end{align*}
where $\overline{m}_{V^u(\sigma_R\ul{R})},\overline{m}_{V^u(\ul{R})}$ are the (not normalized) Riemannian volumes of the respective leaves; The reason we can separate the second equation into the product of two limits is that the first limit exists and is finite and non zero, and since the limit on the previous line exists and is finite.
Plugging in $\tilde{A}=V^u(\ul{R})$ yields $m_{V^u(\sigma_R\ul{R})}(f^{-1}[V^u(\ul{R})])=\lim\limits_{n\rightarrow\infty}\frac{\Vol (f^{-n-1}[V^u(\ul{R})])}{\Vol (f^{-n}[V^u(\sigma_R\ul{R})])}$. Substituting $\tilde{A}=f[A]$ gives the requested result.
\end{proof}

\noindent\textbf{Remark:} For the connection between $\phi$ and ``the geometric potential" see \textsection \ref{lalipop}.

\begin{definition}
For every $\ul{R}\in\Sig$, define the absolutely continuous probability measure $m_{V^u(\ul{R})}$ on $V^u(\ul{R})$ by Theorem \ref{naturalmeasures}. We call $\{m_{V^u(\ul{R})}\}_{\ul{R}\in\Sig_L}$ the {\em family of natural measures}. We also let

 $$\forall \ul{R}\in\Sig_L, m_{\ul{R}}:=\mathbb{1}_{W^u(\ul{R})}\cdot m_{V^u(\ul{R})}\text{, where }W^u(\ul{R}):=\bigcap_{j=0}^\infty f^{j}[R_{-j}].$$
\end{definition}
($W^u(\ul{R})$ may be empty unless $\ul{R}\in\Sig_L^\circ$, and then $m_{\ul{R}}$ may be the zero measure).

In what follows, we will mold the structure of a TMS onto the family of natural measures. This will allow us to apply results in thermodynamic formalism to construct invariant measures with absolutely continuous conditional measures.

\subsection{Analytic Properties of the Symbolic Construction and $\phi$}\label{analyticpropsofsymb}\text{ }
In this section we establish the continuity properties of smooth leaf measures and the potential $\phi$ from Theorem \ref{naturalmeasures}. These properties are needed for the study of thermodynamic properties on $\Sig_L$ later on (see \textsection \ref{lalipop}).

\begin{claim}\label{phibar} Let $\phi:\Sig_L\rightarrow(-\infty,0]$ be as in Theorem \ref{naturalmeasures}, then $\forall \ul{R}\in\Sig_L, m_{\ul{R}}\circ f^{-1}=\sum_{\sigma_R\ul{S}=\ul{R}}m_{\ul{S}}\cdot e^{\phi(\ul{S})}$.
\end{claim}
\begin{proof} Corollary \ref{decomposition} says that  $f[W^u(\ul{R})]=\bigcupdot_{\sigma_R\ul{S}=\ul{R}}W^u(\ul{S})$. For any measurable $A\subseteq f[V^u(\ul{R})]$,
\begin{align*}
    m_{\ul{R}}\circ f^{-1}(A)=&m_{V^u(\ul{R})}(W^u(\ul{R})\cap f^{-1}[A])=m_{V^u(\ul{R})}\left(f^{-1}\left[\bigcupdot_{\sigma_R\ul{S}=\ul{R}}W^u(\ul{S})\cap A\right]\right)\\
    =&\sum_{\sigma_R\ul{S}=\ul{R}}m_{V^u(\ul{R})}(f^{-1}[W^u(\ul{S})\cap A])=\sum_{\sigma_R\ul{S}=\ul{R}}e^{\phi(\ul{S})}m_{V^u(\ul{S})}(W^u(\ul{S})\cap A)=\sum_{\sigma_R\ul{S}=\ul{R}}e^{\phi(\ul{S})}m_{\ul{S}}(A),
\end{align*}
because by Theorem \ref{naturalmeasures}, $m_{V^u(\ul{R})}\circ
f^{-1}|_{V^u(\ul{S})}=e^{\phi(\ul{S})}\cdot m_{V^u(\ul{S})}
$ for any $\ul{S}$ s.t. $\sigma_R\ul{S}=\ul{R}$.
\end{proof}


\begin{definition}
	$M_f:=\max_{x\in M}\{\|d_xf\|,\|d_x(f^{-1})\|\}$. This is finite since $d_\cdot f,d_\cdot f^{-1}$ are H\"older continuous and $M$ is compact.
\end{definition}
\begin{definition}\label{localisometries}
For every $x\in M$ there is an open neighborhood $D$ of diameter less than $\rho$ and a smooth map $\Theta_D:TD\rightarrow\mathbb{R}^d$ s.t. :
\begin{enumerate}
    \item $\Theta_D:T_xM\rightarrow\mathbb{R}^d$ is a linear isometry for every $x\in D$.

    \item Define $\nu_x:=\Theta_D|_{T_xM}^{-1}:\mathbb{R}^d\rightarrow T_xM$, then $(x,u)\mapsto(\exp_x\circ\nu_x)(u)$ is smooth and Lipschitz on $D\times B_2(0)$ w.r.t. the metric $d(x,x')+|u-u'|$, where $|\cdot|$ is the $\ell^\infty$ norm on $\mathbb{R}^d$.

    \item $x\mapsto\nu_x^{-1}\circ \exp_x^{-1}$ is a Lipschitz map from $D$ to $C^2(D,\mathbb{R}^d)=\{C^2\text{-maps from }D\text{ to }\mathbb{R}^d\}$. Let $\mathcal{D}$ be a finite cover of $M$ by such neighborhoods. Denote with $\varpi(\mathcal{D})$ the Lebesgue number of that cover: If $d(x,y)<\varpi(\mathcal{D})$ then $x$ and $y$ belong to the same $D$ for some $D$.
    \item $\exists E_0=E_0(M)$ s.t. $\forall D\in\mathcal{D}$, $(x_1,u_1,v_1)\mapsto (\Theta_D\circ d_{v_1}\exp_{x_1})(\nu_{x_1}u_1)$ is $E_0$-Lipschitz on $D\times B_2(0)\times B_2(0)$, w.r.t. the metric $d(x_1,x_2)+|u_1-u_2|+|v_1-v_2|$.
    \item $\exists H_0=H_0(M,f)$ s.t. $\forall D,D_+,D_-\in\mathcal{D}$, $\forall x,y,z\in D$, s.t. $f(x),f(y)\in D^+, f^{-1}(x),f^{-1}(z)\in D^- $, $\|\Theta_{D^-}d_x(f^{-1})\nu_x-\Theta_{D^-}d_y(f^{-1})\nu_y\|,\|\Theta_{D^+}d_xf\nu_x-\Theta_{D^+}d_yf\nu_y\|\leq H_0\cdot d(x,y)^\beta$.
    \item We assume w.l.o.g. that $\epsilon>0$ is small enough, so $\sup\limits_{|v|<\epsilon}\{\|d_v\exp_x-\mathrm{Id}_{T_xM}\|,\|d_v\exp^{-1}_x-\mathrm{Id}_{T_xM}\|\}\leq\frac{1}{2}$, $\forall x\in M$.
\end{enumerate}
\end{definition}

The following two definitions are versions of similar definitions in \cite{Sarig,SBO,KM}
. We use the notation ``$u/s$" to define both $u$-manifolds and $s$-manifolds, without having to write everything twice.
\begin{definition}\label{def135} Let $x\in \HWT_\chi$, a {\em $u$-manifold} in $\psi_x$ is a manifold $V^u\subset M$ of the form
$$V^u=\psi_x[\{(F_1^u(t_{s(x)+1},...,t_d),...,F_{s(x)}^u(t_{s(x)+1},...,t_d),t_{s(x)+1},...t_d) : |t_i|\leq q\}],$$
where $0<q\leq Q_\epsilon(x)$, and $\overrightarrow{F}^u$ is a $C^{1+\beta/3}$ function s.t. $\max\limits_{\overline{R_q(0)}}|\overrightarrow{F}^u|_\infty\leq Q_\epsilon(x)$.

Similarly we define an {\em $s$-manifold} in $\psi_x$:
$$V^s=\psi_x[\{(t_1,...,t_{s(x)},F_{s(x)+1}^s(t_1,...,t_{s(x)}),...,F_d^s(t_1,...,t_{s(x)})): |t_i|\leq q\}],$$
with the same requirements for $\overrightarrow{F}^s$ and $q$. 
The function $\overrightarrow{F}=\overrightarrow{F}^{u/s}$ is called the {\em representing function} of $V^{u/s}$ at $\psi_x$. The parameters of a $u/s$ manifold in $\psi_x$ are:
\begin{itemize}
\item $\alpha$-parameter: $\alpha(V^{u/s}):=\|d_{\cdot}\overrightarrow{F}\|_{\beta/3}:=\max\limits_{\overline{R_q(0)}}\|d_{\cdot}\overrightarrow{F}\|+\text{H\"ol}_{\beta/3}(d_{\cdot}\overrightarrow{F})$,

where $\text{H\"ol}_{\beta/3}(d_{\cdot}\overrightarrow{F}):=\max\limits_{\vec{t_1},\vec{t_2}\in\overline{R_q(0)}}\{\frac{\|d_{\overrightarrow{t_1}}\overrightarrow{F}-d_{\overrightarrow{t_2}}\overrightarrow{F}\|}{|\overrightarrow{t_1}-\overrightarrow{t_2}|^{\beta/3}}\}$ and $\|A\|:=\sup\limits_{v\neq0}\frac{|Av|_\infty}{|v|_\infty}$.
\item $\gamma$-parameter: $\gamma(V^{u/s}):=\|d_0\overrightarrow{F}\|$.
\item $\varphi$-parameter: $\varphi(V^{u/s}):=|\overrightarrow{F}(0)|_\infty$.
\item $q$-parameter: $q(V^{u/s}):=q$.
\end{itemize}

A {\em $(u/s, \alpha,\gamma,\varphi,q)$-manifold} in $\psi_x$ is a $u/s$-manifold $V^{u/s}$ in $\psi_x$ whose parameters satisfy $\alpha(V^{u/s})\leq \alpha,\gamma(V^{u/s})\leq\gamma,\varphi(V^{u/s})\leq\varphi,q(V^{u/s})\leq q$.
\end{definition}
Notice that the dimensions of an $s$ or a $u$ manifold in $\psi_x$ depend on $x$. Their sum is $d$. Recall the definition of $\HWT_\chi$ (Definition \ref{NUHsharp}).
\begin{definition}\label{admissible} Suppose $x\in \HWT_\chi$ and $0<p^s,p^u\leq Q_\epsilon(x)$ (i.e. $\psi_x^{p^s,p^u}$ is a double Pesin-chart)
. A $u/s$-admissible manifold in $\psi_x^{p^s,p^u}$ is a {\em $(u/s, \alpha,\gamma,\varphi,q)$-manifold} in $\psi_x$ s.t.
$$\alpha\leq\frac{1}{2},\gamma\leq\frac{1}{2}(p^u\wedge p^s)^{\beta/3},\varphi\leq10^{-3}(p^u\wedge p^s),\text{ and }q = \begin{cases} p^u & u\text{-manifolds} \\
p^s & s\text{-manifolds} \end{cases}.$$
\end{definition}

\noindent\textbf{Remark:}  By Claim \ref{chikfila}, $\forall \ul{u}\in \Sigma$ there exist a maximal dimension stable leaf $V^s(\ul{u})= V^s((u_i)_{i\geq0})$ and a maximal dimension unstable leaf $V^u(\ul{u})= V^u((u_i)_{i\leq0})$
. The construction in \cite[Proposition~3.12]{SBO} (and \cite[Proposition~4.15]{Sarig} when $d=2$) in fact tells us that $V^s(\ul{u})$ and $V^u(\ul{u})$ are admissible stable and unstable manifolds in $u_0=\psi_{x_0}^{p^s_0,p^u_0}$, respectively.

\begin{definition}
	Let $V^u$ and $W^u$ be $u$-admissible manifolds in $\psi_x^{p^s,p^u}$, and let $F^u, G^u$ be their representing functions, respectively. Then,
	\begin{align*}
	d_{C^0}(V^u,W^u):= \max_{R_{p^u}(0)}|F^u-G^u|,\text{ }d_{C^1}(V^u,W^u):= \max_{R_{p^u}(0)}|F^u-G^u|+\max_{R_{p^u}(0)}\|d_\cdot F^u-d_\cdot G^u\|,
	\end{align*}
	where $R_{p^u}(0):=\{u\in\mathbb{R}^d:|u|_\infty\leq p^u\}$.
\end{definition}
\begin{definition}
	Let $\ul{R},\ul{S}\in \Sig_L$ be two chains s.t. $R_0=S_0$, then
	\begin{align*}
	d_{C^0}(V^u(\ul{R}), V^u(\ul{S})):=&\max_{\ul{u}\CAR\ul{R},\ul{v}\CAR\ul{S},u_0=v_0}\left\{d_{C^0}(V^u(\ul{u}),V^u(\ul{v})\right\},\\
	d_{C^1}(V^u(\ul{R}), V^u(\ul{S})):=&\max_{\ul{u}\CAR\ul{R},\ul{v}\CAR\ul{S},u_0=v_0}\left\{d_{C^1}(V^u(\ul{u}),V^u(\ul{v}))\right\}.
	\end{align*}
\end{definition}
\noindent The maximum is well defined and finite since $\forall R\in\mathcal{R}$, $\#\{u\in\mathcal{V}: Z(u)\supseteq R\}<\infty$, and $\forall \ul{u},\ul{v}\in \Sigma_L$ s.t. $\ul{u},\ul{v}\CAR\ul{R}$ and $u_0=v_0=\psi_{x}^{p^s,p^u}$, $V^u(\ul{u})=V^u(\ul{v})$, by \cite[Proposition~4.15]{SBO}.
\begin{claim}\label{VRlikeVu} $\exists K>0,\theta\in(0,1)$ s.t. $d(\ul{R},\ul{S})\leq e^{-n}\Rightarrow d_{C^1}(V^u(\ul{R}),V^u(\ul{S}))\leq K\cdot\theta^n$, $\forall n\geq1$.
\end{claim}
\begin{proof}
Let $u\in\mathcal{V}$ s.t. $Z(u)\supseteq R_0=S_0$. Use Lemma \ref{nextstepcover} repeatedly $n$ times to obtain an admissible sequence $(u_{-n+1},...,u_{-1},u)$ s.t. $Z(u_{-i})\supseteq R_{-i}=S_{-i}$, $\forall 0\leq i\leq n-1$. Continue using Lemma \ref{nextstepcover} to obtain two admissible chains $\ul{w}_{\ul{R}}(u), \ul{w}_{\ul{S}}(u)\in\Sigma_L\cap [u]$ s.t. $\ul{w}_{\ul{R}}(u)\CAR\ul{R}$, $\ul{w}_{\ul{S}}(u)\CAR\ul{S}$, and $d(\ul{w}_{\ul{R}}(u), \ul{w}_{\ul{S}}(u))\leq e^{-n}$. By \cite[Proposition~4.15]{SBO}, $\forall \ul{w}'\in \Sigma_L\cap [u]$ s.t. $\ul{w}'\CAR\ul{R}$, $V^u(\ul{w}_{\ul{R}}(u))=V^u(\ul{w}')$, and similarly with $\ul{w}_{\ul{S}}(u)$. By \cite[Proposition~3.12]{SBO}, $\exists K>0,\theta\in(0,1)$ which depend on $\beta$ and $\chi$ s.t. $d_{C^1}(V^u(\ul{w}_{\ul{R}}(u)),V^u(\ul{w}_{\ul{S}}(u)))\leq K\cdot\theta^n$. Therefore,
\begin{align*}
	d_{C^0}(V^u(\ul{R}),V^u(\ul{S}))\leq d_{C^1}(V^u(\ul{R}),V^u(\ul{S})) =\max_{u\in\mathcal{V}:Z(u)\supseteq R_0}\left\{d_{C^1}\left(V^u(\ul{w}_{\ul{R}}(u)), V^u(\ul{w}_{\ul{S}}(u))\right)\right\}\leq K\cdot \theta^n.
\end{align*}
\end{proof}

The following lemma is the main technical step in the study of the regularity of the function $\phi$ (see Theorem \ref{naturalmeasures}). Recall that  $\phi$ is defined in terms of densities on unstable leaves. The difficulty is that when we iterate backwards, unstable leaves of different points may drift apart with a large exponential rate. Therefore the estimates are required to be very careful. The main idea we use is the fact that a weak exponential rate can be made stronger by considering the accumulative rate over several iterates.
\begin{lemma}\label{0.1}
Let $\ul{R},\ul{S}\in\Sig_L$
. 
There exist $n_0\in \mathbb{N}$ and $\theta_2\in(0,1)$ which depend only on $\chi,\beta$ and $f$, and a diffeomorphism $I:V^u(\ul{R})\rightarrow V^u(\ul{S})$, s.t. if $d(\ul{R},\ul{S})=e^{-n}$, $n\geq n_0$, then for
$\rho^{\ul{R}}:=\frac{dm_{V^u(\ul{R})}}{d\lambda_{V^u(\ul{R})}},\rho^{\ul{S}}:=\frac{dm_{V^u(\ul{S})}}{d\lambda_{V^u(\ul{S})}}$,  
$\|\rho^{\ul{R}}-\rho^{\ul{S}}\circ I\|_\infty\leq \epsilon\theta_2^n$, where $\lambda_{V^u(\ul{R})},\lambda_{V^u(\ul{S})}$ are the normalized Riamannian volume on the respective leaf, and $\epsilon>0$ is as in Theorem \ref{epsilonika}.
\end{lemma}
\begin{proof} Let $\ul{u},\ul{v}\in\Sigma_L$ s.t. $\ul{u}\CAR\ul{R},\ul{v}\CAR\ul{S}$, $d(\ul{u},\ul{v})= e^{-n}\leq e^{-n_0}$, and $d_{C^1}(V^u(\ul{R}), V^u(\ul{S}))= d_{C^1}(V^u(\ul{u}), V^u(\ul{v})) $ (as in the proof of Claim \ref{VRlikeVu}). Denote by
$F^{\ul{v}}$ and $F^{\ul{u}}$ the representing functions of $V^u(\ul{v})$ and $V^u(\ul{u})$ respectively. Write $u=\dim V^u(\ul{R})=\dim V^u(\ul{S})$ (since $R_0=S_0$).

\medskip
\underline{Part 0:} Write $\forall i\geq0$, $u_{-i}=\psi_{x_{-i}}^{p_{-i}^s,p_{-i}^u},v_{-i}=\psi_{y_{-i}}^{q_{-i}^s,q_{-i}^u}$ and recall that $u_{-i}=v_{-i}$ for all $i\leq n$ by assumption. Hence, in particular, $\psi_{x_0}=\psi_{y_0}$. In addition $R_0=S_0$, thus $O:=\pi_u\circ\psi_{x_0}^{-1}[W(R_0)]$, where $\pi_u$ is the projection onto the last $u$ coordinates, is an open subset of $\mathbb{R}^d$ for which $V^u(\ul{R})=(\psi_{x_0}\circ\tilde{F}^{\ul{u}})[O],V^u(\ul{S})=(\psi_{x_0}\circ\tilde{F}^{\ul{v}})[O]$ where $\tilde{F}^{\ul{u}/\ul{v}}(t)=(F^{\ul{u}/\ul{v}}(t),t)$. $I$ is defined in the following way: $$I:\psi_{x_0}(F^{\ul{u}}(t),t)\mapsto\psi_{x_0}(F^{\ul{v}}(t),t).$$ Equivalently, $I=\psi_{x_0}\circ \tilde{F}^{\ul{v}}\circ \pi_u\circ \psi_{x_0}^{-1}$, and  $\tilde{F}^{\ul{u}/\ul{v}}$ are diffeomorphisms onto their images. So $I$ is a diffeomorphism from $V^u(\ul{R})$ to $V^u(\ul{S})$. $I$ satisfies
\begin{equation}\label{adidax}
d(I(x),x)\leq \|\psi_{x_0}\|d_{C^0}(F^{\ul{u}},F^{\ul{v}})\leq 2d_{C^0}(F^{\ul{u}},F^{\ul{v}}).	
\end{equation}

The identity map of $V^u(\ul{R})$ can be written as $\mathrm{Id}=\psi_{x_0}\circ \tilde{F}^{\ul{u}}\circ\pi_u\circ\psi_{x_0}^{-1}:V^u(\ul{R})\rightarrow V^u(\ul{R})$. Let $z\in V^u(\ul{R})$, and let $\Theta_{D}:TD\rightarrow \mathbb{R}^d$ be a local isometry as in Definition \ref{localisometries} s.t. $D$ is a neighborhood which contains $z,I(z)$. Then,
\begin{align}\label{ams}
\|\Theta_D d_zI-\Theta_D d_z\mathrm{Id}\|&=\|\Theta_Dd_z(\psi_{x_0}\circ \tilde{F}^{\ul{v}}\circ\pi_u\circ\psi_{x_0}^{-1})-\Theta_Dd_z(\psi_{x_0}\circ \tilde{F}^{\ul{u}}\circ\pi_u\circ\psi_{x_0}^{-1}) \nonumber\|\\
&= \|[\Theta_Dd_t(\psi_{x_0}\circ\tilde{F}^{\ul{v}})-\Theta_Dd_t(\psi_{x_0}\tilde{F}^{\ul{u}})]\pi_u\circ d_z(\psi_{x_0}^{-1}))\|, \hspace{0.25cm}\text{where } t=\pi_u\psi_{x_0}^{-1}(z),
\nonumber\\
&\leq 2\|C_\chi^{-1}(x_0)\|\cdot\|\Theta_Dd_t(\psi_{x_0}\circ\tilde{F}^{\ul{v}})-\Theta_Dd_t(\psi_{x_0}\circ\tilde{F}^{\ul{u}})\|\nonumber\\
&=2\|C_\chi^{-1}(x_0)\|\cdot\|\Theta_Dd_{\tilde{F}^{\ul{v}}(t)}\psi_{x_0}d_t\tilde{F}^{\ul{v}}-\Theta_Dd_{\tilde{F}^{\ul{u}}(t)}\psi_{x_0}d_t\tilde{F}^{\ul{u}}\|\nonumber\\
&=2\|C_\chi^{-1}(x_0)\|\cdot\|\Theta_Dd_{\tilde{F}^{\ul{v}}(t)}\psi_{x_0}(d_t\tilde{F}^{\ul{v}}-d_t\tilde{F}^{\ul{u}})+(\Theta_Dd_{\tilde{F}^{\ul{v}}(t)}\psi_{x_0}-\Theta_Dd_{\tilde{F}^{\ul{u}}(t)}\psi_{x_0})d_t\tilde{F}^{\ul{u}}\|\nonumber\\
&\leq2\|C_\chi^{-1}(x_0)\|\cdot\left[d_{C^1}(\tilde{F}^{\ul{u}},\tilde{F}^{\ul{v}})\|d_\cdot\psi_{x_0}\|+\mathrm{Lip}(d_\cdot\psi_{x_0})d_{C^0}(\tilde{F}^{\ul{u}},\tilde{F}^{\ul{v}})\cdot\|d_\cdot\tilde{F}^{\ul{u}}\|\right]\nonumber\\
&\leq 8\|C_\chi^{-1}(x_0)\|\cdot d_{C^1}(V^u(\ul{R}),V^u(\ul{S})),
\end{align}
It follows that,
\begin{equation}\label{adidax2}
	|\Jac(d_zI)-1|=|\Jac(d_zI)-\Jac(d_z\mathrm{Id})|\leq C_1 \|\Theta_D d_zI-\Theta_D d_z\mathrm{Id}\|\leq 8C_1\|C_\chi^{-1}(x_0)\|d_{C^1}(V^u(\ul{R}),V^u(\ul{S})),
\end{equation}
where $C_1$ is the Lipschitz constant for the absolute value of the determinant on the ball of $(d\times d)$-matrices with a bounded operator norm of $2M_f$, where $\|d_zI\|\leq 2M_f$ by \eqref{ams}.\footnote{Write $u=dim V^u(\ul{R})$, and let $A, B$ be $u\times u$ matrices s.t. $|a_{ij}-b_{ij}|\leq \delta$ for all $i,j\leq u$. So $\det A=\sum_{\sigma\in S_{u}}\mathrm{sgn}(\sigma)\prod_{i=1}^{u}a_{i\sigma(i)}$. Then $|\det A-\det B|\leq \sum_{\sigma\in S_{u}}|\prod_{i=1}^{u}a_{i\sigma(i)}-\prod_{i=1}^{u}b_{i\sigma(i)}|\leq |S_{u}|\cdot u\|B\|_{Fr}^{u-1}\delta$. Take maximum over $u\leq d-1$.} 
By \cite[Proposition~3.12(5)]{SBO}, $d_{C^1}(V^u(\ul{u}),V^u(\ul{v}))\leq K\theta^n$; and in addition (see the remark after \cite[Definition~3.2]{SBO}), $d_{C^1}(V^u(\ul{u}),V^u(\ul{v}))\leq 3(p_0^u)^\frac{\beta}{3}$. Therefore,
\begin{align}\label{DeltaAir2}
d_{C^1}(V^u(\ul{R}),V^u(\ul{S}))=&d_{C^1}(V^u(\ul{u}),V^u(\ul{v}))\nonumber\\
=& d_{C^1}(V^u(\ul{u}),V^u(\ul{v}))^\frac{5}{6}\cdot d_{C^1}(V^u(\ul{u}),V^u(\ul{v}))^\frac{1}{6}
\leq 3(p_0^u)^\frac{5\beta}{18}\cdot K^\frac{1}{6}\cdot \theta^{\frac{1}{6}n}.
\end{align}
Similarly, by the definition of admissible manifolds $d_{C^0}(V^u(\ul{R}),V^u(\ul{S}))\leq p^u_0$, and so
\begin{align}\label{DeltaAir3.0}
d_{C^0}(V^u(\ul{R}),V^u(\ul{S}))=&d_{C^0}
= d_{C^0}(V^u(\ul{R}),V^u(\ul{S}))^\frac{5}{6}\cdot d_{C^0}(V^u(\ul{R}),V^u(\ul{S}))^\frac{1}{6}
\leq 3(p_0^u)^\frac{5}{6}\cdot K^\frac{1}{6}\cdot \theta^{\frac{1}{6}n}.
\end{align}

\noindent Substituting these estimates in equation \eqref{adidax2} gives (for sufficiently small $\epsilon>0$),
\begin{equation}\label{DeltaAir}|\Jac(d_zI)-1|\leq8C_1\|C_\chi^{-1}(x_0)\|\cdot2(p_0^u)^\frac{5\beta}{18}\tilde{K}\tilde{\theta}^n
 \leq \epsilon(p_0^u)^\frac{\beta}{4} \tilde{\theta}^n,\end{equation}
where $\tilde{\theta}:=\theta^\frac{\beta}{6}\in (0,1)$, and $\tilde{K}:=K^\frac{1}{6}>0$.

\medskip
\underline{Part 1:}
Given $\ul{R}\in \Sig_L$, define as in equation \eqref{JakeAmir2},  $$g_n^{\ul{R}}(y):=\int_{V^u(\ul{R})}\exp[\sum_{k=0}^{n-1}\log\Jac(d_{f^{-k}(z)}f^{-1}|_{T_{f^{-k}(z)}f^{-k}[V^u(\ul{R})]})-\log\Jac(d_{f^{-k}(y)}f^{-1}|_{T_{f^{-k}(y)}f^{-k}[V^u(\ul{R})]})]d\overline{m}_{\ul{R}}(z),$$
where $\overline{m}_{\ul{R}}$ is the induced Riemannian volume of $V^u(\ul{R})$ (not normalized).
We saw in Theorem \ref{naturalmeasures} that for a fixed $z\in V^u(\ul{R})$, the inner sum (call it $S_{n,z}^{\ul{R}}(y)$) converges uniformly to a limit denoted by $S_z^{\ul{R}}(y)$. Therefore, also uniformly in $y,z$, \footnote{If $a_n\xrightarrow[n\rightarrow\infty]{} a$, and $\exists c_n\downarrow0$ s.t. $\forall  m\geq n$, $|a_n-a_m|\leq c_n$, then $|a_n-a|\leq |a_n-a_m|+|a_m-a|$, $\forall m\geq0$. Since the second term tends to 0 as $m\rightarrow\infty$, we get $|a_n-a|\leq c_n$ for all $n$. We apply this to the bounds from equation \eqref{DeltaFlight}.}
\begin{equation}\label{BestWest3}|S_{n,z}^{\ul{R}}(y)-S_z^{\ul{R}}(y)|\leq \epsilon^\frac{3}{4}(p_0^u)^\frac{\beta}{4}e^{-\frac{\chi\beta}{8}n}= \epsilon^\frac{3}{4}(p_0^u)^\frac{\beta}{4}\theta_3^n,
\end{equation} where $\theta_3:=e^{-\frac{\chi\beta}{8}}\in(0,1)$.


\medskip
\underline{Part 2:} Let $g^{\ul{R}}:=\lim_{n\rightarrow\infty} g_n^{\ul{R}}$ be the uniform and finite limit as in equation \eqref{JakeAmir}. Then,
 \begin{align}\label{part2I}
|g^{\ul{R}}(y)-(g^{\ul{S}}\circ I)(y)| \leq & \left|\int_{V^u(\ul{R})}e^{S_z^{\ul{R}}(y)}d\overline{m}_{\ul{R}}(z)-\int_{V^u(\ul{S})}e^{S_{z'}^{\ul{S}}(I(y))}d\overline{m}_{\ul{S}}(z')\right|\nonumber \\ \nonumber
 = & \left|\int_{V^u(\ul{R})}e^{S_z^{\ul{R}}(y)}d\overline{m}_{\ul{R}}(z)-\int_{V^u(\ul{R})}e^{S_{I(z)}^{\ul{S}}(I(y))}\Jac(d_zI)d\overline{m}_{\ul{R}}(z)\right|\\ \nonumber
 \leq & \int_{V^u(\ul{R})}|e^{S_z^{\ul{R}}(y)}-e^{S_{I(z)}^{\ul{S}}(I(y))}|d\overline{m}_{\ul{R}}(z)+\int_{V^u(\ul{R})}e^{S_{I(z)}^{\ul{S}}(I(y))}|\Jac(d_zI)-1|d\overline{m}_{\ul{R}}(z)\\ \nonumber
 = & \int_{V^u(\ul{R})}e^{S_{I(z)}^{\ul{S}}(I(y))}\cdot|e^{S_z^{\ul{R}}(y)-S_{I(z)}^{\ul{S}}(I(y))}-1|d\overline{m}_{\ul{R}}(z)\\ \nonumber
 +&\int_{V^u(\ul{R})}e^{S_{I(z)}^{\ul{S}}(I(y))}|\Jac(d_zI)-1|d\overline{m}_{\ul{R}}(z)\\
 \leq & e^{\epsilon}\Big(\int_{V^u(\ul{R})}|e^{S_z^{\ul{R}}(y)-S_{I(z)}^{\ul{S}}(I(y))}-1|d\overline{m}_{\ul{R}}(z) +\int_{V^u(\ul{R})}|\Jac(d_zI)-1|d\overline{m}_{\ul{R}}(z)\Big),
 \end{align}
 because $S_a^{\ul{R}}(b), S_c^{\ul{R}}(d) =\pm \epsilon$ for all $a,b\in V^u(\ul{R}), c,d\in V^u(\ul{S})$, because of equation \eqref{forholderphi}.

 By part 0,
 $|\Jac(d_zI)-1|\leq \epsilon(p_0^u)^\frac{\beta}{4}\tilde{\theta}^n$. Plugging this in equation \eqref{part2I} yields \begin{equation}\label{BestWest2}|g^{\ul{R}}(y)-g^{\ul{S}}\circ I(y)| \leq e^\epsilon\cdot\Vol(V^u(\ul{R}))\left(2
 \sup_{y,z\in V^u(\ul{R})}|S_z^{\ul{R}}(y)-S_{I(z)}^{\ul{S}}(I(y))|+ \epsilon(p_0^u)^\frac{\beta}{4}\tilde{\theta}^n\right).\end{equation}

\medskip
\underline{Part 3:} By \cite[Lemma~2.15]{SBO}, $\exists \omega_0\geq1$ which depends only on $M,f$ and $\beta$, s.t. $\frac{Q_\epsilon\circ f(\cdot)}{Q_\epsilon(\cdot)}=\omega_0^{\pm1}$. Let $\kappa=\kappa(f,\chi)>0$ be the constant given by Theorem \ref{pesinoseledec}. Define,
\begin{equation}\label{defr}
	r:=\log_{\tilde{\theta}}\left(\frac{\theta}{(2\omega_0)^\frac{\beta}{4}+\kappa+1}\right)+1, \text{ }n_0:=\lceil r\rceil\in\mathbb{N},
\end{equation}
and notice that $r>\frac{6}{\beta}>1$. 
By assumption $n\geq n_0$, so $m:=\lfloor \frac{n}{r}\rfloor\geq 1$. Then,
\begin{align}\label{firsttwosums}
&|S_z^{\ul{R}}(y)-S_{I(z)}^{\ul{S}}(I(y))|\leq |S_z^{\ul{R}}(y)-S_{m,z}^{\ul{R}}(y)|+|S_{m,z}^{\ul{R}}(y)-S_{m,I(z)}^{\ul{S}}(I(y))|+|S_{m,I(z)}^{\ul{S}}(I(y))-S_{I(z)}^{\ul{S}}(I(y))| \nonumber\\
&(\because\text{eq. }\eqref{BestWest3}) \leq  2\epsilon^\frac{3}{4}(p_0^u)^\frac{\beta}{4}\cdot\theta_3^{\frac{n}{r}-1} \nonumber\\
&+\sum_{k=0}^{m-1}\Big| \log\Jac(d_{f^{-k}(z)}f^{-1}\rvert_{T_{f^{-k}(z)}f^{-k}[V^u(\ul{R})]})-\log\Jac(d_{f^{-k}(y)}f^{-1}\rvert_{T_{f^{-k}(y)}f^{-k}[V^u(\ul{R})]})\nonumber\\
&+  \log\Jac(d_{f^{-k}(I(y))}f^{-1}\rvert_{T_{f^{-k}(I(y))}f^{-k}[V^u(\ul{S})]})-\log\Jac(d_{f^{-k}(I(z))}f^{-1}\rvert_{T_{f^{-k}(I(z))}f^{-k}[V^u(\ul{S})]})\Big|\nonumber\\
&\leq
(p_0^u)^\frac{\beta}{4}\frac{2\epsilon^\frac{3}{4}}{\theta_3}\cdot(\theta_3^\frac{1}{r})^n+\sum_{k=0}^{m-1}\Big| \log\Jac(d_{f^{-k}(y)}f^{-1}\rvert_{T_{f^{-k}(y)}f^{-k}[V^u(\ul{R})]})-\log\Jac(d_{f^{-k}(I(y))}f^{-1}\rvert_{T_{f^{-k}(I(y))}f^{-k}[V^u(\ul{S})}])\Big|\nonumber\\
&+  \sum_{k=0}^{m-1}\Big|\log\Jac(d_{f^{-k}(z)}f^{-1}\rvert_{T_{f^{-k}(z)}f^{-k}[V^u(\ul{R})]})-\log\Jac(d_{f^{-k}(I(z))}f^{-1}\rvert_{T_{f^{-k}(I(z))}f^{-k}[V^u(\ul{S})]})\Big|.
\end{align}

\medskip
\underline{Part 4:} We wish to bound the expressions of the form

$$|\log\Jac(d_{f^{-k}(z)}f^{-1}\rvert_{T_{f^{-k}(z)}f^{-k}[V^u(\ul{R})]})-\log\Jac(d_{f^{-k}(I(z))}f^{-1}\rvert_{T_{f^{-k}(I(z))}f^{-k}[V^u(\ul{S})]})\Big|\text{, for }k\leq m.$$
Set $I_k:V^u(\sigma_R^k\ul{R})\rightarrow V^u(\sigma_R^k\ul{S})$, $\psi_{x_{-k}}(F^{\sigma_R^k\ul{u}}(t),t)$ $_{\mapsto}^{I_k}$ $\psi_{x_{-k}}(F^{\sigma_R^k\ul{v}}(t),t)$, where $F^{\sigma_R^k\ul{u}},F^{\sigma_R^k\ul{v}}$ are the representing functions of $V^u(\sigma_R^k\ul{u}),V^u(\sigma_R^k\ul{v})$ (as in the definition of $I$). By Corollary \ref{containment}, $f^{-k}[V^u(\ul{R})]\subseteq V^u(\sigma_R^k\ul{R}),f^{-k}[V^u(\ul{S})]\subseteq V^u(\sigma_R^k\ul{S})$.

\begin{align}\label{louis}
&\frac{\Jac(d_{f^{-k}(z)}f^{-1}\rvert_{T_{f^{-k}(z)}V^u(\sigma_R^k\ul{R})})}{\Jac(d_{f^{-k}(I(z))}f^{-1}\rvert_{T_{f^{-k}(I(z))}V^u(\sigma_R^k\ul{S})})}=\nonumber\\
&=\frac{\Jac(d_{f^{-k}(z)}f^{-1}\rvert_{T_{f^{-k}(z)}V^u(\sigma_R^k\ul{R})})}{\Jac(d_{I_k(f^{-k}(z))}f^{-1}\rvert_{T_{I_k(f^{-k}(z))}V^u(\sigma_R^k\ul{S})})}\cdot\frac{\Jac(d_{I_k(f^{-k}(z))}f^{-1}\rvert_{T_{I_k(f^{-k}(z))}V^u(\sigma_R^k\ul{S})})}{\Jac(d_{f^{-k}(I(z))}f^{-1}\rvert_{T_{f^{-k}(I(z))}V^u(\sigma_R^k\ul{S})})}.
\end{align}

We proceed to bound the first term.
Consider the local isometries $\Theta_{D_k}:TD_k\rightarrow \mathbb{R}^d$, $k\geq0$, as in Definition \ref{localisometries}, s.t. $D_k$ is a neighborhood which contains $\psi_{x_{-k}}[R_{Q_\epsilon(x_{-k})}(0)]\supseteq V^u(\sigma_R^k\ul{R}),V^u(\sigma_R^k\ul{S})$; and where $\nu_p:\mathbb{R}^d\rightarrow T_pM$ is the local inverse such that $\nu_p\Theta_{D_k}=\mathrm{Id}_{T_pM}$ for all $p\in D_k$. For easier notation, we omit the restriction to appropriate tangent spaces when they are clear from the context.
\begin{align}\label{natityahu}
    &\|\Theta_{D_{k+1}}d_{f^{-k}(z)}f^{-1}|_{T_{f^{-k}(z)}V^u(\sigma_R^k\ul{R})}-\Theta_{D_{k+1}}d_{I_k(f^{-k}(z))}f^{-1}d_{f^{-k}(z)}I_k\|\\
    &=\|\Theta_{D_{k+1}}d_{f^{-k}(z)}f^{-1}\underset{\mathrm{Id}}{\underbrace{\nu_{f^{-k}(z)}\Theta_{D_{k}} }} \underset{\mathrm{Id}}{\underbrace{d_{f^{-k}(z)}\mathrm{Id}_{V^u(\sigma^k_R\ul{u})}}}-\Theta_{D_{k+1}}d_{I_k(f^{-k}(z))}f^{-1}\underset{\mathrm{Id}}{\underbrace{\nu_{I_k(f^{-k}(z))}\Theta_{D_k}}}d_{f^{-k}(z)}I_k\|\nonumber\\
    &\leq\|\Theta_{D_{k+1}}d_{f^{-k}(z)}f^{-1} |_{T_{f^{-k}(z)}V^u(\sigma_R^k\ul{R})}\nu_{f^{-k}(z)}- \Theta_{D_{k+1}}d_{I_k(f^{-k}(z))}f^{-1}\nu_{I_k(f^{-k}(z))}\|\cdot \|d_\cdot I_k\|+\nonumber\\
    &\hspace{0.475cm}\|d_\cdot f^{-1}\|\cdot \|\Theta_{D_{k}} \mathrm{Id}_{T_{f^{-k}(z)}V^u(\sigma^k_R\ul{u})}-\Theta_{D_k}d_{f^{-k}(z)}I_k\|\nonumber\\
    &\leq \|d_\cdot I_k\|\cdot H_0 d(f^{-k}(z),I_k(f^{-k}(z)))^\beta+M_f\|\Theta_{D_k}\mathrm{Id}_{T_{f^{-k}(z)}V^u(\sigma^k_R\ul{R})}-\Theta_{D_k}d_{f^{-k}(z)}I_k\| \nonumber\\
    &\leq 2H_0\cdot2d_{C^0}(V^u(\sigma_R^k\ul{R}),V^u(\sigma_R^k\ul{S}))^\beta+M_f8\|C_\chi^{-1}(x_{-k})\|d_{C^1}(V^u(\sigma_R^k\ul{R}),V^u(\sigma_R^k\ul{S}))\nonumber\\
    &\leq (4H_0+8M_f)\epsilon(p_{-k}^u)^\frac{\beta}{4}\tilde{\theta}^{n-k}, \nonumber
\end{align}
 where $H_0$ is given by Definition \ref{localisometries}. The last line is by equation \eqref{ams} and by equations \eqref{DeltaAir2} and \eqref{DeltaAir3.0}, applied to the shifted sequences $\sigma_R^k\ul{u}, \sigma_R^k\ul{v} $. Then we get,
 \begin{align*}
     &|\Jac(d_{f^{-k}(z)}f^{-1}\rvert_{T_{f^{-k}(z)}V^u(\sigma_R^k\ul{R})})-\Jac(d_{I_k(f^{-k}(z))}f^{-1}\rvert_{T_{I_k(f^{-k}(z))}V^u(\sigma_R^k\ul{S})})|\\
     &\leq|\Jac(d_{f^{-k}(z)}f^{-1} |_{T_{f^{-k}(z)}V^u(\sigma_R^k\ul{R})})-\Jac(d_{f^{-k}(z)}(f^{-1}\circ I_k))|+M_f^d|\Jac(d_{f^{-k}(z)}I_k)-1|\\
     &\leq C_1(4H_0+8M_f)\epsilon(p_{-k}^u)^\frac{\beta}{4}\tilde{\theta}^{n-k}+M_f^d\epsilon(p_{-k}^u)^\frac{\beta}{4}\tilde{\theta}^{n-k}\leq C_1(4H_0+9M_f^d)\epsilon (p_0^u)^\frac{\beta}{4}(e^\epsilon\omega_0)^{\frac{k\beta}{4}}\tilde{\theta}^{n-k},
 \end{align*}
where $C_1$ is the Lipschitz constant for the Jacobian and $p_{-k}^u\leq (e^\epsilon \omega_0)^k\cdot p_0^u$ (see equation \eqref{radioblanew2} in Appendix A). In addition, since,\begin{equation*}\Jac(d_{f^{-k}(z)}f^{-1}|_{T_{f^{-k}(z)}V^u(\sigma_R^k\ul{R}))}),\text{ }\Jac(d_{I_k(f^{-k}(z))}f^{-1}d_{f^{-k}(z)}I_k|_{T_{f^{-k}(z)}V^u(\sigma_R^k\ul{R})})\geq \frac{1}{(2M_f)^d},\end{equation*}
since $|1\pm t|=e^{\pm2t}$ for all sufficiently small $t>0$, we get,
\begin{align}\label{circla}
\frac{\Jac(d_{f^{-k}(z)}f^{-1}\rvert_{T_{f^{-k}(z)}V^u(\sigma_R^k\ul{R})})}{\Jac(d_{I_k(f^{-k}(z))}f^{-1}\rvert_{T_{I_k(f^{-k}(z))}V^u(\sigma_R^k\ul{S})})}=&\exp\left(\pm 2C_1(2M_f)^d(4H_0+9M_f^d)\epsilon (p_0^u) ^\frac{\beta}{4} (e^\epsilon\omega_0) ^{\frac{k\beta}{4}}\tilde{\theta}^{n-k}\right)\nonumber\\
=&\exp\left(\pm\epsilon^\frac{3}{4} (p_0^u) ^\frac{\beta}{4} (e^\epsilon\omega_0) ^{\frac{k\beta}{4}}\tilde{\theta}^{n-k}\right)=\exp\left(\pm\epsilon^\frac{3}{4} (p_0^u) ^\frac{\beta}{4}(\frac{(e^\epsilon\omega_0) ^{\frac{\beta}{4}}}{\tilde{\theta}})^k(\tilde{\theta}^\frac{n}{m})^m\right)\nonumber\\
=&\exp\left(\pm\epsilon^\frac{3}{4} (p_0^u) ^\frac{\beta}{4}(\frac{(e^\epsilon\omega_0)^{\frac{\beta}{4}}}{\tilde{\theta}})^m(\tilde{\theta}^r)^m\right)\nonumber\\
=&\exp\left(\pm\epsilon^\frac{3}{4} (p_0^u) ^\frac{\beta}{4}\theta^m\right)=\exp(\pm\epsilon^2\theta^m).
\end{align}
The last line is due to the choice of $r$ in equation \eqref{defr}, and since $e^\epsilon\leq 2$ for $\epsilon>0$ small enough. We continue to bound the second term of equation \eqref{louis}. Define $f_{x_{i}x_{i+1}}:=\psi_{x_{i+1}}^{-1}\circ f\circ\psi_{x_i}$. By \cite[Proposition~2.21]{SBO}, we can write $$f_{x_ix_{i+1}}(v_s,v_u)=(D_s v_s+h_s(v_s,v_u),D_u v_u+h_u(v_s,v_u))$$ for $v_{u}=\pi_{u}\psi_{x_i}^{-1}\xi_i$, $v_{s}=\pi_{s}\psi_{x_i}^{-1}\xi_i$, $\forall\xi_i\text{ tangent to } \psi_{x_i}[R_{Q_\epsilon(x_i)}(0)]$,
where $\pi_s$ is the projection onto the $(d-u)$ first coordinates, $\kappa^{-1}\leq\|D_s^{-1}\|^{-1},\|D_s\|\leq e^{-\chi}$ and $e^\chi\leq\|D_u^{-1}\|^{-1},\|D_u\|\leq\kappa$,
$\|\frac{\partial(h_s,h_u)}{\partial(v_s,v_u)}\|<\epsilon$, and $\|\frac{\partial(h_s,h_u)}{\partial(v_s,v_u)}\rvert_{v_1}-\frac{\partial(h_s,h_u)}{\partial(v_s,v_u)}\rvert_{v_2}\|\leq\epsilon|v_1-v_2|^{\beta/3}$ on $R_{Q_\epsilon(x_i)}(0)$. A similar statement holds for $f_{x_ix_{i+1}}^{-1}$. See also Theorem \ref{pesinoseledec}. Then,
\begin{align}
|\psi_{x_{-k}}^{-1}(f^{-k}(I(z)))-\psi_{x_{-k}}^{-1}(I_k(f^{-k}(z)))|\leq& |\psi_{x_{-k}}^{-1}(f^{-k}(I(z)))-\psi_{x_{-k}}^{-1}(f^{-k}(z))| \nonumber \\
+&|\psi_{x_{-k}}^{-1}(f^{-k}(z))-\psi_{x_{-k}}^{-1}(I_k(f^{-k}(z)))|\nonumber\\
\leq&(\kappa+\epsilon)^kd_{C^0}(V^u(\ul{u}),V^u(\ul{v}))+d_{C^0}(V^u(\sigma_R^k\ul{u}),V^u(\sigma_R^k\ul{v}))\nonumber\\
\leq& (\kappa+\epsilon)^kd_{C^1}(V^u(\ul{u}),V^u(\ul{v}))+d_{C^1}(V^u(\sigma_R^k\ul{u}),V^u(\sigma_R^k\ul{v})).
\end{align}

By the estimates of equation \eqref{DeltaAir2} and equation \eqref{DeltaAir}, applied to the shifted sequences, we get $\forall 0\leq k\leq m$,
\begin{equation}\label{BestWest}
d_{C^1}(V^u(\sigma_R^k\ul{u}),V^u(\sigma_R^k\ul{v}))\leq (p^u_{-k})^\frac{\beta}{4}\tilde{\theta}^{n-k}
.
\end{equation}
In addition, by the admissibility of the chain $\ul{u}$ (or $\ul{v}$), $p_{0}^u\leq e^{\epsilon k}p_{-k}^u$ (see \cite[Definition~2.23]{SBO}).
Plugging this back in equation \eqref{BestWest} yields,
\begin{align}\label{circe}
|\psi_{x_{-k}}^{-1}(f^{-k}(I(z)))-\psi_{x_{-k}}^{-1}(I_k(f^{-k}(z)))|\leq&(\kappa+\epsilon)^k\tilde{\theta}^{n}(p_0^u)^\frac{\beta }{4}+(p_{-k}^u)^\frac{\beta}{4}\tilde{\theta}^{n-k}\nonumber\\
\leq &\left((e^\epsilon(\kappa+\epsilon))^m(\tilde{\theta}^{r})^m+\tilde{\theta}^{n-k}\right)(p_{-k}^u)^{\frac{\beta}{4}}\nonumber\\
\leq& \left((\kappa+1)^m(\tilde{\theta}^{r})^m+\tilde{\theta}^{n-k}\right)(p_{-k}^u)^{\frac{\beta}{4}}\leq 2(p_{-k}^u)^\frac{\beta}{4}\theta^m.
\end{align}
 Define $A_k:T_{f^{-k}(I(z))}V^u(\sigma_R^k\ul{S})\rightarrow T_{I_k(f^{-k}(z))}V^u(\sigma_R^k\ul{S})$, by $$A_k d_{\psi_{x_{-k}}^{-1}(f^{-k}(I(z)))}\psi_{x_{-k}}\Big(\begin{array}{c}
u\\
d_{\pi_u\psi_{x_{-k}}^{-1}(f^{-k}(I(z)))}F^{\sigma_R^k\ul{v}}u\\
\end{array}\Big):= d_{\psi_{x_{-k}}^{-1}(I_k(f^{-k}(z)))}\psi_{x_{-k}}\Big(\begin{array}{c}
u\\
d_{\pi_u\psi_{x_{-k}}^{-1}(I_k(f^{-k}(z)))}F^{\sigma_R^k\ul{v}}u\\
\end{array}\Big).$$

By carrying out the same type of estimates as in equation \eqref{ams}, it follows that

\begin{align*}
\|\Theta_{D_k}\mathrm{Id}_{T_{f^{-k}(z)}V^u(\sigma^k_R\ul{S})}-\Theta_{D_k} A_k\|\leq& 2\|C_\chi^{-1}(x_{-k})\|\cdot\Big(\text{H\"ol}_{\frac{\beta}{3}}(d_\cdot F^{\sigma_R^k\ul{v}})\cdot2|\psi_{x_{-k}}^{-1}(I_k(f^{-k}(z)))-\psi_{x_{-k}}^{-1}(f^{-k}(I(z)))|^\frac{\beta}{3}\\
+&2|\psi_{x_{-k}}^{-1}(I_k(f^{-k}(z)))-\psi_{x_{-k}}^{-1}(f^{-k}(I(z)))|\cdot2\Big).
\end{align*}

Together with equation \eqref{circe} we get
\begin{align*}
\|\Theta_{D_k}\mathrm{Id}_{T_{f^{-k}(z)}V^u(\sigma^k_R\ul{S})}-\Theta_{D_k} A_k\|\leq&8\|C_\chi^{-1}(x_{-k})\|(2p_{-k}^u\theta^m)^{\frac{\beta}{3}}\leq\epsilon^4(\theta^{\frac{\beta}{3}})^m.	
\end{align*}

It follows that $\Jac(A_k)=e^{\pm2d\epsilon^3(\theta^\frac{\beta}{3})^m}$. Next, as in equation \eqref{natityahu},
\begin{align*}
    &\|\Theta_{D_{k+1}}d_{f^{-k}(I(z))}f^{-1}|_{T_{f^{-k}(I(z))}V^u(\sigma_R^k\ul{S})}-\Theta_{D_{k+1}}d_{I_k(f^{-k}(z))}f^{-1}A_k\|=\\
    &=\|\Theta_{D_{k+1}}d_{f^{-k}(I(z))}f^{-1}\nu_{f^{-k}(I(z))}\Theta_{D_k}d_{f^{-k}(I(z))}\mathrm{Id}_{V^u(\sigma^k_R\ul{S})}-\Theta_{D_{k+1}}d_{I_k(f^{-k}(z))}f^{-1}\nu_{I_k(f^{-k}(z))}\Theta_{D_k}A_k\|\\
    &\leq \|A_k\|\cdot H_0 d(f^{-k}(I(z)),I_k(f^{-k}(z)))^\beta+M_f\|\Theta_{D_k}d_{f^{-k}(I(z))}\mathrm{Id}_{V^u(\sigma^k_R\ul{S})}-\Theta_{D_k}A_k\|\\
    &\leq 2H_0\cdot(2|\psi_{x_{-k}}^{-1}(I_k(f^{-k}(z)))-\psi_{x_{-k}}^{-1}(f^{-k}(I(z)))|)^\beta+M_f\epsilon^4(\theta^\frac{\beta}{3})^m\leq \epsilon^3(\theta^\frac{\beta}{3})^m.
\end{align*}
Therefore, \begin{align*}\frac{\Jac(d_{I_k(f^{-k}(z))}f^{-1}\rvert_{T_{I_k(f^{-k}(z))}V^u(\sigma_R^k\ul{S})})}{\Jac(d_{f^{-k}(I(z))}f^{-1}\rvert_{T_{f^{-k}(I(z))}V^u(\sigma_R^k\ul{S})})}=&\frac{\Jac(d_{I_k(f^{-k}(z))}f^{-1}A_k)}{\Jac(d_{f^{-k}(I(z))}f^{-1} |_{T_{f^{-k}(I(z))}V^u(\sigma_R^k\ul{S})})}\cdot\frac{1}{\Jac(A_k)}\\
=&e^{\pm(2(2M_f)^dC_1\epsilon^3(\theta^\frac{\beta}{3})^m+2d\epsilon^3(\theta^\frac{\beta}{3})^m)}.
\end{align*} This, equation \eqref{circla}, and equation \eqref{louis} yield: $$\frac{\Jac(d_{f^{-k}(z)}f^{-1}\rvert_{T_{f^{-k}(z)}f^{-k}[V^u(\ul{R})]})}{\Jac(d_{f^{-k}(I(z))}f^{-1}\rvert_{T_{f^{-k}(I(z))}f^{-k}[V^u(\ul{S})]})}=\exp(\pm2\epsilon^{2.5}(\theta^\frac{\beta}{3})^m)=\exp(\pm2\frac{\epsilon^{2.5}}{\theta^\frac{\beta}{3}}(\theta^\frac{\beta}{3r})^n)=\exp(\pm\epsilon^2(\theta^\frac{\beta}{3r})^n).$$
Hence, by equation \eqref{firsttwosums},
\begin{align*}|S_z^{\ul{R}}(y)-S_{I(z)}^{\ul{S}}(I(y))|\leq&(p_0^u)^\frac{\beta}{4}\frac{2 \epsilon^\frac{3}{4}}{\theta_3}\cdot(\theta_3^\frac{1}{r})^n+\epsilon^2\sum_{k=0}^m(\theta^\frac{\beta}{3r})^n\leq (p_0^u)^\frac{\beta}{4}\frac{2\epsilon^\frac{3}{4}}{\theta_3}\cdot(\theta_3^\frac{1}{r})^n+\epsilon^2n\cdot(\theta^\frac{\beta}{3r})^n\\
\leq&\frac{\epsilon^\frac{7}{4}}{2}\cdot(\theta_3^\frac{1}{r})^n+\frac{\epsilon^\frac{7}{4}}{2}(\theta^\frac{\beta}{6r})^n \text{ (for small enough }\epsilon)\leq \epsilon^\frac{7}{4}\cdot(\max\{\theta_3^\frac{1}{r},\theta^\frac{\beta}{6r}\})^n= \epsilon^\frac{7}{4} \theta_2^n,
\end{align*}
where we define \begin{equation}\label{theta2late}\theta_2:= \max\{\theta_3^\frac{1}{r},\theta^\frac{\beta}{6r}
\}= \max\left\{\left(e^{-\frac{\chi\beta}{8}}\right)^\frac{1}{r},\theta^\frac{\beta}{6r}
\right\}\geq\tilde{\theta}.\end{equation}  Plugging this into equation \eqref{BestWest2} yields (for $\epsilon$ sufficiently small),
\begin{equation}\label{BestWest4}|g^{\ul{R}}(y)-g^{\ul{S}}\circ I(y)| \leq e^\epsilon\Vol(V^u(\ul{R}))(2\epsilon^\frac{7}{4}\theta_2^n+\epsilon(p_0^u)^\frac{\beta}{4}\tilde{\theta}^n)\leq \epsilon^\frac{3}{2}\Vol(V^u(\ul{R}))\theta_2^n.\end{equation}

\medskip
\underline{Part 5:}

\begin{align*}
|\Vol(V^u(\ul{R}))-\Vol(V^u(\ul{S}))| =&|\int_O\Jac(d_t(\psi_{x_0}\tilde{F}^{\ul{u}}))d\mathrm{Leb}(t)-\int_O\Jac(d_t(\psi_{x_0}\tilde{F}^{\ul{v}}))d\mathrm{Leb}(t)|\\
\leq&\int_O|\Jac(d_t(\psi_{x_0}\tilde{F}^{\ul{u}}))-\Jac(d_t(\psi_{x_0}\tilde{F}^{\ul{v}}))|d\mathrm{Leb}(t)\\
=& \int_O\Jac(d_t (\psi_{x_0}\widetilde{F}^u))\cdot |1-\Jac(d_t(\psi_{x_0}\tilde{F}^{\ul{v}}))\Jac(d_{\psi_{x_0}\tilde{F}^{\ul{u}}(t)}(\psi_{x_0}\tilde{F}^{\ul{u}})^{-1}|d\mathrm{Leb}(t) \\
=& \int_O \Jac(d_t(\psi_{x_0}\widetilde{F}^u))\cdot|1-\Jac(d_{\psi_{x_0}\tilde{F}^{\ul{u}}(t)}I)|d\mathrm{Leb}(t).
\end{align*}
In equation \eqref{DeltaAir} we saw $|\Jac(d_\cdot I)-1|\leq \epsilon(p_0^u)^\frac{\beta}{4} \tilde{\theta}^n$. Therefore,

\begin{align}\label{BestWest5}
\frac{|\Vol(V^u(\ul{R}))-\Vol(V^u(\ul{S}))|}{\Vol(V^u(\ul{R}))} \leq&\frac{\epsilon(p_0^u)^\frac{\beta}{4} \tilde{\theta}^n\int_O \Jac(d_t(\psi_{x_0}\widetilde{F}^u))\mathrm{Leb}(t)}{\int_O \Jac(d_t(\psi_{x_0}\widetilde{F}^u))\mathrm{Leb}(t)}=\epsilon^\frac{3}{2}\tilde{\theta}^n,
\end{align}
when $\epsilon$ is sufficiently small.
 In particular,
\begin{align}\label{forpart5}
	|\Vol(V^u(\ul{R}))-\Vol(V^u(\ul{S}))|\leq \Vol(V^u(\ul{R}))\cdot \epsilon^\frac{3}{2}\tilde{\theta}^n\text{, and }
	\frac{\Vol(V^u(\ul{S}))}{\Vol(V^u(\ul{R}))}=e^{\pm\epsilon}.
\end{align}

\medskip
\underline{Part 6:} Recall, $g^{\ul{R}}= e^{\pm\epsilon}\Vol(V^u(\ul{R}))$ since $g^{\ul{R}}(y)=\int\limits_{V^u(\ul{R})}e^{S^{\ul{R}}_{z}(y)}d\overline{m}_{\ul{R}}(z)$ (and similarly for $g^{\ul{S}}$), and $\rho^{\ul{R}}=\frac{\Vol(V^u(\ul{R}))}{g^{\ul{R}}}$, as in Theorem \ref{naturalmeasures}. Hence

\begin{align*}
\|\rho^{\ul{R}}-\rho^{\ul{S}}\circ I\|=&\|\frac{\Vol(V^u(\ul{R}))}{g^{\ul{R}}}-\frac{\Vol(V^u(\ul{S}))}{g^{\ul{S}}\circ I}\|\\
\leq& \|\frac{1}{g^{\ul{R}}}\|\cdot\|\frac{1}{g^{\ul{S}}}\|\cdot\|\Vol(V^u(\ul{R}))(g^{\ul{R}}-g^{\ul{S}}\circ I)+g^{\ul{R}}\cdot(\Vol(V^u(\ul{S}))-\Vol(V^u(\ul{R})))\|\\
\leq& e^{2\epsilon}\frac{1}{\Vol(V^u(\ul{S}))}\|g^{\ul{R}}-g^{\ul{S}}\circ I\|+e^{3\epsilon}\frac{1}{\Vol(V^u(\ul{S}))}|\Vol(V^u(\ul{R}))-\Vol(V^u(\ul{S}))|\\
(\because\text{eq. }\eqref{BestWest4})\leq& \frac{\Vol(V^u(\ul{R}))}{\Vol(V^u(\ul{S}))}\Big[e^{2\epsilon}\frac{1}{\Vol(V^u(\ul{R}))}\epsilon^\frac{3}{2}\Vol(V^u(\ul{R}))\theta_2^n+e^{3\epsilon}\frac{1}{\Vol(V^u(\ul{R}))}\cdot\Vol(V^u(\ul{R}))\epsilon^\frac{3}{2}\tilde{\theta}^n\Big]\\
\leq& e^\epsilon\cdot \epsilon^\frac{3}{2}\cdot(e^{2\epsilon}+e^{3\epsilon})\cdot\theta_2^n\leq \epsilon \theta_2^n\text{ for all small enough }\epsilon.
\end{align*}
\end{proof}

\begin{claim}\label{VolHol}
For $n_0\in \mathbb{N}$ and $\theta_2\in(0,1)$ 
as in Lemma \ref{0.1}, if $d(\ul{R},\ul{S})=e^{-n}$, $n\geq n_0$, then $|\phi(\ul{R})-\phi(\ul{S})|\leq \sqrt{\epsilon}\theta_2^n$. In addition, $\phi$ has summable variations: $\sum_{k=2}^\infty \mathrm{var}_k(\phi)<\infty$, where $\mathrm{var}_k(\phi)=\sup\{|\phi(\ul{R})-\phi(\ul{S})|:d(\ul{R},\ul{S})\leq e^{-k}\}$.
\end{claim}
\begin{proof}
Let $n_0=\lceil r\rceil$, where $r$ and $n_0$ are constants given by equation \eqref{defr}. We assume $d(\ul{R},\ul{S})= e^{-n}\leq e^{-n_0}$
. Recall the formula from the statement of Theorem \ref{naturalmeasures},
\begin{align*}
    e^{\phi(\ul{R})}=&\lim_{m \rightarrow\infty}\frac{\Vol(f^{-m}f^{-1}[V^u(\ul{R})])}{\Vol(f^{-m}[V^u(\sigma_R\ul{R})])}=\lim_{m \rightarrow\infty}\int_{f^{-1}[V^u(\ul{R})]}\frac{d\overline{m}_{\sigma_R\ul{R}}(y)}{\int_{V^u(\sigma_R\ul{R})}\frac{\Jac(d_zf^{-m})}{\Jac(d_yf^{-m})}d\overline{m}_{\sigma_R\ul{R}}(z)}\\
    =&\int_{f^{-1}[V^u(\ul{R})]}\frac{d\overline{m}_{\sigma_R\ul{R}}(y)}{g^{\sigma_R\ul{R}}(y)}=\int_{V^u(\ul{R})}\frac{\Jac(d_yf^{-1})d\overline{m}_{\ul{R}}(y)}{\left(g^{\sigma_R\ul{R}}\circ f^{-1}\right)(y)},
\end{align*}
where $\overline{m}_{\ul{R}}$ and $\overline{m}_{\sigma_R\ul{R}}$ are the induced Riemannian volumes of $V^u(\ul{R})$ and $V^u(\sigma_R\ul{R})$, the Jacobians refer to the Jacobians of the restriction of the differential to the relevant unstable space, and $g^{\sigma_R\ul{R}}$ is defined in equation \eqref{JakeAmir}.

Assume 
w.l.o.g. that $n\geq2$. Consider the maps $I_1:V^u(\ul{R})\rightarrow V^u(\ul{S})$, $I_2:V^u(\sigma_R\ul{R})\rightarrow V^u(\sigma_R\ul{S})$ given by Lemma \ref{0.1}. Then,
\begin{equation}\label{phis}
e^{\phi(\ul{S})}=\int_{V^u(\ul{S})}\frac{\Jac(d_{y'}f^{-1})d\overline{m}_{\ul{S}}(y')}{g^{\sigma_R\ul{S}}\circ f^{-1}(y')}=\int_{V^u(\ul{R})}\frac{\Jac(d_{I_1(y)}f^{-1}) \Jac(d_{y}I_1) d\overline{m}_{\ul{R}}(y)}{\left(g^{\sigma_R\ul{S}}\circ f^{-1}\circ I_1\right)(y)}.
\end{equation}
 We now wish to bound the ratio between the integrands in the formul\ae \text{ }of $e^{\phi(\ul{S})}$ and $e^{\phi(\ul{R})}$. By equation \eqref{DeltaAir}, $\Jac(d_yI_1)=e^{\pm\epsilon \tilde{\theta}^n}$ where $\tilde{\theta}\in (0,1)$ is a constant depending only on $\chi,\beta$ (which equals $\theta^\frac{1}{6}$). By equation \eqref{circla} (with $m=k=0$), $\frac{\Jac(d_{I_1(y)}f^{-1})}{\Jac(d_yf^{-1})}=e^{\pm\epsilon^2\theta_2^n}$, where $\theta_2\in[\tilde{\theta},1)$ is a constant depending only on $\chi,\beta$ and is given by equation \eqref{theta2late}. We are therefore left to bound $\frac{\left(g^{\sigma_R\ul{S}}\circ f^{-1}\circ I_1\right)(y)}{\left(g^{\sigma_R\ul{R}}\circ f^{-1}\right)(y)}$:
 \begin{align}\label{chinesesubway}
     |g^{\sigma_R\ul{R}}(f^{-1}(y))-g^{\sigma_R\ul{S}}( f^{-1}(I_1(y)))|\leq&|g^{\sigma_R\ul{R}}( f^{-1}(y))-g^{\sigma_R\ul{S}}\circ I_2(f^{-1}(y))|\\
     +&|g^{\sigma_R\ul{S}}\circ I_2(f^{-1}(y))-g^{\sigma_R\ul{S}}(f^{-1}(I_1(y)))|\nonumber\\
     \leq& \epsilon^\frac{3}{2} \Vol(V^u(\sigma_R\ul{S}))\theta_2^{n-1}+ |g^{\sigma_R\ul{S}}\circ I_2(f^{-1}(y))-g^{\sigma_R\ul{S}}(f^{-1}(I_1(y)))|,\nonumber
 \end{align}
where the last line is by equation \eqref{BestWest4}
.
To bound the second summand, denote $t_1=\left(f^{-1}\circ I_1\right)(y),t_2=\left(I_2\circ f^{-1}\right)(y)$, then $\forall a,b\in V^u(\sigma_R\ul{S})$,

\noindent $S_{m,a}^{\sigma_R\ul{S}}(b):=\sum_{k=0}^{m-1}\log\Jac(d_{f^{-k}(a)}f^{-1}|_{T_{f^{-k}(a)}f^{-k}[V^u(\sigma_R\ul{S})]})$ $-\log\Jac(d_{f^{-k}(b)}f^{-1}|_{T_{f^{-k}(b)}f^{-k}[V^u(\sigma_R\ul{S})]})$, and
\begin{align}\label{bananot}
    |g^{\sigma_R\ul{S}}(t_2)-g^{\sigma_R\ul{S}}(t_1)|=&\lim_{m\rightarrow\infty}|\int_{V^u(\sigma_R\ul{S})}(e^{S_{m,z}^{\sigma_R\ul{S}}(t_2)}-e^{S_{m,z}^{\sigma_R\ul{S}}(t_1)})d\overline{m}_{\sigma_R\ul{S}}(z)|\nonumber\\
    \leq&\limsup_{m\rightarrow\infty}\int_{V^u(\sigma_R\ul{S})}e^{S_{m,z}^{\sigma_R\ul{S}}(t_1)}|e^{S_{m,z}^{\sigma_R\ul{S}}(t_2)-S_{m,z}^{\sigma_R\ul{S}}(t_1)}-1|d\overline{m}_{\sigma_R\ul{S}}(z)\nonumber\\
    =&\limsup_{m\rightarrow\infty}\int_{V^u(\sigma_R\ul{S})}e^{S_{m,z}^{\sigma_R\ul{S}}(t_1)}|e^{S_{m,t_1}^{\sigma_R\ul{S}}(t_2)}-1|d\overline{m}_{\sigma_R\ul{S}}(z)\nonumber\\
    \leq & e^\epsilon\Vol(V^u(\sigma_R\ul{S}))\sup_m\left|\frac{\Jac(d_{t_1}f^{-m}|_{T_{t_1}V^u(\sigma_R\ul{S})})}{\Jac(d_{t_2}f^{-m}|_{T_{t_2}V^u(\sigma_R\ul{S})})}-1\right|,
\end{align}
where we use equation \eqref{forholderphi} (with $n=0$) to bound $S_{m,z}^{(\cdot)}(b)=\pm\epsilon$. In addition, equation \eqref{forholderphi} (with $n=0$) also implies $\frac{\Jac(d_{t_2}f^{-m}|_{T_{t_2}V^u(\sigma_R\ul{S})})}{\Jac(d_{t_1}f^{-m}|_{T_{t_1}V^u(\sigma_R\ul{S})})}=e^{\pm\epsilon\cdot d(t_1,t_2)^\frac{\beta}{4}}$. We wish to bound the exponent:
\begin{align*}d(t_1,t_2)=d(f^{-1}(I_1(y)),I_2(f^{-1}(y)))\leq &d(f^{-1}(I_1(y)),f^{-1}(y))+d(f^{-1}(y),I_2(f^{-1}(y)))\\
\leq &M_f d(y,I_1(y))+d(f^{-1}(y),I_2(f^{-1}(y)));
\end{align*}
 equation \eqref{adidax} bounds this by $M_f2d_{C^0}(V^u(\ul{R}),V^u(\ul{S}))+2d_{C^0}(V^u(\sigma_R\ul{R}),V^u(\sigma_R\ul{S}))$.

By equation \eqref{DeltaAir3.0}, for all small enough $\epsilon>0$, $d_{C^0}(V^u(\ul{R}),V^u(\ul{S}))\leq \epsilon^\frac{10}{\beta} \left(\theta^\frac{1}{6}\right) ^{n}$  and $d_{C^0}(V^u(\sigma_R\ul{R}),V^u(\sigma_R\ul{S}))\leq \epsilon^\frac{10}{\beta} \left(\theta^\frac{1}{6}\right) ^{n-1}$. Thus in total,
\begin{align*}
|g^{\sigma_R\ul{S}}(t_2)-g^{\sigma_R\ul{S}}(t_1)|\leq& e^\epsilon \Vol(V^u(\sigma_R\ul{S})) \cdot2\cdot\epsilon\cdot \left(4M_f\cdot \epsilon^\frac{10}{\beta} \left(\theta^\frac{1}{6}\right)^{n-1}\right)^\frac{\beta}{4}
\leq 2\epsilon^2 \Vol(V^u(\sigma_R\ul{S})) \left(\theta^\frac{\beta}{6\cdot 4}\right)^{n-1}\\
\leq &2\epsilon^2 \Vol(V^u(\sigma_R\ul{S})) \theta_2^{n-1},
\end{align*}
where $\theta_2\geq \theta^\frac{\beta}{6r}$ and $r>\frac{6}{\beta}>4$ is given by equation \eqref{defr}
. Plugging this back in equation \eqref{chinesesubway} yields,
\begin{align*}\label{chinesesubway2}
	|g^{\sigma_R\ul{R}}(f^{-1}(y))-g^{\sigma_R\ul{S}}(f^{-1}(I_1(y)))|\leq 2\epsilon^2 \Vol(V^u(\sigma_R\ul{S})) \theta_2^{n-1}+ \epsilon^\frac{3}{2} \Vol(V^u(\sigma_R\ul{S}))\theta_2^{n-1}\leq 2\epsilon^\frac{3}{2}\Vol(V^u(\sigma_R\ul{S}))\theta_2^{n-1}.
\end{align*}
Putting this together with the fact that $g^{\sigma_R\ul{S}}=e^{\pm\epsilon}\Vol(V^u(\sigma_R\ul{S}))$ (since $S_{m,z}^{\sigma_R\ul{S}}(\cdot)=\pm\epsilon$, $\forall z\in V^u(\sigma_R\ul{S})$), gives us,
$$\frac{g^{\sigma_R\ul{R}}(f^{-1}(y))}{g^{\sigma_R\ul{S}}(f^{-1}(I_1(y)))}=\exp(\pm4\epsilon^\frac{3}{2}\frac{1}{\theta_2}\theta_2^n)=\exp(\pm\epsilon\theta_2^n),$$
when $\epsilon$ is sufficiently small.
Now, by equation \eqref{phis} and the bounds $\frac{\Jac(d_{I_1(y)}f^{-1})}{\Jac(d_yf^{-1})}=e^{\pm\epsilon^2\theta_2^n}$, and since 
$\Jac(d_yI_1)=e^{\pm\epsilon \tilde{\theta}^n}$ (by equation \eqref{DeltaAir}),
$$\frac{e^{\phi(\ul{S})}}{e^{\phi(\ul{R})}}=\exp(\pm(\epsilon\theta_2^n+\epsilon^2\theta_2^n+ \epsilon \tilde{\theta}^n))=\exp(\pm\sqrt{\epsilon}\theta_2^n)\text{ (for small enough }\epsilon).$$

\medskip
We are then left to show summable variations. It is enough to show that $|\phi(\ul{R})-\phi(\ul{S})|$ is bounded uniformly whenever $d(\ul{R},\ul{S})\leq e^{-2}$. Indeed,
\begin{equation}\label{philocallybounded}
    \phi(\ul{R})=\log\int_{f^{-1}[V^u(\ul{R})]}\frac{d\overline{m}_{\sigma_R\ul{R}}(y)}{g^{\sigma_R\ul{R}}(y)}=\log\left(e^{\pm\epsilon}\frac{\Vol(f^{-1}[V^u(\ul{R})])}{\Vol(V^u(\sigma_R\ul{R}))}\right).
\end{equation}
By equation \eqref{forpart5}, whenever $d(\ul{R},\ul{S})\leq e^{-2}$, $\Vol(V^u(\ul{R}))=e^{\pm\epsilon}\Vol(V^u(\ul{S})),\Vol(V^u(\sigma_R\ul{R}))=e^{\pm\epsilon}\Vol(V^u(\sigma_R\ul{S}))$. Therefore,
\begin{equation}\label{phibounded}|\phi(\ul{R})-\phi(\ul{S})|\leq4\epsilon.\end{equation}
\end{proof}

\begin{cor}\label{g-Holder}
 $\forall\ul{R}\in \Sig_L$, the function $\rho^{\ul{R}}=\frac{dm_{V^u(\ul{R})}}{d\lambda_{V^u(\ul{R})}}:V^u(\ul{R})\rightarrow [e^{-\epsilon},e^\epsilon]$ is H\"older continuous, with a H\"older constant and exponent uniform in $\ul{R}$. 
\end{cor}
\begin{proof}
From equation \eqref{bananot} 
and the calculations made right after it we get, $$|g^{\ul{R}}(t)-g^{\ul{R}}(s)|\leq e^\epsilon \Vol(V^u(\ul{R}))(2\cdot\epsilon\cdot d(t,s)^\frac{\beta}{4}),\text{ }\forall \ul{R}\in\Sig_L, t,s\in V^u(\ul{R}).$$
Hence $\frac{g^{\ul{R}}}{\Vol(V^u(\ul{R}))}=\frac{1}{\rho^{\ul{R}}}$ is $(2\epsilon e^\epsilon,\frac{\beta}{4})$-H\"older continuous. Since $\rho^{\ul{R}}=e^{\pm\epsilon}$, $\rho^{\ul{R}}$ is $(2\epsilon e^{3\epsilon},\frac{\beta}{4})$-H\"older continuous.
\end{proof}


\subsection{Absolutely Continuous Leaf Measures}\label{acLeafMeasures}
In \textsection \ref{smoothies} we constructed smooth leaf measures with a simple transformation law for the action of $f^{-1}$. However, as explained in section \ref{ProofOutline} (step 3), these measures do not necessarily have a simple transformation for the action of $f$, because of overlaps between  unstable leaves $V^u(\underline{S})$ associated to different $\underline{S}\in\sigma_R^{-1}(\underline{R})$.

 In this section we use the canonical coding (see \textsection \ref{CanoCodi}) to restrict the measures $m_{V^u(\underline{R})}$  to absolutely continuous measures $\mu_{\underline{R}}\ll m_{V^u(\underline{R})}$ which do admit a simple transformation law for the action of $f$ (Proposition \ref{PropMuR}) . 
 This is important for us, because the transformation law in Proposition \ref{PropMuR} implies that $\int_{\Sig_L} \mu_{\underline{R}}dp$ is $f$-invariant, whenever $p$ is a $\phi$-conformal measure.  See the proof of Theorem \ref{theorem1} below.
 
 %
%

\begin{definition}\label{A_R}
Define
\begin{align*}
	\Sig_R:=&\{(S_i')_{i\geq 0}: (S_i')_{i\in\mathbb{Z}}\in \Sig\},\\
	\Sig_R^\circ:=&\{\ul{S}\in\mathcal{R}^\mathbb{N}:\bigcap_{i=0}^\infty f^{-i}[S_i]\neq\varnothing\}.
\end{align*}
For any $\ul{S}\in \Sig_R$, define a local stable leaf $V^s(\ul{S}):=W(S_0)\cap V^s(\ul{u})$ for some (any) $\ul{u}\in \Sigma$ s.t. $(u_i)_{i\geq0}\CAR\ul{S}$ (see Lemma \ref{properness}). For any $R\in\mathcal{R}$, define $$A_R:=\bigcupdot\{V^s(\ul{S}):\ul{S}\in\Sig_R^\circ,S_0=R\}.$$
\end{definition}
As in the remark after \ref{canonicparts}, $\Sig_R^\circ$ is invariant.
\begin{lemma}\label{ARdisjoint}
	Definition \ref{A_R} is proper: the union $\bigcupdot\{V^s(\ul{S}):\ul{S}\in\Sig_R^\circ,S_0=R\}$ is indeed disjoint.
\end{lemma}
\begin{proof}
Let $R\in\mathcal{R}$, and let $\ul{S}^1,\ul{S}^2\in [R]\cap\Sig_R^\circ$. Let $\ul{R}\in [R]\cap\Sig_L^\circ$. Let $a\in \bigcap_{i\geq0}f^{i}[R_{-i}]$, $b_1\in \bigcap_{i\geq0}f^{-i}[S^1_{i}]$, and  $b_2\in \bigcap_{i\geq0}f^{-i}[S^2_{i}]$. Write $x:=[a,b_1]_R$ and $y:=[a,b_2]_R$ (recall Definition \ref{Doomsday}). We claim that if $V^s(\ul{S}^1)\cap V^s(\ul{S}^2)\neq\varnothing$, then $x=y$, thus $\ul{S}^1=(R(f^i(x)))_{i\geq0}= (R(f^i(y)))_{i\geq0} =\ul{S}^2$.

Since $\{x\}=V^u(\ul{R})\cap V^s(\ul{S}^1)$, $\{y\}=V^u(\ul{R})\cap V^s(\ul{S}^2)$, it is enough to show $V^s(\ul{S}^1)\cap V^s(\ul{S}^2)\neq\varnothing\Rightarrow V^s(\ul{S}^1) = V^s(\ul{S}^2)$.

Indeed, since $V^s(\ul{S}^1)$ and $V^s(\ul{S}^2)$ extend over the same window $W(R)$ (Definition \ref{unstablemanifold}), and two stable leaves which span over the same window either coincide or are disjoint (recall Lemma \ref{properness}, and see \cite[Proposition~4.15]{SBO}).
\end{proof}

\begin{claim}\label{ARMeasurability}
	For $S\in\mathcal{R}$, $A_S$ is Borel measurable.
\end{claim}
\begin{proof}
Fix $\ul{S}\in \Sig_L^\circ\cap [S]$. $W^u(\ul{S})=\bigcap_{j\geq0}f^j[S_{-j}]$ is measurable. Let $\tau:\{\ul{S}\}\times \left([S]\cap\Sig_R\right)\rightarrow [S]\cap\Sig_R$ be the projection onto the non-negative coordinates. This map is continuous and one-to-one on its domain. $[S]\cap\Sig_R^\circ=\tau\circ \ul{R}[W^u(\ul{S})]$, where $\ul{R}(x):=(R(f^i(x)))_{i\in\mathbb{Z}}$ is the itinerary of $x$.\footnote{To see this, show double inclusion. Given $\ul{S}^+\in [S]\cap \Sig_R^\circ$, let $\ul{S}^\pm:=\ul{S}\cdot \ul{S}^+$ be an admissible concatenation, and so $\ul{S}^+=\tau(\ul{S}^\pm)$ while $\Sig^\circ\ni\ul{S}^\pm=\ul{R}(\widehat{\pi}(\ul{S}^\pm))$ and $\widehat{\pi}(\ul{S}^\pm)\in W^u(\ul{S})$. For the other inclusion, take $x\in W^s(\ul{S})$, then $\tau\circ\ul{R}(x)\in[S]\cap \Sig_R^\circ$.} $\ul{R}[W^u(\ul{S})]\subseteq \{\ul{S}\}\times [S]\cap\Sig_R$. Thus $[S]\cap \Sig_R^\circ$ is Borel measurable.

Let $\ul{S}^+,\ul{Q}^+\in [S]\bigcap\Sig_R^\circ$ s.t. $\ul{S}^+\neq \ul{Q}^+$
. By Lemma \ref{ARdisjoint}
$V^s(\ul{S}^+)\cap V^s(\ul{Q}^+)=\varnothing$.

Fix a chart $u=\psi_{x_0}^{p^s_0,p^u_0}$ s.t. $Z(u)\supseteq S$. $\forall \ul{S}^+\in[S]\cap\Sig_R^\circ$, let $F_{\ul{S}^+}$ be the representing function of $V^s(\ul{S}^+)$ in $u$. Recall the definition of $W(S)$ in Definition \ref{unstablemanifold}. Consider now the map $\Xi:\left([S]\cap \Sig_R^\circ\right)\times
\Big(\pi_u\circ \psi_{x_0}^{-1}[W(S)]\Big)\rightarrow M$ (where $\pi_u$ is the projection onto the $u$-ccordinates), $(\ul{S}^+,t)\mapsto
\psi_{x_0}\Big((t,F_{\ul{S}^+}(t))\Big)$. This map is continuous, and by the previous paragraph, one-to-one. Also, $\Xi\Big[\Big([S]\cap \Sig_R^\circ\Big)\times
\Big(\pi_u\circ\psi_{x_0}^{-1}[W(S)]\Big)\Big]=A_S$. Continuous injective images of Borel sets are Borel sets, hence $A_S$ is a Borel set.
\end{proof}

\begin{claim}\label{bisligrill}
$\forall \ul{R}\in \Sig_L^\circ$, $m_{\ul{R}}=\mathbb{1}_{A_{R_0}}\cdot m_{V^u(\ul{R})}$.
\end{claim}
\begin{proof}
Recall, $m_{\ul{R}}=\mathbb{1}_{W^u(\ul{R})}\cdot m_{V^u(\ul{R})}$, where $W^u(\ul{R})=\bigcap_{j\geq0}f^j[R_{-j}]$. We will show that $W^u(\ul{R})=V^u(\ul{R})\cap A_{R_0}$. The inclusion $\subseteq$ is easy: Fix $x\in W^u(\ul{R})$ and let $\ul{R}(x):=$ the itinerary of $x$, then $(\ul{R}(x)_i)_{i\geq0}\in\Sig_R^\circ$, and $\{x\}=\{\widehat{\pi}(\ul{R}(x))\}=V^u(\ul{R})\cap V^s((\ul{R}(x)_i)_{i\geq0})$. Now for the other inclusion: Let $x\in V^u(\ul{R})\cap A_{R_0}$, then $x\in V^s(\ul{S}),\ul{S}\in\Sig_R^\circ,S_0=R_0$. Let $y\in W^u(\ul{R})$, $z\in\bigcap_{i\geq0}f^{-i}[S_i]$. One can easily check that $x$ must equal $[y,z]_{R_0}$; whence, by the Markov property, $(\ul{R}(x)_i)_{i\leq0}=\ul{R}$. So $x\in W^u(\ul{R})$.
\end{proof}
\begin{definition}\label{muR}
{\em The extended space of absolutely continuous measures}:
$\forall \ul{R}\in \Sig_L$, $\mu_{\ul{R}}:=\mathbb{1}_{A_{R_0}}\cdot m_{V^u(\ul{R})}$.
\end{definition}

By Claim \ref{bisligrill}, if $\ul{R}\in \Sig^\circ_L$ then $\mu_{\ul{R}}=m_{\ul{R}}$.

\begin{lemma}\label{0.3}
$\forall h\in C(M)$, $\psi_h(\ul{R}):=\mu_{\ul{R}}(h)$ is continuous on $\Sig_L$.
\end{lemma}
\begin{proof}
Fix a partition member $T$, and let $A_T:=\bigcupdot\{V^s(\ul{S}):\bigcap_{i\geq0}f^{-i}[S_i]\neq\varnothing,S_0=T\}$. For every $\ul{R},\ul{S}\in \Sig_L$ s.t. $R_0=S_0=T$, let $\Gamma:V^u(\ul{S})\cap A_T\rightarrow V^u(\ul{R})\cap A_T$ be the holonomy map along the stable leaves. By Pesin's absolute continuity theorem \cite[Theorem~8.6.1]{BP}, $\|\Jac(\Gamma)-1\|\leq K_T\theta^n$ whenever $d(\ul{R},\ul{S})\leq e^{-n}$, where $\theta$ is as defined in Claim \ref{VRlikeVu}, and $K_T$ is a positive constant depending on the partition member $T$. Here $\Jac(\Gamma)$ refers to the Radon-Nikodym derivative of the mapping, and not the standard derivative (as it may not exist).

 \medskip
Let $h\in C(M)$, w.l.o.g. $\|h\|_\infty=1$. $M$ is compact, whence $\epsilon_h(\delta):=\sup\{|h(x)-h(y)|:d(x,y)<\delta\}\xrightarrow[\delta\rightarrow0^+]{}0$. 
Assume $d(\ul{R},\ul{S})=e^{-n}, n\geq 1$. If $x\in V^u(\ul{S})\cap A_T$, then $\exists\ul{Q}\in\Sig_R^\circ\cap[T]$ s.t. $x=\widehat{\pi}(\ul{S}\cdot \ul{Q}),\Gamma(x)=\widehat{\pi}(\ul{R}\cdot \ul{Q})$, where $\ul{S}\cdot \ul{Q},\ul{R}\cdot \ul{Q}$ are the concatenations of two one-sided chains which begin with the same symbol. In addition, $d(\ul{R}\cdot\ul{Q},\ul{S}\cdot\ul{Q})=d(\ul{R},\ul{S})=e^{-n}$. Then by Claim \ref{VRlikeVu}, $d(x,\Gamma(x))\leq3d_{C^0}(V^u(\ul{R}),V^u(\ul{S}))\leq 3K\theta^n$ where $\theta\in(0,1), K>0$ are constants and the factor of $3$ comes from the Lipschitz bound 
of $\ul{R}'\mapsto \text{unique element of }\Big( V^u(\ul{R}')\cap V^s(\ul{Q})\Big)$, $\ul{R}'\in [T]\cap \Sig_L$ (see \cite[Proposition~3.5(3)]{SBO}). Now recall the mapping $I:V^u(\ul{S})\rightarrow V^u(\ul{R})$ from Lemma \ref{0.1}.

\medskip
\textbf{Step 1:} $d(\Gamma(x), I(x))\leq d(\Gamma(x),x) +d(x,I(x))$, and by part 0 in the proof of Lemma \ref{0.1}, $d(x,I(x))\leq 2d_{C^0}(V^u(\ul{R}),V^u(\ul{S}))\leq 2K\theta^n$, whence in total $d(\Gamma(x), I(x))\leq5K\theta^n$.

\medskip
\textbf{Step 2:} Recall $m_{V^u(\ul{R})}=\frac{1}{g^{\ul{R}}}\cdot\overline{m}_{V^u(\ul{R})}$ where $\overline{m}_{V^u(\ul{R})}$ is the induced Riemannian leaf volume of $V^u(\ul{R})$. By part 4 in Lemma \ref{0.1}, $\|g^{\ul{R}}\circ I-g^{\ul{S}}\|\leq\epsilon^\frac{3}{2} \Vol (V^u(\ul{S})) \theta_2^n$ where $\theta_2\in(0,1)$ is a constant. Therefore
$$\|\frac{1}{g^{\ul{R}}\circ I}-\frac{1}{g^{\ul{S}}}\|\leq\|\frac{1}{g^{\ul{R}}}\|\cdot\|\frac{\Vol(V^u(\ul{S}))}{g^{\ul{S}}}\|\cdot\epsilon^\frac{3}{2}\theta_2^n.$$ By Theorem \ref{naturalmeasures},  $\|\frac{1}{g^{\ul{R}}}\|\leq \frac{1}{\Vol(V^u(\ul{R}))}e^{\pm\epsilon}$,$\|\frac{1}{g^{\ul{S}}}\|\leq \frac{1}{\Vol(V^u(\ul{S}))}e^{\pm\epsilon}$. So in total $\|\frac{1}{g^{\ul{R}}\circ I}-\frac{1}{g^{\ul{S}}}\|\leq \frac{1}{\Vol(V^u(\ul{R}))}e^{\pm2\epsilon}\epsilon^\frac{3}{2} \theta_2^n$. By part 5 of Lemma \ref{0.1}, $\frac{\Vol(V^u(\ul{R}))}{\Vol(V^u(\ul{S}))}=e^{\pm\epsilon}$ if $R_0=S_0$. Let $C_{tmp}(T):=\sup\{\frac{1}{\Vol(V^u(\ul{R}'))};R_0'=R_0=S_0=T\}\leq e^{\epsilon}\frac{1}{\Vol(V^u(\ul{R}))}<\infty$, then $\|\frac{1}{g^{\ul{R}}\circ I}-\frac{1}{g^{\ul{S}}}\|\leq C_{tmp}(T)e^{\pm2\epsilon} \epsilon^\frac{3}{2} \theta_2^n\leq C_{tmp}(T)\theta_2^n$ for $\epsilon>0$ small enough.

\medskip
\textbf{Step 3:} By Corollary \ref{g-Holder}, $\rho^{\ul{S}}=\frac{\Vol(V^u(\ul{S}))}{g^{\ul{S}}}$ is $(2\epsilon e^{3\epsilon},\frac{\beta}{4})-$H\"older continuous. Therefore $\frac{1}{g^{\ul{S}}}$ is

$(\frac{1}{\Vol(V^u(\ul{S}))}2\epsilon e^{3\epsilon},\frac{\beta}{4})-$H\"older continuous, whence $(C_{tmp}(T)2\epsilon e^{3\epsilon},\frac{\beta}{4})-$H\"older continuous. Combining this with steps 1 and 2 yields:

\begin{align*}
    \|\frac{1}{g^{\ul{R}}\circ \Gamma}-\frac{1}{g^{\ul{S}}}\|&\leq \|\frac{1}{g^{\ul{R}}\circ I}-\frac{1}{g^{\ul{S}}}\|+\|\frac{1}{g^{\ul{R}}\circ \Gamma}-\frac{1}{g^{\ul{R}}\circ I}\|\leq C_{tmp}(T)\theta_2^n+ C_{tmp}(T)2\epsilon e^{3\epsilon}\cdot d(\Gamma(x),I(x))^\frac{\beta}{4}\\
    &\leq C_{tmp}(T)\theta_2^n+ C_{tmp}(T)2\epsilon e^{3\epsilon}\cdot(5K\theta^n)^\frac{\beta}{4}\leq \tilde{C}_{tmp}(T)\theta_2^n,
\end{align*}
  where $\tilde{C}_{tmp}(T)$ is a global constant of, and $\theta^\frac{\beta}{4}\leq \theta_2<1$. In addition, $\|h-h\circ\Gamma\|\leq \epsilon_h(d(x,\Gamma(x)))\leq\epsilon_h(3K\theta^n)$.

  \medskip
  \textbf{Step 4:} $\Gamma$ is a bijection. So, 

\begin{align*}
    \int\limits_{V^u(\ul{R})\cap A_T}h\frac{1}{g^{\ul{R}}}d\overline{m}_{V^u(\ul{R})}=&\int\limits_{\Gamma[V^u(\ul{S})\cap A_T]}h\frac{1}{g^{\ul{R}}}d\overline{m}_{V^u(\ul{R})}=\int\limits_{V^u(\ul{S})\cap A_T}h\circ\Gamma\frac{1}{g^{\ul{R}}\circ\Gamma}\Jac(\Gamma)d\overline{m}_{V^u(\ul{S})}\\
    =&\int\limits_{V^u(\ul{S})\cap A_T}h\frac{1}{g^{\ul{S}}}d\overline{m}_{V^u(\ul{S})}+\int\limits_{V^u(\ul{S})\cap A_T}(h\circ\Gamma-h)\frac{1}{g^{\ul{R}}\circ\Gamma}\Jac(\Gamma)d\overline{m}_{V^u(\ul{S})}\\ +&\int\limits_{V^u(\ul{S})\cap A_T}h\cdot(\frac{1}{g^{\ul{R}}\circ\Gamma}-\frac{1}{g^{\ul{S}}})\Jac(\Gamma)d\overline{m}_{V^u(\ul{S})}+\int\limits_{V^u(\ul{S})\cap A_T}h\frac{1}{g^{\ul{S}}}(\Jac(\Gamma)-1)d\overline{m}_{V^u(\ul{S})}\\
     =&\int_{V^u(\ul{S})\cap A_T}h\frac{1}{g^{\ul{S}}}d\overline{m}_{V^u(\ul{S})}\pm(\epsilon_h(3K\theta^n)+\|\Jac(\Gamma)\|\Vol(V^u(\ul{S}))\tilde{C}_{tmp}(T)\theta_2^n+\|\Jac(\Gamma)-1\|),
 \end{align*}
 where the last transition used the bound we achieved in step 3. Thus,

  \begin{align*}
      |\mu_{\ul{R}}(h)-\mu_{\ul{S}}(h)|&=|\int_{V^u(\ul{R})\cap A_T}h\frac{1}{g^{\ul{R}}}d\overline{m}_{V^u(\ul{R})}-\int_{V^u(\ul{S})\cap A_T}h\frac{1}{g^{\ul{S}}}d\overline{m}_{V^u(\ul{S})}|\\
      &\leq\epsilon_h(3K\theta^n)+(1+K_{T})\Vol(V^u(\ul{S}))\tilde{C}_{tmp}(T)\theta_2^n+K_{T}\theta^n),
  \end{align*}
where $\Vol(V^u(\ul{S}))$ is bounded uniformly by $\epsilon$. Therefore, $\psi_h$ is uniformly-continuous on cylinders, and so it is continuous on $\Sig_L$.
\end{proof}

\begin{definition}\label{RuelleOpe}
We define the following Ruelle operator on $C(\Sig_L)$:
$$(L_\phi\tilde{\psi})(\ul{R}):=\sum_{\sigma_R\ul{S}=\ul{R}}e^{\phi(\ul{S})}\tilde{\psi}(\ul{S}).$$
\end{definition}
The sum has finitely many terms since $\Sig$ (and thus $\Sig_L$) is locally-compact (see the remark after Definition \ref{Doomsday}), and therefore $\#\{S:R_0\rightarrow S\}$ is finite.

\medskip

The following property is the cornerstone of our approach. It allows us to translate the search for the  measure $p$ which integrates the conditional leaf measures into an $f$-invariant measure, into a fixed point problem for Ruelle's operator on a countable Markov shift,
(see \textsection \ref{lalipop}).
\begin{prop}\label{PropMuR} $\forall h\in C(M),\ul{R}\in \Sig_L$, $\mu_{\ul{R}}(h\circ f)=\sum_{\sigma_R\ul{S}=\ul{R}}\mu_{\ul{S}}(h)e^{\phi(\ul{S})}$.
\end{prop}
\begin{proof}
 Fix any $h\in C(M)$. By Lemma \ref{0.3},
 $\psi_h(\ul{R}):=\mu_{\ul{R}}(h)$ and $\psi_{h\circ f}(\ul{R}):=\mu_{\ul{R}}(h\circ f)$ are
 continuous on $\Sig_L$.
 By Claim \ref{phibar} and Claim \ref{bisligrill}, $\forall \ul{R}\in \Sig_L^\circ$, $\psi_{h\circ f}(\ul{R})=L_\phi \psi_h(\ul{R})$. By the remark after Definition \ref{canonicparts}, $\Sig_L^\circ$ is dense in $\Sig_L$. Therefore, by the continuity of $L_\phi$ ($\phi$ is continuous and $\Sig$ is locally-finite), $\psi_{h\circ f}(\ul{R})=L_\phi \psi_h(\ul{R})$ for $\ul{R}\in \Sig_L\setminus \Sig_L^\circ$ as well.
\end{proof}
\section{The Leaf Condition}\label{leafocondo}
In the previous sections we constructed a family of leaf measures $\{\mu_{\underline{R}}\}_{\underline{R}\in\Sig_L}$ with a simple transformation law for the action of $f$ (Proposition \ref{PropMuR}). Our task now is to construct a measure $p$ on $\Sig_L$ so that $\int_{\Sig_L}\mu_{\underline{R}}dp$ is $f$-invariant. We will see later, that because of the transformation rule in Proposition \ref{PropMuR}, it is sufficient to know that $p$ is {\em  $\phi$-conformal:} $L_\phi^\ast p=p$. 

To show that $p$ exists, we will use the generalized Ruelle's Perron-Frobenius Theorem for countable Markov shifts  \cite{SarigNR}. This theorem is stated for irreducible (i.e. topologically transitive) countable Markov shifts. $\Sig_L$ is not necessarily irreducible. In this section we find an irreducible component of $\Sig_L$ where $\mu_{\underline{R}}\neq 0$. 
To do this we will use the following leaf condition:


\begin{definition}\label{leafcondo}
We say that {\em the leaf condition} is satisfied, if there exists an unstable leaf of maximal dimension which gives $\HWT_\chi$ a positive leaf volume. We say that the leaf condition is satisfied for a measurable set $A\in\mathcal{B}$, if there exists an unstable leaf of maximal dimension which gives $A$ a positive leaf volume.
\end{definition}
\noindent
(Note that this is weaker than \eqref{Leaf-Condition}.)

A similar ``leaf condition" was introduced earlier by Climenhaga, Dolgopyat, and Pesin in \cite{PesinVaughnDima}.
\begin{definition}
$$\Sig^{\osharp}:=\bigcupdot_{[S]\subseteq \Sig}\left([S]\cap \Sig_L^\#\right)\times\left([S]\cap \Sig_R^\circ\right).$$
\end{definition}
Notice, $\Sig^\circ\subseteq \Sig^{\osharp}\subseteq \Sig^\#$, and the invariance of $\Sig_R^\circ,\Sig_L^\#$ implies the invariance of $\Sig^{\osharp}$.

\noindent Recall the measures $\{\mu_{\ul{R}}\}_{\ul{R}\in\Sig_L}$ from Definition \ref{muR}.

\begin{lemma}\label{canonicassumption}
If there exists an unstable leaf of maximal dimension, $V^u$, whose Riemannian volume gives $\HWT_\chi$ a positive measure, then there exist a maximal irreducible component $\langle S\rangle^\mathbb{Z}\cap \Sig$ and a periodic chain $\ul{S}\in \Sig_L\cap\langle S\rangle^{-\mathbb{N}}$ s.t. $\mu_{\ul{S}}(\widehat{\pi}[\Sig^{\osharp}\cap\langle S\rangle^\mathbb{Z}])>0$.
\end{lemma}
\begin{proof}
$V^u$ gives a positive volume to $\bigcupdot\mathcal{R}$, which is a countable union. Therefore, $\exists R\in\mathcal{R}$ s.t. $V^u\cap R$ has a positive leaf volume in $V^u$. 
Since $\forall y\in R\cap V^u$, $\dim(V^u)=\dim V^u(\ul{R}(y))$, 
$\exists \ul{R}\in\Sig^\circ\cap[R]$ s.t. $m_{V^u(\ul{R})}(R)>0$. For every $ x\in R$, $\exists S\in\mathcal{R}$ s.t. $R(f^i(x))=S$ for infinitely often $i\geq0$. Since there is only a countable number of finite cylinders, $\exists l\geq2$ and a cylinder $[R,S_1,...,S_{l-2},S]$ s.t. $\{x\in R:\ul{R}(x)\in [R,S_1,...,S_{l-2},S], \#\{i:R(f^i(x))=S\}=\infty\}$ has a positive leaf volume in $V^u(\ul{R})$. Let $\ul{R}'$ be the admissible concatenation $\ul{R}\cdot (R,S_1,...,S_{l-2},S)\in \Sig_L$. Then since $f^{-l}[V^u(\ul{R}')]\supseteq \{x\in V^u(\ul{R})\cap R0:\ul{R}(x)\in [R,S_1,...,S_{l-2},S] \}$, $V^u(\ul{R}')$ gives a positive leaf volume to $\{x\in S: \#\{i\geq0:R(f^i(x))=S\}=\infty\}$. Since $S$ is a recurring symbol in the future of any of these points, $\exists \ul{S}\in [S]\subseteq \Sig_L$ which is periodic.

The holonomy map along stable leaves $\Gamma:V^u(\ul{R}')\cap \{x\in S: \#\{i\geq0:R(f^i(x))=S\}=\infty\}\rightarrow V^u(\ul{S})$ is defined by $\Gamma(x)=\widehat{\pi}((S_i)_{i\leq0}\cdot(R(f^i(x))_{i\geq0}))\in\widehat{\pi}[\Sig^{\osharp}\cap\langle S\rangle^\mathbb{Z}]$, where $\cdot$ denotes concatenation. By Pesin's absolute continuity theorem \cite[Theorem~8.6.1]{BP}, $m_{V^u(\ul{S})}(\Gamma[\{x\in S\cap V^u(\ul{R}'): \#\{i\geq0:R(f^i(x))=S\}=\infty\}])>0$. 
Thus
\begin{equation}\label{forFootnote8}
\mu_{\ul{S}}(\widehat{\pi}[\Sig^{\osharp}\cap\langle S\rangle^\mathbb{Z}])\geq \mu_{\ul{S}}(\bigcupdot\{V^s(\ul{S}^+):\ul{S}^+\in\Sig_R^\circ\cap[S],\#\{i\geq0: S_i^+=S\}=\infty)\})>0.
\end{equation}
\end{proof}
\noindent\textbf{Remark:}
\begin{enumerate}
 	\item Let $\ul{S}^\pm$ be the periodic extension of $\ul{S}$ 	to $\Sig$, then $p':=\widehat{\pi}(\ul{S}^\pm)$ is a 	periodic point. It follows that $H_\chi(p')\supseteq 	\widehat{\pi}[\langle S\rangle^\mathbb{Z}\cap \Sig^{\osharp}]	$, whence $\mu_{\ul{S}}	(H_\chi(p'))>0$ (see Definition \ref{homoclinicclass}).
	\item From this point on, we focus on one ergodic homoclinic 	class, and constructing an SRB measure on it. 
	By Lemma \ref{canonicassumption}, we may restrict our attention to a maximal irreducible component
. We therefore assume w.l.o.g. that $\Sig_L,\Sig$ are irreducible.
\end{enumerate}

\begin{prop}\label{psitilda}
$\psi:\Sig \rightarrow \mathbb{R}^+\cup\{0\}$, $\psi(\ul{R}):=\mu_{\ul{R}}(1)$ is a non-negative continuous eigenfunction of $L_\phi$ with eigenvalue $1$. If there exists a chain $\ul{R}\in \Sig_L$ s.t. $\mu_{\ul{R}}(1)>0$, then $\psi$ is also positive everywhere.
\end{prop}
\begin{proof}

 By Lemma \ref{0.3}, $\psi=\psi_1$ is continuous, and by Proposition \ref{PropMuR} $L_\phi\psi=\psi$. 
 If there exists a chain $\ul{R}\in \Sig_L$ s.t. $\mu_{\ul{R}}(1)>0$, then 
 by continuity, $\exists n_{\ul{R}}$ s.t. $d(\ul{R},\ul{S})\leq e^{-n_{\ul{R}}}\Rightarrow \psi(\ul{S})>0$. This in turn means that $\forall m\geq0$, $\psi(\sigma_R^m\ul{S})>0$, since $L_\phi\psi=\psi$. By irreducibility, $\forall \ul{S}\in \Sig_L$ $\exists \ul{S}'\in\Sig_L$ and $\exists m_{\ul{S}}\geq0$ s.t. $d(\ul{S}',\ul{R})\leq e^{-n_{\ul{R}}}$ and $\sigma_R^{m_{\ul{S}}}\ul{S}'=\ul{S}$. Therefore $\psi(\ul{S})>0$,
and it follows that $\psi$ is positive everywhere
.

\end{proof}

\section{Recurrence and the Gurevich Pressure}\label{lalipop}\text{ }

We continue with the construction of the $\phi$-conformal measure $p$. As explained in the previous section, the plan is to use the generalized Ruelle's Perron-Frobenius Theorem for countable Markov shifts from \cite{SarigNR}. This theorem says that a conservative $\phi$-conformal measure $p$ exists if and only if two conditions hold: $\phi$ is recurrent, and $\phi$ has zero Gurevich pressure (see below). In this section we check these conditions, using the leaf condition. 

%

We begin by recalling the definition of the Gurevich pressure. For a potential $\zeta:\Sig_L\rightarrow\mathbb{R}$ with summable variations (see Claim \ref{VolHol}), the {\em partition functions} are $Z_n(\zeta,R):=\sum\limits_{\ul{S}\in \Sig_L\cap[R],\sigma_R^n\ul{S}=\ul{S}}e^{\sum_{k=0}^{n-1}\zeta(\sigma_R^k\ul{S})}$ (the sum has finitely many terms, since $\Sig_L$ is locally compact). The {\em Gurevich pressure} of the potential $\zeta$ is $P_G(\zeta):=\limsup\limits_{n\rightarrow\infty}\frac{1}{n}\log Z_n(\zeta,R)\in(-\infty,\infty]$. When $\Sig_L$ is transitive, the limit is independent of the choice of the symbol $R$ (see \cite[Proposition~3.2]{SarigTDF}). We say that the potential $\zeta$ is {\em recurrent} if $\sum_{n\geq1}e^{-nP_G(\zeta)}Z_n(\zeta,R)=\infty$ for some symbol (in this case the sum diverges for all $R$, see \cite[Corollary~3.1]{SarigTDF}). We say that the potential $\zeta$ is {\em positive recurrent} if it is recurrent, and $\sum_{n\geq1}n\cdot e^{-nP_G(\zeta)} \sum\limits_{\overset{\ul{S}\in \Sig_L\cap[R]\text{, s.t. }\sigma_R^n\ul{S}=\ul{S},}{\text{ and }\forall 0< l<n, S_{-l}\neq R}}e^{\sum_{k=0}^{n-1}\zeta(\sigma_R^k\ul{S})}<\infty$ for some (any) symbol $R$. Again, this property turns out to be independent of $R$. For more details, see \cite[\textsection~3.1.3]{SarigTDF}. 
For a detailed review of the properties of recurrent potentials, see \cite{SarigTDF}.

\noindent Recall the (one-sided) potential $\phi$ from Theorem \ref{naturalmeasures}.

\begin{definition}\label{TGP} The {\em geometric potential} is the (two-sided) potential
$$\varphi:\Sig\rightarrow [-d\cdot\log M_f,d\cdot\log M_f],\varphi(\ul{R}):=\log\Jac(d_{\widehat{\pi}(\ul{R})}f^{-1}|_{T_{\widehat{\pi}(\ul{R})}V^u(\ul{R})}).$$
\end{definition}

\subsection{Upper Bound for the Gurevich Pressure}

\begin{theorem}\label{ruellemargulis}
$P_G(\phi)\leq0$, where $\phi$ is defined in Theorem \ref{naturalmeasures}.
\end{theorem}
\begin{proof}
Let $\phi_n:=\sum_{k=0}^{n-1}\phi\circ\sigma_R^k$, $n\geq1$. Notice:
 \begin{align*}
     e^{\phi(\ul{R})}&\cdot m_{V^u(\ul{R})}=m_{V^u(\sigma_R\ul{R})}\circ f^{-1}|_{V^u(\ul{R})},\\
     \Rightarrow e^{\phi_n(\ul{R})}&\cdot m_{V^u(\ul{R})}=m_{V^u(\sigma_R^n\ul{R})}\circ f^{-n}|_{V^u(\ul{R})},\\
     \Rightarrow e^{\phi_n(\ul{R})}&=m_{V^u(\ul{R})}(f^{-n}[V^u(\ul{R})])\text{,  if }\sigma_R^n\ul{R}=\ul{R}.
 \end{align*}

Let $P_{n,L}:=$ the $n$-periodic points in $\Sig_L$, and let $P_{n}:=$ the $n$-periodic points in $\Sig$. Then there is a natural bijection $E:P_{n,L}\leftrightarrow P_n$. Define $\Delta_n(\ul{R}):=\phi_n(\ul{R})-\varphi_n(E(\ul{R}))$, where $\varphi_n(\ul{S}):=\sum_{k=0}^{n-1}\varphi(\sigma^{-k}\ul{S})$.

 \medskip
 \underline{Part 1:} $\Delta_n\rvert_{P_{n,L}}(\cdot)$ is bounded uniformly in $n$.

 \medskip
 \underline{Proof:} Fix some $\ul{R}\in P_{n,L}$, and write $x:=\widehat{\pi}(E(\ul{R}))$. Denote by $\lambda_{V^u(\ul{R})}$ and $\lambda_{f^{-n}[V^u(\ul{R})]}$ the (non-normalized) Riemannian leaf volume on $V^u(\ul{R})$ and $f^{-n}[V^u(\ul{R})] $, then
 \begin{align}\label{forAppend}
     e^{\Delta_n(\ul{R})}=&m_{V^u(\ul{R})}(f^{-n}[V^u(\ul{R})])\cdot\frac{1}{\Jac(d_xf^{-n}|_{T_xV^u(\ul{R})})}\nonumber\\
     =&\int_{f^{-n}[V^u(\ul{R})]}\frac{1}{g^{\ul{R}}(t)}d\lambda _{f^{-n}[V^u(\ul{R})]}(t)\cdot\frac{1}{\Jac(d_xf^{-n}|_{T_xV^u(\ul{R})})}\text{ }(g^{\ul{R}}\text{ is the limit function defined in equation }\eqref{JakeAmir})\nonumber\\
     =&\int_{V^u(\ul{R})}\frac{1}{g^{\ul{R}}(f^{-n}(t))}\cdot\Jac(d_tf^{-n}|_{T_tV^u(\ul{R})})d\lambda _{V^u(\ul{R})}(t)\cdot\frac{1}{\Jac(d_xf^{-n}|_{T_xV^u(\ul{R})})}\nonumber\\
     =&\int_{V^u(\ul{R})}\frac{1}{g^{\ul{R}}(f^{-n}(t))}\cdot\frac{\Jac(d_tf^{-n}|_{T_tV^u(\ul{R})})}{\Jac(d_xf^{-n}|_{T_xV^u(\ul{R})})}d\lambda _{V^u(\ul{R})}(t).
 \end{align}
We saw in the proof of Theorem \ref{naturalmeasures} that $g^{\ul{R}}=\Vol(V^u(\ul{R}))e^{\pm\epsilon}$. Equation \eqref{forholderphi} (with $n=0$) implies that $\frac{\Jac(d_tf^{-n}|_{T_tV^u(\ul{R})})}{\Jac(d_xf^{-n}|_{T_xV^u(\ul{R})})}=e^{\pm\epsilon}$. So $e^{\Delta_n(\ul{R})}=e^{\pm2\epsilon}$, and therefore $|\Delta_n|\leq2\epsilon$ for all $n$.

 \medskip
 \underline{Part 2:} By \cite[Proposition~6.1]{SBO}, and the $C^{1+\beta}$ regularity of $f$, $\varphi$ is H\"older continuous. Therefore, by Sinai's theorem (in its version for countable Markov shifts, as in \cite{Daon}), there exists a potential $\varphi^-:\Sig\rightarrow\mathbb{R}$, which is bounded on cylinders, and such that $(R_i=S_i, \forall i\leq0)\Rightarrow \varphi^-(\ul{R})=\varphi^-(\ul{S})$; and there exists a bounded and H\"older continuous function $A:\Sig\rightarrow \mathbb{R}$ such that $$\varphi=\varphi^-+A\circ \sigma^{-1}-A.$$
 In this case we say that $\varphi-\varphi^-$ is a coboundary. In particular, this formula implies that $\varphi^{-}$ is H\"older continuous. It follows that,
 $$\|\varphi_n-\varphi^-_n\|_{\infty}\leq 2\|A\|_\infty.$$ $\varphi^-$ can be naturally identified with a potential on $\Sig_L$, s.t. $\|\phi_n-\varphi^-_n\|_{P_{n,L},\infty}\leq 2(\|A\|_\infty+\epsilon)<\infty$, $\forall n\geq1$.

 \medskip
 \underline{Part 3:} Fix some symbol $R$. Notice that since $A$ and $\varphi$ are bounded, $\varphi^-$ is bounded. Bounded H\"older continuous potentials satisfy the variational principle: $P_G(\varphi^-)=\sup\{h_{\nu}(\sigma_R)+\int \varphi^-d\nu:\nu\text{ is an inv. prob. on }\Sig_L\}$ (see \cite[Theorem~4.4]{SarigTDF}). So:
 \begin{align}\label{admitarrete}
     P_G(\phi)=&\limsup\frac{1}{n}\log Z_n(\phi,R)=\limsup\frac{1}{n}\log Z_n(\varphi^-,R)\text{ }(\because\text{by parts 1 and 2})\nonumber\\
     =&\sup\{h_\nu(\sigma_R)+\int\varphi^-d\nu:\nu\text{ is an inv. prob. on }\Sig_L\}\text{ }(\because \text{variational principle})\nonumber\\
     \stackrel{(1)}{=}&\sup\{h_{\nu'}(\sigma^{-1})+\int\varphi^-d\nu':\nu'\text{ is an inv. prob. on }\Sig\}\nonumber\\
     =&\sup\{h_{\nu'}(\sigma^{-1})+\int\varphi d\nu':\nu'\text{ is an inv. prob. on }\Sig\} \stackrel{(2)}{\leq}0\text{ }(\because\varphi-\varphi^-\text{ is a coboundary}).\nonumber
 \end{align}
\noindent$(1)$ is by the entropy preserving natural bijection between shift invariant measures on $\Sig$ and on $\Sig_L$.
$(2)$ is because $\widehat{\pi}$ is finite-to-1 on $\Sig^\#$, a set of full measure w.r.t. any invariant probability measure, whence $\nu'\mapsto\nu'\circ \widehat{\pi}^{-1}$ preserves entropy. Each such measure is an invariant probability Borel measure for $f:M\rightarrow M$, 
so $\int\varphi d\nu'=\int \log\Jac(d_x f^{-1}|_{H^u(x)})d(\nu'\circ\widehat{\pi}^{-1})(x)\leq -h_{\nu'\circ\widehat{\pi}^{-1}}(f)= -h_{\nu'}(\sigma^{-1}) $ by the Ruelle-Margulis inequality and since $h_{\nu'}(\sigma)= h_{\nu'}(\sigma^{-1})$.
\end{proof}


\subsection{Recurrence and the Leaf Condition}

\begin{theorem}\label{PG0}
If $\exists \ul{R}\in\Sig_L$ s.t. $\mu_{\ul{R}}(1)>0$, then $P_G(\phi)=0$ and $\phi$ is recurrent.
\end{theorem}
\begin{proof}
By Theorem \ref{ruellemargulis} $P_G(\phi)\leq0$. Therefore, it is enough to show $\sum_{n\geq1}Z_n(\phi,R)=\infty$ for some (any) symbol $R$ s.t. $[R]\subseteq \Sig_L$.\footnote{It implies that $P_G(\phi)\geq0$, therefore $P_G(\phi)=0$ and so $\sum_{n\geq1}Z_n(\phi,R)e^{-nP_G(\phi)}=\sum_{n\geq1}Z_n(\phi,R)=\infty $.} 
By Lemma \ref{canonicassumption} and equation \eqref{forFootnote8} in it, we may assume w.l.o.g. that  $\mu_{\ul{R}}(\bigcupdot\{V^s(\ul{S}):\ul{S}\in\Sig_R^\circ\cap[R_0], \#\{i\geq0:S_i=R_0\}=\infty\})>0$.

Define $A_n:=\bigcup_{\underline{W}=R_0,W_1,...,W_{n-2},R_0}f^{-n}[V^u(\ul{R}\cdot \ul{W})]$, where the ``$\cdot$" product denotes an admissible concatenation. Notice, the inclusion $$V^u(\ul{R})\cap\left(\bigcupdot\{V^s(\ul{S}):\ul{S}\in\Sig_R^\#\cap[R_0], \#\{i\geq0:S_i=R_0\}=\infty\}\right)\subseteq\limsup A_n=\bigcap_{N=1}^\infty\bigcup_{n=N}^\infty A_n$$ is given naturally by the coding of each point in the LHS set. It follows that $\mu_{\ul{R}}(\limsup A_n)>0$, and in particular,
\begin{equation}\label{ForBlue}
    m_{V^u(\ul{R})}(\limsup A_n)>0.
\end{equation}
Therefore, by the Borel-Cantelli lemma,
$$\infty=\sum_{n\geq1}m_{V^u(\ul{R})}(A_n)\leq \sum_{n\geq1}\sum_{\underline{W}=R_0,W_1,...,W_{n-1},R_0}m_{V^u(\ul{R})}(f^{-n}[V^u(\ul{R}\cdot\ul{W})])=\sum_{n\geq1}\sum_{\underline{W}=R_0,W_1,...,W_{n-1	},R_0}e^{\phi_n(\ul{R}\cdot\ul{W})},$$
where the last equality is by Theorem \ref{naturalmeasures}. For every $n\geq1$, for every $\ul{W}=R_0,W_1,...,W_{n-1},R_0$, write $\ul{R}^{\ul{W}}$ for the periodic concatenation of $\ul{W}$ to itself. It follows that $d(\ul{R}\cdot\ul{W},\ul{R}^{\ul{W}})\leq e^{-n}$. Therefore, by Claim \ref{VolHol}, and since $\phi$ is bounded on $[R_0]$,\footnote{By equation \eqref{phibounded},  $\phi$ is bounded on cylinders of length $2$, and there are only finitely many such cylinders contained in $[R_0]$ by the local-compactness of $\Sig_L$.} $\exists C>0,n_0\in\mathbb{N}$ s.t. $\forall n\geq n_0$, $e^{\phi_n(\ul{R}\cdot\ul{W})}=C^{\pm1}e^{\phi_n(\ul{R}^{\ul{W}})}$. We get
$$\infty\leq C^{\pm1}\sum_{n\geq1}\sum_{\underline{W}=R_0,W_1,...,W_{n-1},R_0}e^{\phi_n(\ul{R}^{\ul{W}})}=C^{\pm1}\sum_{n\geq1}Z_n(\phi,R_0).$$
\end{proof}

\section{Existence of a Hyperbolic SRB Measure}\label{existenceanduniqueness}
In the previous sections we constructed a space of absolutely continuous leaf measures which are carried by hyperbolic points, and which have a specified transformation law with good continuity properties. In \textsection \ref{leafocondo} we used the leaf condition in order to extract a maximal irreducible component where no leaf measures are trivial. In \textsection \ref{lalipop} we used the leaf condition to check that $\phi$ is recurrent, and has zero Gurevich pressure.
We can now apply the generalized Ruelle'e Perron-Frobenius Theorem of \cite{SarigNR} and obtain a conservative  $\phi$-conformal measure $p$ on $\Sig_L$. 

Let
$
\mu:=\int_{\Sig_L}\mu_{\underline{R}}dp.
$
In this section we show that subject to the Leaf condition \eqref{Leaf-Condition}, $\mu$ is a finite hyperbolic measure with absolutely continuous conditionals on unstable leaves. 
We remark that now  is the first time in the proof we are using  the leaf condition on $\HWT^{\mathrm{PR}}_\chi$ (as in \eqref{Leaf-Condition}), and not just the leaf condition on $\HWT_\chi$ (as in Definition \ref{leafcondo}). This is needed to guarantee that $\mu$ is finite.


\subsection{$\sigma$-Finite Measures with Absolutely Continuous Conditional Measures on Unstable Leaves}
\begin{definition}\label{ipadpro} Let $(X,\mathcal{B}',\nu)$ be a measure space. Let $T:X\rightarrow X$ be a measurable transformation.
\begin{enumerate}
    \item A measurable set $W\in\mathcal{B}'$ is called {\em wandering} if $\{T^{-n}[W]\}_{n\geq0}$ are pairwise disjoint modulo $\nu$.
    \item $\nu$ is called {\em conservative} if $\nu$ gives every wandering set a measure 0.
    \item  $\nu$ is called {\em non-singular} if $\nu\sim\nu\circ T^{-1}$. In this case, $(X,\mathcal{B}',\nu,T)$ is said to be a {\em non-singular transformation}.
\end{enumerate}
\end{definition}
Halmos' recurrence theorem (\cite[\textsection~1.1.1]{aaronson}) states that if $(X,\mathcal{B}',\nu,T)$ is a non-singular transformation, then $\nu$ is conservative if and only if $\sum_{n\geq0}\mathbb{1}_E\circ T^n=\infty$ $\nu$-a.e. on $E$ for every $E\in \mathcal{B}'$ s.t. $\nu(E)>0$ (i.e. the Poincar\'{e} recurrence theorem holds: in every positive measure set, almost every point returns to it infinitely many times). Every invariant probability measure is conservative. We work in the context of non-singular transformations, and use the characterization of conservativity by Halmos as the definition. Notice that if $(X,\mathcal{B}',\nu,T)$ is an invertible transformation, then any wandering set for $T$ is a wandering set for $T^{-1}$.

\begin{theorem}\label{theorem1} Assume there exists a maximal dimension unstable leaf $V^u$ which gives $\HWT_\chi$ a positive leaf volume (for some $\chi>0$). Then, there exists a hyperbolic periodic point $q$ s.t. $H_\chi(q)$ carries an ergodic, conservative, $\sigma$-finite, and $f$-invariant measure $\mu$ with absolutely continuous conditional measures on unstable leaves. This measure is finite if and only if $\phi$ is positive recurrent, and in this case it is an SRB  measure. 
\end{theorem}
\begin{proof} By Lemma \ref{canonicassumption}, we may assume w.l.o.g. that $\Sig_L$ is irreducible, and that there exists a periodic chain $\ul{R}\in\Sig_L$ s.t. $\mu_{\ul{R}}(1)>0$. Then, by Theorem \ref{PG0}, $P_G(\phi)=0$ and
$\phi$ is recurrent. By \cite[Theorem~1]{SarigNR}, there exists a unique (up to normalization) $\phi$-conformal measure $p$ on $\Sig_L$ (that is, $L_\phi^*p=e^{P_G(\phi)}p=p$, where $L_\phi^*$ is the dual operator of $L_\phi$); and $p$ is ergodic, conservative, non-singular, and finite on cylinders.\footnote{The condition of topological mixing can be assumed w.l.o.g. because of the Spectral Decomposition theorem for topological Markov shifts (see \cite[Theorem~2.5]{SarigTDFSymposium})
.} Define $$\mu:=\int\limits_{\Sig_L}\mu_{\ul{R}}dp(\ul{R}),$$
where $\{\mu_{\ul{R}}\}_{\ul{R}\in \Sig_L}$ is the extended space of absolutely continuous measures on local unstable leaves from Definition \ref{muR}. $\mu$ is not the zero measure, since $\mu(1)=\int\mu_{\ul{R}}(1)dp=\int\psi(\ul{R})dp$,
 where $\psi$ is a continuous positive function (see Proposition \ref{psitilda}); so $\mu(1)>0$. The restriction of $\mu$ to $R\in \mathcal{R}$ is finite (see Claim 3 below), and is equal up to a normalization to a probability measure whose conditional measures on local unstable leaves are absolutely continuous.

\medskip
\underline{Claim 1:} $\mu\circ f^{-1}=\mu$.

\medskip
\underline{Proof:} By Proposition \ref{PropMuR}, $\mu_{\ul{R}}\circ f^{-1}=\sum_{\sigma_R\ul{S}=\ul{R}}e^{\phi(\ul{S})}\mu_{\ul{S}}$, whence $$\mu\circ f^{-1}=\int \mu_{\ul{R}}\circ f^{-1}dp=\int \sum_{\sigma_R\ul{S}=\ul{R}}e^{\phi(\ul{S})}\mu_{\ul{S}}dp=\int \mu_{\ul{R}}dp=\mu,$$
where the penultimate equality is because $L_\phi^*p=e^{P_G(\phi)}p=p$.

\medskip
\underline{Claim 2:} $\mu$ is finite iff $\phi$ is positive-recurrent.

\medskip
\underline{Proof:} As in the beginning of the proof, write $\mu(1)=p(\psi)$, where $\psi$ is a positive continuous eigenfunction of $L_\phi$ with eigenvalue $1$. Then $\psi$ must be the unique (up to normalization) harmonic function associated with $p$ (see \cite[Theorem~3.4]{SarigTDF},\cite[Theorem~1]{SarigNR}). Then, by \cite[Proposition~3.5]{SarigTDF}, $p(\psi)<\infty$ iff $\phi$ is positive-recurrent.

\medskip
\underline{Claim 3:} $\mu$ is conservative. The first item in the list below also shows that $\mu$ is finite on Pesin level sets (see Definition \ref{PesinLevelSets} below).

\medskip
\underline{Proof:}
Let $A\subseteq M$ be a measurable set s.t. $\mu(A)>0$. Write $\mathcal{R}=\{R^{(i)}\}_{i\in \mathbb{N}}$, $k_i:=\cup\{V^u(\ul{R})\cap A_{R^{(i)}}:\ul{R}\in [R^{(i)}]\cap \Sig_L^\#\}$ (measurability of $k_i$ can be shown similarly to the measurability of $A_{R^{(i)}}$, see Claim \ref{ARMeasurability}). It is clear that $\mu$ is carried by $\bigcup_{i\geq0}k_i$, since $p$ is conservative and carried by $\Sig_L^\#$. We show the following three steps to complete the proof:
\begin{itemize}
    \item \underline{$\forall i\in\mathbb{N}$, $\mu(k_i)<\infty$}:  If $\ul{S},\ul{R}\in\Sig^\#_L,R_0=R^{(i)}$ and $(V^u(\ul{R})\cap A_{R^{(i)}})\cap(V^u(\ul{S})\cap A_{S_0})\neq\varnothing$, then $\exists x$ s.t. $x=\widehat{\pi}(\ul{R}^\pm)=\widehat{\pi}(\ul{S}^\pm)$ where $\ul{R}^\pm,\ul{S}^\pm\in\Sig^\#$ and $S^\pm_0=S_0,R^\pm_0=R_0=R^{(i)}$. Therefore $S_0\in\{S\subseteq Z(v):Z(u)\cap Z(v)\neq\varnothing, Z(u)\supseteq R^{(i)}\}$, and this collection is finite by the local finiteness of $\mathcal{Z}$ and its refinement. In addition, $p$ is finite on cylinders. Thus, $\mu(k_i)\leq \sum_{S\subseteq Z(v):Z(u)\cap Z(v)\neq\varnothing, Z(u)\supseteq R^{(i)}}p([S])<\infty$.
    \item \underline{$\forall i\geq0$, $\mu$-a.e. $x\in A\cap k_i$ $\exists n_j\uparrow\infty$ s.t. $f^{-n_j}(x)\in A\cap k_i$}: For every $i\geq0$, $p$-a.e. $\ul{R}\in[R^{(i)}]$ returns to $[R^{(i)}]$ infinitely often with iterations of $\sigma_R$ since $p$ is conservative. For any such chain and $n\geq0$, $f^{-n}[V^u(\ul{R})\cap A_{R_0}]\subseteq V^u(\sigma_R^n\ul{R})\cap A_{R_{-n}}$.\footnote{If $x=\widehat{\pi}(\ul{R}\cdot\ul{S})$, where the dot means an admissible concatenation, and $\ul{S}\in\Sig_R^\circ$, then $f^{-n}(x)=\widehat{\pi}(\sigma_R^n\ul{R}\cdot((R_{-n},...,R_0)\cdot\ul{S}))$ and $(R_{-n},...,R_0)\cdot\ul{S}\in\Sig_R^\circ$ (by the Markov property); in addition $\ul{R}\in \Sig_L^\#\Rightarrow \sigma_R^n\ul{R}\in\Sig_L^\#$.} Therefore, $\mu$-a.e. point in $k_i$ returns to $k_i$ infinitely often with iterations  of $f^{-1}$. The first return map to $k_i$, $\overline{f}$, is well defined. If $\mu(k_i\cap A)>0$, then $\mu|_{k_i}(\cdot):=\frac{\mu(k_i\cap \cdot)}{\mu(k_i)}$ is $\overline{f}$-invariant, and finite by the first step. Therefore, by the Poincar\'{e} recurrence theorem, $\mu$-a.e. point in $k_i\cap A$ returns to $k_i\cap A$ infinitely often with iterations of $f^{-1}$.
    \item \underline{$\mu$ is conservative}: Assume for contradiction that $B:=\{x\in A:\nexists n\geq0, f^n(x)\in A\}$ has a positive $\mu$ measure. $B=\bigcup_{i\geq0}(B\cap k_i)$ mod $\mu$. Therefore, by step 2, $\mu$-a.e. point in $B$ returns to it infinitely often with iterations of $f^{-1}$. Take any such recurrent point $x$, w.l.o.g. $x,f^{-n}(x)\in B\subseteq A$, whence $f^n(f^{-n}(x))=x\in A$ - a contradiction to the definition of $B$! So $\mu(B)=0$. Therefore, almost every point in $A$ returns to it infinitely often with iterations of $f$ (recall the remark following Definition \ref{ipadpro}).
\end{itemize}

\medskip
\underline{Claim 4:} $\mu$ is carried by $H_\chi(q)\subseteq \HWT_\chi$, where $q$ is a hyperbolic periodic point.

\medskip
\underline{Proof:} Since $p$ is conservative, $\mu$ is carried by a union of sets which carry $\{\mu_{\ul{R}}\}_{\ul{R}\in\Sig_L^\#}$. Each such measure is carried by $V^u(\ul{R})\cap A_{R_0}\subseteq \widehat{\pi}[\Sig^\#]$. Let $\ul{R}'\in\Sig$ be any periodic chain, and define $q:=\widehat{\pi}(\ul{R}')$. Then, by irreducibility, $\widehat{\pi}[\Sig^\#]\subseteq H_\chi(q)$.

\medskip
\underline{Claim 5:} $\mu$ is ergodic.

\medskip
\underline{Proof:} It follows directly from Proposition \ref{GeneralizedEntropyFormula} below, since $p$ is ergodic.
\end{proof}


\begin{prop}\label{GeneralizedEntropyFormula} Under the assumptions and notations of Theorem \ref{theorem1}, the measure $\mu$ 
is proportional to $\widehat{\psi\cdot p}\circ \widehat{\pi}^{-1}$, where $\widehat{\psi\cdot p}$ is the unique shift-invariant extension of $\psi\cdot p$ to $\Sig$ such that $\widehat{\psi\cdot p}\circ \tau^{-1}=\psi\cdot p$, and $\tau$ is the projection to the non-positive coordinates.\end{prop}
\begin{proof}
As explained in claim 1 of Theorem \ref{theorem1}, $\mu_{\ul{R}}(1)=\psi(\ul{R})$ by the uniqueness of the continuous and positive $\phi$-harmonic function on $\Sig_L$ ($\psi$ is only determined up to scaling, so we choose the version $\psi(\ul{R})=\mu_{\ul{R}}(1)$). Write $\forall\ul{R}\in\Sig^\#_L$, $p_{\ul{R}}:=\frac{\mu_{\ul{R}}}{\psi(\ul{R})}=\frac{\mu_{\ul{R}}}{\mu_{\ul{R}}(1)}$, a probability measure on $V^u(\ul{R})$, which is absolutely continuous w.r.t. its leaf volume. One can check easily that $\psi\cdot p$ must indeed be invariant. Then,
$$\mu=\int_{\Sig_L}\mu_{\ul{R}}dp=\int_{\Sig_L}p_{\ul{R}}d(\psi\cdot p)=\int_{\Sig}p_{\tau(\ul{R})}d\widehat{\psi\cdot p}.$$
Notice that $\forall \ul{R}\in \Sig_L$, $p_{\ul{R}}\Big(\Big\{x\in M: \ul{R}\in\tau\Big[\widehat{\pi}^{-1}\left[\left\{x\right\}\right]\Big]\Big\}^c\Big)=0$. Then, for every Borel measurable $E$,
$$\mu(E)=\int_{\Sig}p_{\tau(\ul{R})}(E)d\widehat{\psi\cdot p}=\int_{\widehat{\pi}^{-1}[E]}p_{\tau(\ul{R})}(E)d\widehat{\psi\cdot p}\leq\int_{\widehat{\pi}^{-1}[E]}1d\widehat{\psi\cdot p}=\widehat{\psi\cdot p}\circ\widehat{\pi}^{-1}(E).$$
Now, since $p$ is conservative and ergodic (and $\sigma$-finite), so must $\psi\cdot p$ be, and so also $\widehat{\psi\cdot p}$, and in turn also $\widehat{\psi\cdot p}\circ\widehat{\pi}^{-1}$. In claim 3 of Theorem \ref{theorem1} we have seen that $\mu$ is conservative (and $\sigma$-finite). So $\mu$ is an invariant, conservative, $\sigma$-finite measure, dominated by an ergodic, invariant, conservative, $\sigma$-finite measure $\widehat{\psi\cdot p}\circ\widehat{\pi}^{-1}$, whence $\mu\propto\widehat{\psi\cdot p}\circ\widehat{\pi}^{-1}$ (with a proportion constant less or equal to $1$).
\end{proof}

\begin{cor}[Entropy formula]\label{kukilida}
Under the assumptions of Theorem \ref{theorem1}, if the measure $\mu$ in Theorem \ref{theorem1} is finite then it satisfies the entropy formula $h_\mu(f)=\int \log\Jac(d_xf|_{H^u(x)}) d\mu(x)$.
\end{cor}
\begin{proof}
	Under the notation of Theorem \ref{theorem1}, let $\varphi:\Sig\rightarrow \mathbb{R}$ be the geometric potential (recall Definition \ref{TGP}). As in Theorem \ref{ruellemargulis}, let $\varphi^-:\Sig_L\rightarrow\mathbb{R}$, $A:\Sig\rightarrow \mathbb{R}$, s.t. $A$ is bounded and $\varphi=\varphi^-+A\circ \sigma^{-1}-A$ (when $\varphi^-$ is associated with its natural extension to $\Sig$). By part 3 in the proof of Theorem \ref{ruellemargulis} and by Theorem \ref{PG0}, $P_G(\varphi^-)=0$. $\varphi$ is bounded by $\log M_f$, and since $A$ is also bounded, $\varphi^-$ is bounded as well. By the variational principle, $P_G(\varphi^-)=\sup\{h_{\nu}(\sigma_R)+\int \varphi^-d\nu:\nu\text{ is an inv. prob. on }\Sig_L\}$ (see \cite[Theorem~4.4]{SarigTDF}).
	
Let $p'$ be the $\varphi^-$-conformal measure on $\Sig_L$, and let $p$ be the $\phi$-conformal measure on $\Sig_L$. By part 2 in Theorem \ref{ruellemargulis}, recurrence of $\phi$ implies recurrence of $\varphi^-$, and so $p'$ and $p$ are conservative and ergodic. Let $\psi'$ and $\psi$ be the unique (up to normalization) continuous and positive $\varphi^-$-harmonic function and $\phi$-harmonic function on $\Sig_L$. Once can check that $\psi\cdot p$, $\psi'\cdot p'$ are invariant measures.

\ul{Part 1:} $\psi\cdot p=\psi'\cdot p'$.

\ul{Proof:} Fix $[R]\subseteq \Sig_L$. As in Claim \ref{VolHol}, there exists $C_R>1$ s.t. $|\log\psi|\leq C_R$ on $[R]$ and if $n\geq n_0$, then $\sup\limits_{\overset{d(\ul{R},\ul{S})\leq e^{-n},}{S_{-n+1}=R_{-n+1}=R}}|\phi_n(\ul{R})-\phi_n(\ul{S})|\leq C_R$. Let $(R,\ul{w},R)$ be a word of length $n\geq n_0$, and let $\zeta_{(R,\ul{w},R)}\in \Sig_L$ be the periodic extension of $(R,\ul{w},R)$. Then,

\begin{align}\label{screamshishi}
(\psi\cdot p)([R,\ul{w},R])=&\int \mathbb{1}_{[R,\ul{w},R]}\cdot \psi dp
= e^{\pm C_R}\int L_{\phi}^{n-1}\mathbb{1}_{[R,\ul{w},R]}dp= 
\\
= &e^{\pm C_R}\int \sum_{\sigma_R^{n-1}\ul{S}=\ul{R}}e^{\phi_{n-1}(\ul{S})}\mathbb{1}_{[R,\ul{w},R]}(\ul{S})dp(\ul{R})=e^{\pm 2C_R}e^{\phi_{n-1}(\zeta_{(R,\ul{w},R)})}p([R])
.\nonumber
\end{align}
$p([R])$ is finite by \cite[Theorem~1]{SarigNR}, and positive since $[w_{|\ul{w}|-1}]\subseteq \sigma_R[R] $. W.l.o.g. $C_R\geq |\log p(\sigma_R[R])|$, then $(\psi\cdot p)([R,\ul{w},R])=e^{\pm3 C_R} e^{\phi_n(\zeta_{(R,\ul{w},R)})}$. By \cite[Proposition~6.1]{SBO} and the $C^{1+\beta}$ regularity of $f$, $\varphi$ is H\"older continuous, then as in equation \eqref{screamshishi}, $\exists C_R'>1$ independent of $\ul{w}$ s.t. $(\psi'\cdot p')([R,\ul{w},R])=e^{\pm3C_R'} e^{\varphi^-_{n-1}(\zeta_{(R,\ul{w},R)})}$.

Since $\psi\cdot p$ and $\psi'\cdot p'$ are both conservative, they are carried by the $\sigma$-algebra generated by cylinders of the form $\{[R,R_{-n+2},...,R_{-1},R]\}_{[R]\subseteq \Sig_L, n\geq n_0}$. Then $\psi\cdot p\sim \psi'\cdot p'$. By Theorem \ref{theorem1}, $\phi$ must be positive recurrent. Then, by part 2 of Theorem \ref{ruellemargulis}, $\varphi^-$ must be positive recurrent as well; hence $\psi\cdot p,\psi'\cdot p'$ are finite measures. Since $\psi\cdot p$ and $\psi'\cdot p'$ are equivalent ergodic invariant finite measures, $\psi\cdot p\propto \psi'\cdot p'$. W.l.o.g. the scaling of $\psi,\psi'$ is chosen so $\psi\cdot p= \psi'\cdot p'$ is a probability measure.

\ul{Part 2:} Let $m$ be the unique invariant extension of $\psi'\cdot p'$ (equivalently, $\psi\cdot p$) to $\Sig$; in addition, the extension satisfies $h_{\psi'\cdot p'}(\sigma_R)=h_m(\sigma)$. By \cite[Proposition~8.1]{CyrSarig}, $\psi'\cdot p'$ is the equilibrium measure of $\varphi^-$. Thus, $0=P_{G}(\varphi^-)=h_{\psi'\cdot p'}(\sigma_R)+\int \varphi^-d(\psi'\cdot p')= h_{m}(\sigma)+\int \varphi^-dm= h_{m}(\sigma)+\int \varphi dm $.

\ul{Part 3:} $m$ is carried by $\Sig^\#$, and $\widehat{\pi}|_{\Sig^\#}$ is finite-to-one. Then $h_{m\circ \widehat{\pi}^{-1}}(f)=h_{m}(\sigma)$. Thus, by Proposition \ref{GeneralizedEntropyFormula} and the definition of $\varphi$ (Definition \ref{TGP}), $h_{\mu}(f)=\int \log\Jac(d_xf|_{H^u(x)})d\mu(x)$.
\end{proof}

\noindent\textbf{Remark:} Corollary \ref{kukilida} can be extended to all hyperbolic SRB measures without using the Ledrappier-Young entropy formula. Details will appear elsewhere.

\subsection{Positive Recurrence and Finiteness}\label{posrecnew}
In Theorem \ref{theorem1} we saw that the leaf condition with $\HWT_\chi$ implies the existence of an ergodic, conservative, $f$-invariant measure with absolutely continuous conditional measures on unstable leaves, which is finite if and only if the geometric potential (equivalently, $\phi$) is positive recurrent. We now provide a different characterization of the positive recurrence of $\phi$ in terms of a stronger leaf condition.

Recall Theorem \ref{epsilonika}, and the constant $\epsilon_\chi>0$ s.t. $\HWT_\chi=\HWT_\chi^{\epsilon_\chi}$.
\begin{definition}\label{PesinLevelSets}
	A {\em Pesin level set} of $\chi\text{-}\mathrm{summ}$, is a set $\Lambda\subseteq \chi\text{-}\mathrm{summ}$, s.t. $\exists N\in\mathbb{N}$ s.t.
	$\forall x\in \Lambda$, there exists a function $q:\{f^n(x)\}_{n\in\mathbb{Z}}\rightarrow (0,\epsilon_\chi]\cap \{e^{\frac{-\ell\cdot\epsilon_\chi}{3}}\}_{\ell\geq0}$ s.t. $q(x)\geq \frac{1}{N}$, $\frac{q\circ f}{q}=e^{\pm\epsilon_\chi}$, and $\forall n\in \mathbb{Z}$, $q\circ f^n(x)\leq Q_{\epsilon_\chi}\circ f^n(x)$. $\Lambda$ is said to be {\em of level $N$}.
\end{definition}
This definition is due to Pesin, see \cite[Definition~2.2.6]{BP}.

\begin{claim}\label{Nov19}
For every Pesin level set $\Lambda$, $\Lambda\cap\HWT_\chi$ is contained in a finite number of partition elements of $\mathcal{R}$. Conversely, every partition element of $\mathcal{R}$ is contained in a Pesin level set.	
\end{claim}
\begin{proof}
We begin with the converse statement. Given $R\in\mathcal{R}$, let $v\in\mathcal{V}$, $v=\psi_x^{p^s,p^u}$, s.t. $R\subseteq Z(v)$. By definition, every point $y\in R$ admits a coding $\ul{v}\in[v]\cap \Sigma^\#$, and thus belongs to the Pesin level set of level $\lceil \frac{1}{p^s\wedge p^u}\rceil$ (since one can choose the kernel $q(f^i(y))=p_i^s\wedge p_i^u$, where $\ul{v}=(\psi_{x_i}^{p^s_i,p^u_i})_{i\in\mathbb{Z}}$).

Now for the first statement. Let $\Lambda$ be a Pesin level set of level $N\in\mathbb{N}$.

\ul{Step 1:} $\forall x\in \Lambda\cap\HWT_\chi$, $x$ can be coded by a chain $\ul{v}\in \Sigma^\#$, where $v_i=\psi_{x_i}^{p^s_i,p^u_i}$, $i\in\mathbb{Z}$, and $p^s_i\wedge p^u_i\geq q(f^i(x))$.

Proof: $\exists q_1,q_2:\{f^n(x)\}_{n\in\mathbb{Z}}\rightarrow(0,\epsilon_\chi]\cap\{e^{\frac{-\epsilon_\chi\cdot\ell}{3}}\}_{\ell\geq0}$ where $q_1$ is given by the fact that $x$ belongs to the Pesin level set of level $N$, $\Lambda$, and $q_2$ is given by the recurrent $\epsilon_\chi$-weak temperability of $x$. Define $q:=\max\{q_1,q_2\}$. It follows trivially that $q:\{f^n(x)\}_{n\in\mathbb{Z}}\rightarrow(0,\epsilon_\chi]\cap\{e^{\frac{-\epsilon_\chi\cdot\ell}{3}}\}_{\ell\geq0}$, $q(f^n(x))\leq Q_{\epsilon_\chi}(f^n(x))$ $\forall n\in\mathbb{Z}$, and $\limsup\limits_{n\rightarrow\pm\infty}q\circ f^n(x)>0$. In addition, $\frac{q\circ f}{q}=\frac{\max\{q_1\circ f, q_2\circ f\}}{\max\{q_1, q_2\}}=\max\{\frac{q_1\circ f}{\max\{q_1, q_2\}},\frac{q_2\circ f}{\max\{q_1, q_2\}}\}\leq \max\{\frac{q_1\circ f}{q_1},\frac{q_2\circ f}{q_2}\}\leq \max\{e^{\epsilon_\chi}, e^{\epsilon_\chi}\}= e^{\epsilon_\chi}$. Similarly, $\frac{q\circ f}{q}\geq e^{-\epsilon_\chi}$. By \cite[Proposition~4.5,Lemma~4.6]{Sarig}, $\forall x\in \Lambda\cap\HWT_\chi$, $x$ can be coded by a chain $\ul{v}\in \Sigma^\#$, where $v_i=\psi_{x_i}^{p^s_i,p^u_i}$, $i\in\mathbb{Z}$, and $p^s_i\wedge p^u_i\geq q(f^i(x))$.

\ul{Step 2:} $\forall \ul{u},\ul{v}\in\Sigma^\#$ s.t. $\pi(\ul{u})=\pi(\ul{v})$, $\frac{p^s_0\wedge p^u_0}{q^s_0\wedge q^u_0}=e^{\pm\epsilon_\chi^\frac{1}{3}}$.

Proof: This is the content of \cite[Proposition~8.4]{Sarig} when $d=2$ (see \cite[Lemma~4.12]{SBO} for $d>2$).

\ul{Step 3:} Consider the set of double Pesin-charts $A_N:=\{\psi_x^{p^s,p^u}\in \mathcal{V}: p^s\cap p^u\geq\frac{1}{2N}\}$. By the discreteness of $\mathcal{V}$ (see \cite[Propostion~3.5]{Sarig}), $A_N$ is finite. Thus, $\widetilde{A}_N:=\{R\in \mathcal{R}:R\subseteq Z_v, v\in A_N\}$ is finite as well. Thus, by step 1 and step 2, $\Lambda\cap \HWT_\chi\subseteq \bigcupdot \widetilde{A}_N$.
\end{proof}	

\begin{definition}\label{posrecpts}
	The points in $\HWT_\chi^{\mathrm{PR}}:=\{x\in\HWT_
	\chi:\exists r_x>1\text{ s.t. }\limsup\limits_{n\rightarrow\infty}\frac{1}{n}\sum_{k=0}^{n-1}\mathbb{1}_{\Lambda_{r_x}}\circ f^k(x)>0\}$, where $\Lambda_{r_x}$ is a Pesin level set of level $r_x$, are called the {\em positively recurrent points} in $\HWT_\chi$.
\end{definition}

\begin{theorem}[The Ratio Ergodic Theorem]\label{Hopfs} Let $(X,\mathcal{B},\mu,T)$ be a $\sigma$-finite measure preserving transformation, and assume that $\mu$ is conservative. Then for $\mu$-a.e. $x\in X$, $\forall g,h\in L^1(\mu)$ s.t. $h\geq0$ and $\int hd\mu>0$,
$$\frac{\sum_{k=0}^ng\circ T^k(x)}{\sum_{k=0}^nh\circ T^k(x)}\xrightarrow[n\rightarrow\pm\infty]{}\frac{\int gd\mu}{\int hd\mu}.$$
\end{theorem}
This theorem is due to E. Hopf, for a modern proof see \cite{Zweimuller}.

\noindent Recall Definition \ref{leafcondo}.

\begin{theorem}\label{posrecclaim}
	There exists a $\chi$-hyperbolic SRB measure if and only if the leaf condition is satisfied for $\HWT_\chi^{\mathrm{PR}}$.
\end{theorem}
\begin{proof}
If there exists a $\chi$-hyperbolic SRB measure, then the leaf condition is satisfied	trivially for $\HWT_\chi^{\mathrm{PR}}$. Next, assume that the leaf condition is satisfied by $\HWT_\chi^{\mathrm{PR}}
$. 
Let $x\in \HWT_\chi^{\mathrm{PR}}$, then by Claim \ref{Nov19}, there must be some symbol $R_x$ s.t. $\limsup \frac{1}{n}\sum_{k=0}^{n-1}\mathbb{1}_{R_x}\circ f^k(x)>0$. Therefore, there exists $R\in\mathcal{R}$ s.t. the leaf condition is satisfied for $\{x\in R: \limsup \frac{1}{n}\sum_{k=0}^{n-1}\mathbb{1}_{R}\circ f^k(x)>0\}$. Hence, $\exists \ul{R}'\in\Sig_L^\#\cap[R]$ s.t. $\mu_{\ul{R}'}(\{x\in R: \limsup \frac{1}{n}\sum_{k=0}^{n-1}\mathbb{1}_{R}\circ f^k(x)>0\})>0$. Let $\ul{R}\in \langle R\rangle^{-\mathbb{N}}\cap \Sig^\#_L\cap [R]$, and let $\Gamma_{\ul{R}}: V^u(\ul{R}')\cap A_R\rightarrow V^u(\ul{R})\cap A_R $ be the holonomy map along the stable leaves in $A_R$. Then, $\forall x\in V^u(\ul{R}')\cap A_R\cap \{x\in R: \limsup \frac{1}{n}\sum_{k=0}^{n-1}\mathbb{1}_{R}\circ f^k(x)>0\}$, $\Gamma_{\ul{R}}(x)$ has a coding in $\Sig^\#$ with the same future as $\ul{R}(x)$. 
Thus, by Claim \ref{Nov19}, $\Gamma_{\ul{R}}(x)\in \HWT_\chi^{\mathrm{PR}}$. By Pesin's absolute continuity theorem \cite[Theorem~8.6.1]{BP}, $\Gamma_{\ul{R}}$ maps a positive leaf volume set to a positive leaf volume set. So
\begin{equation}\label{GammaPR2}\forall \ul{R}\in [R]\cap\Sig_L^\#\cap \langle R\rangle^{-\mathbb{N}}\text{, }\mu_{\ul{R}}(\HWT_\chi^{\mathrm{PR}})>0.\end{equation}
 We can now carry out the construction in Theorem \ref{theorem1}, and obtain the conservative, ergodic, and $f$-invariant measure $\mu=\int_{\Sig_L^\#}\mu_{\ul{R}}dp$, where $\forall \ul{R}\in \Sig_L^\#$, $\mu_{\ul{R}}$ is carried by $V^u(\ul{R})\cap A_{R_0}$; and $p$ gives a positive measure to every cylinder by \cite[Claim~3.5, pg.~76]{SarigTDF}. By equation \eqref{GammaPR2},  $\mu(\HWT_\chi^{\mathrm{PR}})>0$. Therefore, $\exists$ Pesin level set $\Lambda$, s.t. $\mu(\limsup \frac{1}{n}\sum_{k=0}^{n-1} \mathbb{1}_{\Lambda}\circ f^k)>0$, while $\mu(\Lambda)<\infty$ by claim 3 in Theorem \ref{theorem1}.

However, by the ratio ergodic theorem (see Theorem \ref{Hopfs}), $\forall M\in\mathbb{N}$, let $\Lambda_M$ be a Pesin level set of level $M$. Then,
$\limsup\frac{1}{n}\sum_{k=0}^{n-1} \mathbb{1}_{\Lambda}\circ f^k(x)\leq\limsup\frac{\sum_{k=0}^{n-1} \mathbb{1}_{\Lambda}\circ f^k(x)}{\sum_{k=0}^{n-1} \mathbb{1}_{\Lambda_M}\circ f^k(x)}=\frac{\mu(\mathbb{1}_{\Lambda})}{\mu(\mathbb{1}_{\Lambda_M})}$, for $\mu$-a.e. $x$. Assume for contradiction that $\mu$ is infinite. Then, $\frac{\mu(\mathbb{1}_{\Lambda})}{\mu(\mathbb{1}_{\Lambda_M})} \xrightarrow[M\rightarrow\infty]{}0$, and so $\limsup\frac{1}{n}\sum_{k=0}^{n-1} \mathbb{1}_{\Lambda}\circ f^k(x)=0$ for $\mu$-a.e. $x$. A contradiction.
Therefore $\mu$ is finite.
\end{proof}

\begin{appendix}
\section{Additional Properties of $\phi$}\label{append}

Recall the definition of the potential $\phi:\Sig_L\rightarrow \mathbb{R}^-$ from Theorem \ref{naturalmeasures}: $\phi(\ul{R})=\log(m_{V^u(\sigma_R\ul{R})}(f^{-1}[V^u(\ul{R})]))$. The questions of the boundedness of $\phi$ and its cohomology relation with the geometric potential are not necessary for the proof of our main results, but are nonetheless interesting. In this appendix we discuss these questions.

\subsection{Boundness of $\phi$}\label{append1}

Since the range of $\phi$ is $\mathbb{R}^-$, it is clear that $\phi$ is bounded from above. In this section we show that $\phi$ is also bounded from below.
%
%


\begin{lemma}\label{guigant}
$\exists \wh{\gamma}\geq 1$, a  constant depending only on $\chi,M$ and $f$, s.t. $\Vol\left(V^u(\sigma_R\ul{u})\right)\leq \wh{\gamma} \cdot \Vol\left( V^u(\ul{u})\right)$, $\forall\ul{u}\in \Sigma_L$.
\end{lemma}
\begin{proof}
Write $u_{-1}=\psi_{x}^{p^u,p^s}$, $u_0=\psi_y^{q^u,q^s}$. Let $F$ be the representing function of $V^u(\ul{u})$ in $u_{0}$, and $G$ the representing function of $V^u(\sigma_R\ul{u})$ in $u_{-1}$.

\ul{Step 1:} Comparing $p^u$ and $q^u$: Since $(u_{-1},u_0)\in \mathcal{E}$ (i.e. $u_{-1}\rightarrow u_0$, see \cite[Definition~2.23]{SBO}), $q^u=\min\{e^\epsilon p^u,Q_\epsilon(y)\}$. Thus
\begin{align}\label{radioblanew}
\frac{p^u}{q^u}=\frac{p^u}{\min\{e^\epsilon p^u,Q_\epsilon(y)\}}\leq \max\{e^{-\epsilon}, \frac{Q_\epsilon(x)}{Q_\epsilon(y)}\} \text{ }(\because p^u\leq Q_\epsilon(x)).
\end{align}
The fact that $(u_{-1},u_0)\in \mathcal{E}$ also implies that $\psi_{f(x)}^{q^s\wedge q^u}$ $\epsilon$-overlaps $\psi_{y}^{q^s\wedge q^u}$ (see \cite[Definition~2.18]{SBO}); which in particular implies (by \cite[Lemma~2.20]{SBO}) that $\frac{\|C_\chi^{-1}(f(x))\|}{\|C_\chi^{-1}(y)\|}=e^{\pm \epsilon^3}$. Thus, by definition,
\begin{equation*}
Q_\epsilon(y)\leq\frac{\epsilon^\frac{90}{\beta}}{3^\frac{6}{\beta}} \|C_\chi^{-1}(y)\|^\frac{-48}{\beta}\leq \frac{\epsilon^\frac{90}{\beta}}{3^\frac{6}{\beta}} \|C_\chi^{-1}(f(x))\|^\frac{-48}{\beta}\cdot (e^{\epsilon^3})^\frac{48}{\beta}\leq e^{\frac{\epsilon}{3}}\cdot Q_\epsilon(f(x))\cdot (e^{\epsilon^3})^\frac{48}{\beta}\leq Q_\epsilon(f(x))\cdot e^\frac{\epsilon}{2},
\end{equation*}

\noindent when $\epsilon>0$ is sufficiently small. Similarly, $Q_\epsilon(y)\geq e^{-\frac{\epsilon}{2}}Q_\epsilon(f(x))$; and so $\frac{Q_\epsilon(f(x))}{Q_\epsilon(y)}=e^{\pm\frac{\epsilon}{2}}$. In addition, by \cite[Lemma~2.15]{SBO}, $\exists \omega_0\geq1$ depending only on $\chi, M$ and $f$, s.t. $\frac{Q_\epsilon(f(x))}{Q_\epsilon(x)}=\omega_0^{\pm1}$. Plugging these inequalities in equation \eqref{radioblanew}, we obtain
\begin{equation}\label{radioblanew2}
	\frac{p^u}{q^u}\leq \max\{e^{-\epsilon}, \frac{Q_\epsilon(x)}{Q_\epsilon(f(x))}\cdot \frac{Q_\epsilon(f(x))} {Q_\epsilon(y)}\}
	\leq \max\{e^{-\epsilon}, \omega_0\cdot e^\frac{\epsilon}{2}\}\leq e^\epsilon\cdot \omega_0.
\end{equation}

\ul{Step 2:} Write $\widetilde{F}(t):=(F(t),t)$, $\widetilde{G}(t):=(G(t),t)$. Following from Definition \ref{def135}, $\Jac(d_\cdot \widetilde{F}),\Jac(d_\cdot \widetilde{G})=e^{\pm\epsilon}$, for $\epsilon>0$ sufficiently small.

\ul{Step 3:} Comparing $\Jac(C_\chi(x))=|\det C_\chi(x) |$ and $\Jac(C_\chi(y)) =|\det C_\chi(y) | $: By \cite[Theorem~2.4]{SBO}, $\forall \xi\in T_xM$, $$|C_\chi^{-1}(x)\xi|=\sqrt{S^2(x,\xi_s)+U^2(x,\xi_u)},$$
where $\xi=\xi_s+\xi_u$, $\xi_s\in H^s(x)$, $\xi_u\in H^u(x)$, and
\begin{align*}
	S^2(x,\xi_s)=2\sum_{m=0}^\infty|d_xf^m\xi_s|^2 e^{2\chi m}<\infty,\text{ }
		U^2(x,\xi_u)=2\sum_{m=0}^\infty|d_xf^{-m}\xi_u|^2 e^{2\chi m}<\infty.
\end{align*}
The same formula holds with $f(x)$ replacing $x$. One can then check that
\begin{align*}
\frac{U^2(f(x),d_xf\xi_u)}{U^2(x,\xi_u)}=e^{\pm \log(e^{2\chi}+M_f^2)}, \frac{S^2(f(x),d_xf\xi_s)}{S^2(x,\xi_s)}= e^{\pm \log(e^{2\chi}+M_f^2)}.
\end{align*}

It follows then, that $\frac{|C_\chi^{-1}(f(x))d_xf\xi|}{|C_\chi^{-1}(x)\xi|}= e^{\pm \log \sqrt{e^{2\chi}+M^2_f}}$. It follows that $\frac{\Jac(C_\chi(x))}{\Jac(C_\chi(f(x)))}=e^{\pm d\log \left(e^{2\chi}+M^2_f\right)}$. In addition, by \cite[Lemma~2.20]{SBO}, $\frac{\Jac(C_\chi(f(x)))}{\Jac(C_\chi(y))}=e^{\pm d\cdot \epsilon^2}=e^{\pm\epsilon}$ (for $\epsilon>0$ sufficiently small). Thus, in total,
$$\frac{\Jac(C_\chi(x))}{\Jac(C_\chi(y))}= \frac{\Jac(C_\chi(x))}{\Jac(C_\chi(f(x)))}\cdot \frac{\Jac(C_\chi(f(x)))}{\Jac(C_\chi(y))}= e^{\pm\left(\epsilon+d\log \left(e^{2\chi}+M^2_f\right)\right)}.$$

\ul{Step 4:} It follows by the definition of $\widetilde{G},\widetilde{F}$, that $V^u(\ul{u})=\psi_y\circ\widetilde{F}[R_{q^u}(0)]$ and $V^u(\sigma_R\ul{u})=\psi_x\circ\widetilde{G}[R_{p^u}(0)]$,
where $R_{q^u}(0)=\{u\in\mathbb{R}^{\mathrm{dim}H^u}:|u|_\infty\leq q^u\}$, $R_{p^u}(0)=\{u\in\mathbb{R}^ {\mathrm{dim}H^u}:|u|_\infty\leq p^u\}$. Recall, $\psi_x=\exp_x\circ C_\chi(x), \psi_y=\exp_y\circ C_\chi(y)$. Thus
\begin{align*}
\Vol(V^u(\sigma_R\ul{u}))=&\int_{R_{p^u}(0)} \Jac(d_{ C_\chi(x)\circ \widetilde{G}(t)}\exp_x)\cdot \Jac(C_\chi(x))\cdot \Jac(d_t\widetilde{G})d\mathrm{Leb}(t)\\	\leq & 2^d\cdot e^{\epsilon} \cdot(2p^u)^{\mathrm{dim}H^u(x)}\Jac(C_\chi(x))\text{ }(\because \|d_\cdot\exp_x\|\leq 2\text{ by Definition \ref{localisometries}})\\
\leq& 4^d  e^{2\epsilon} e^{d(\epsilon+\log\omega_0)}e^{\epsilon+d\log \left(e^{2\chi}+M^2_f\right)} \int_{R_{q^u}(0)} \Jac(d_{ C_\chi(y)\circ \widetilde{F}(t)}\exp_y)\cdot \Jac(C_\chi(y))\cdot \Jac(d_t\widetilde{F})d\mathrm{Leb}(t)\\
=& 4^d \cdot e^{2\epsilon}\cdot e^{d(\epsilon+\log\omega_0)}\cdot e^{\epsilon+d\log \left(e^{2\chi}+M^2_f\right)}\cdot \Vol(V^u(\ul{u})),
\end{align*}
The lemma follows with $\wh{\gamma}:= 4^d \cdot e^{2}\cdot e^{d(1+\log\omega_0)}\cdot e^{1+d\log \left(e^{2\chi}+M^2_f\right)}$.
\end{proof}

\begin{lemma}\label{guigant2}
Let $\ul{u},\ul{v}\in\Sigma^\#$ s.t. $\pi(\ul{u})= \pi(\ul{v})=p $. Then $\exists A\subseteq V^u(\ul{u})\cap V^u(\ul{v})$ an open set in $V^u(\ul{u})$ which depends only on $\ul{u}$ s.t. $\Vol(A) \geq \Vol(V^u(\ul{u}))\cdot \wh{\alpha} $, where $\wh{\alpha}\in (0,1)$ depends only on $M$.
\end{lemma}
\begin{proof}
Write $u_0=\psi_x^{p^s,p^u}, v_0=\psi_y^{q^s,q^u}$. 
By \cite[Lemma~4.12]{SBO}, $\frac{p^u}{q^u}=e^{\pm\epsilon^\frac{1}{3}}$. Let $B:=R_{\frac{p^u}{100\sqrt{d}}}(0)$ be a 
ball in $|\cdot|_\infty$-norm in $\mathbb{R}^d$ centered at $0$ with a radius of $\frac{p^u}{100\sqrt{d}} $. By \cite[Theorem~4.13]{SBO}, $\psi_y^{-1}\circ \psi_x[B]\subseteq R_{\frac{p^u}{100}+\frac{q^u}{10}+\frac{p^u}{100}\frac{1}{2}\epsilon^\frac{1}{3}}(0)\subseteq R_{\frac{q^u}{5}}(0)$. Thus $\psi_x[B]\subseteq \psi_y[R_{\frac{q^u}{5}}(0)]
\subseteq \psi_y[R_{q^u}(0)]$, and so $\psi_x[B]\cap V^u(\ul{u})\subseteq V^u(\ul{v})$ (since $V^u(\ul{u})$ spans over the window of $u_0$, and $V^u(\ul{v})$ spans over the window of $v_0$, and since both are local unstable leaves of $p$, they must coincide where the windows intersect).

Denote by $F$ the representing function of $V^u(\ul{u})$ in $u_0$. Write $\widetilde{F}(t):=(F(t),t)$. We now wish to compare the volume of $A:=\psi_x\circ \widetilde{F}[R_{\frac{p^u}{\sqrt{d}100}}(0)]\subseteq \psi_x[B]$. We follow the estimates in steps 2 and 4 of Lemma \ref{guigant}:
\begin{align*}
	\Vol(A)=&\int_{R_{\frac{p^u}{\sqrt{d}100}}(0)}\Jac(d_{\widetilde{F}(t)}\psi_x)\Jac(d_t\widetilde{F})d\mathrm{Leb}(t)\\
	=& \int_{R_{\frac{p^u}{\sqrt{d}100}}(0)}\Jac(d_{C_\chi(x)\widetilde{F}(t)}\exp_x)\Jac(C_\chi(x))\Jac(d_t\widetilde{F})d\mathrm{Leb}(t)\\
	\geq& 2^{-d}\cdot e^{-\epsilon}\cdot (2\frac{p^u}{100\sqrt{d}})^{\mathrm{dim}H^u(x)}\Jac(C_\chi(x))\\
	\geq &2^{-2d}\cdot e^{-2\epsilon}\cdot \frac{1}{(100\sqrt{d})^{d}} \int_{R_{p^u}(0)}\Jac(d_{\widetilde{F}(t)}\psi_x)\Jac(d_t\widetilde{F})d\mathrm{Leb}(t)= 2^{-2d}\cdot e^{-2\epsilon}\cdot \frac{1}{(100\sqrt{d})^{d}}\cdot \Vol(V^u(\ul{u})).
\end{align*}
Define $\wh{\alpha}:= 2^{-2d}\cdot e^{-2}\cdot \frac{1}{(100\sqrt{d})^{d}} \in (0,1)$.
 \end{proof}
\begin{cor}
	$\exists \wh{\omega}\geq1$ which depends only on $\chi,M$ and $f$, s.t. $\forall\ul{R}\in \Sig_L$, $\Vol(V^u(\sigma_R\ul{R}))\leq\wh{\omega}\Vol(V^u(\ul{R}))$.
\end{cor}
\begin{proof}
Fix some extension $\ul{R}^{\pm}\in\Sig^\#$ s.t. $R_i=R^{\pm}_i$ $\forall i\leq0$. By Definition \ref{unstablemanifold}, $\Vol(V^u(\sigma_R\ul{R}))=\Vol(\bigcap\limits_{\ul{u}\CAR\sigma_R\ul{R}}V^u(\ul{u}))$. Let $\ul{u}\in\Sigma_L^\#$ be a chain which covers $\ul{R}$, then $\sigma_R\ul{u}\CAR\sigma_R\ul{R}$. Fix a chain $\ul{v}\in\Sigma_L^\#$ s.t. $\ul{v}\CAR\ul{R}$. Let $A=A(\ul{v})$ be the set given by Lemma \ref{guigant2}. It follows that $A\subseteq \bigcap\limits_{\ul{u}\CAR\ul{R}}V^u(\ul{u})$. Then,
\begin{align*}
\Vol(V^u(\sigma_R\ul{R}))=&\Vol\left(\bigcap\limits_{\ul{u}\CAR\sigma_R\ul{R}}V^u(\ul{u}) \right)= \Vol \left(\bigcap\limits_{\ul{v}'\CAR\ul{R}}V^u(\sigma_R\ul{v}') \right) \leq \Vol \left(V^u(\sigma_R\ul{v})\right)\\
\leq &\wh{\gamma} \Vol \left(V^u(\ul{v}) \right)\leq \frac{\wh{\gamma}}{\wh{\alpha}}\Vol \left(A \right)\leq \frac{\wh{\gamma}}{\wh{\alpha}} \Vol \left(\bigcap\limits_{\ul{u}\CAR\ul{R}}V^u(\ul{u}) \right)= \frac{\wh{\gamma}}{\wh{\alpha}} \Vol \left(V^u(\ul{R})) \right)\text{ (}\because\text{Lemma \ref{guigant},Lemma \ref{guigant2})}.
\end{align*}
Define $\wh{\omega}:=\frac{\wh{\gamma}}{\wh{\alpha}}$.
\end{proof}
\begin{cor}\label{guigant3}
	$\phi:\Sig_L\rightarrow (-\infty,0]$ is bounded.
\end{cor}
\begin{proof}
	Let $\ul{R}\in \Sig_L$. Recall, $m_{V^u(\sigma_R\ul{R})}=\frac{\rho_{\sigma_R\ul{R}}}{\Vol(V^u(\sigma_R\ul{R}))}\cdot \lambda_{V^u(\sigma_R\ul{R})}$, where $\rho_{\sigma_R\ul{R}}:V^u(\sigma_R\ul{R})\rightarrow [e^{-\epsilon},e^\epsilon]$ and $\lambda_{V^u(\sigma_R\ul{R})}$ is the (not normalized) Riemannian volume measure of $V^u(\sigma_R\ul{R})$. As in Claim \ref{phibar}, $\phi(\ul{R})=\log m_{V^u(\sigma_R\ul{R})}\left(f^{-1}[V^u(\ul{R})]\right)$. Then,
	
	\begin{align*}
	\phi(\ul{R})=&\log m_{V^u(\sigma_R\ul{R})}\left(f^{-1}[V^u(\ul{R})]\right)=\pm\epsilon+\log\lambda_{V^u(\sigma_R\ul{R})} \left(f^{-1}[V^u(\ul{R})]\right)+\log\frac{1}{\Vol\left(V^u(\sigma_R\ul{R})\right)}\\
	=&\pm\epsilon\pm \log M_f+\log\frac{\Vol(V^u(\ul{R}))}{\Vol(V^u(\sigma_R\ul{R}))}\geq -\epsilon-\log M_f-\log \wh{\omega}>-\infty.	
	\end{align*}

\end{proof}
\noindent\textbf{Remark:} Claim \ref{VolHol} together with the boundness of $\phi$ imply that $\phi:\Sig_L\to\mathbb{R}^-$ is in fact H\"older continuous.

\subsection{Cohomology of $\phi$ to the Geometric Potential}\label{append2}
Recall the geometric potential from Definition \ref{TGP}: $\varphi:\Sig\rightarrow [-d\cdot\log M_f,d\cdot\log M_f]$, $\varphi(\ul{R}):=\log\Jac(d_{\widehat{\pi}(\ul{R})}f^{-1}|_{T_{\widehat{\pi}(\ul{R})}V^u(\ul{R})})$. In Part 1 of Theorem \ref{ruellemargulis} we saw: $\forall n\geq1$,
\begin{align}
	|\phi_n(\ul{R})-\varphi_n(\ul{R}^\pm)|\leq2\epsilon,
\end{align}
where $\ul{R}$ is any $n$-periodic chain in $\Sig_L$, and $\ul{R}^\pm$ is the unique periodic extension of $\ul{R}$ to $\Sig$.

In this section, we show further that in fact $\phi$ and $\varphi$ are cohomologous:

\begin{claim}
There exists a locally H\"older continuous function $\mathcal{L}:\Sig\to\mathbb{R}$ s.t. $$\phi=\varphi + \mathcal{L}-\mathcal{L}\circ \sigma^{-1}$$	
when restricted to irreducible components of $\Sig$.
\end{claim}
$\mathcal{L}$ is called a H\"older continuous {\em transfer function}, and $\mathcal{L}-\mathcal{L} \circ\sigma^{-1}$ is called a {\em coboundary}.
\begin{proof}
Assume w.l.o.g. that $\Sig_L$ is irreducible. By Part 2 in Theorem \ref{ruellemargulis}, there exists a bounded and H\"older continuous function $A:\Sig\to\mathbb{R}$ s.t. $$\varphi^-:=\varphi+A-A\circ \sigma^{-1}$$
is a well-defined and H\"older continuous function $\varphi^-:\Sig_L\to\mathbb{R}$ (i.e. depends only on the negative coordinates of chains).

We continue to show that $\varphi^-$ and $\phi$ are cohomologous with a H\"older continuous transfer function, which in turn implies that $\phi$ and $\varphi$ are cohomologous with a H\"older continuous transfer function.

Let $n\geq 1$, and let $\ul{w}=(R,w_{-(n-2)},\ldots,w_{-1},R)$ be an admissible word of length $n$.	 Let $\ul{W}$ be its periodic extension to $\Sig_L$.

Then it follows that for any $m\in \mathbb{N}$,
\begin{align*}
	\phi_{n\cdot m}(\ul{W})-\varphi^-_{n\cdot m}(\ul{W})=m\cdot\left(\phi_n(\ul{W})-\varphi_n^-(\ul{W}) \right).
\end{align*}
At the same time, $|\phi_{n\cdot m}(\ul{W})-\varphi^-_{n\cdot m}(\ul{W}) |\leq 2(\epsilon+\|A\|_\infty)$ for all $n$ and $m$. Therefore it follows that $\forall n\geq1$, and for every $n$-periodic chain $\ul{R}\in\Sig_L$,
$$\phi_n(\ul{R})-\varphi^-_n(\ul{R})=0.$$
By Livsic's theorem (see \cite[Theorem~1.1]{SarigTDF} for the statement and the proof in the context of countable Markov shifts), there exists a locally H\"older continuous function\footnote{H\"older continuous on compact sets, while the H\"older exponent is uniform on the space. } $\mathcal{L}':\Sig_L\to\mathbb{R}$ s.t. $$\phi=\varphi^-+\mathcal{L}'-\mathcal{L}' \circ \sigma_R.$$

Thus, in total, $$\phi= \varphi+A-A\circ \sigma^{-1}+\mathcal{L}'-\mathcal{L}' \circ \sigma^{-1}= \varphi+(A+\mathcal{L}')-(A+\mathcal{L}') \circ \sigma^{-1} .$$
Set $\mathcal{L}=A+\mathcal{L}'$.
\end{proof}

\noindent \textbf{Remark:} Carrying out \eqref{forAppend} without the assumption $\ul{R}=\sigma_R^n\ul{R}$ yields the more careful estimate $e^{\phi_n-\varphi_n}=e^{\pm2\epsilon}\frac{\Vol(V^u(\ul{R}))}{\Vol(V^u(\sigma_R^n\ul{R}))}$, which in general may not be bounded in $n$, even on irreducible components. This implies that while both $\phi$ and $\varphi$ are H\"older continuous (and bounded), and $\mathcal{L}$ is locally H\"older continuous, and $\phi=\varphi+\mathcal{L}-\mathcal{L}\circ\sigma^{-1}$, in general $\mathcal{L}$ may still be unbounded.

\end{appendix}

\tocless{\section*{Acknowledgements}} This work constitutes a part of a doctoral thesis,  conducted in the Weizmann Institute of Science under the guidance of O. Sarig. I would like to thank Professor Pesin for introducing this question to me, for reading this manuscript with care, and for his useful remarks. In addition, I would like to thank the referees for their useful, detailed, and insightful remarks. I would also like to thank the Weizmann Institute of Science for excellent working conditions. This research was partially supported by ISF grant 1149/18 and BSF grant 2016105.

\bibliographystyle{alpha}
\tocless\bibliography{Elphi}
\end{document}